\documentclass[12pt,reqno,twoside]{amsart}
\usepackage[utf8]{inputenc}
\usepackage[T1]{fontenc}
\usepackage{a4wide}
\usepackage{amsmath}
\usepackage{amscd}
\usepackage{amssymb,latexsym}
\usepackage{adjustbox}
\usepackage{bm}
\usepackage{booktabs}
\usepackage{relsize}
\usepackage{extarrows}
\usepackage{graphicx}
\usepackage{mathrsfs}
\usepackage[stable]{footmisc}
\usepackage{xcolor}
\usepackage[colorlinks,linkcolor=teal,citecolor=teal]{hyperref}

\begin{document}
	\numberwithin{equation}{section}
	\newtheorem{thm}{Theorem}[section]
	\newtheorem{prop}[thm]{Proposition}
	\newtheorem{cor}[thm]{Corollary}
	\newtheorem{lemma}[thm]{Lemma}
	\newtheorem{definition}{Definition}
	\numberwithin{definition}{section} 
	\newtheorem{remark}{Remark}
	\newtheorem{question}{Question}
	\newtheorem{example}{Example}
	\newcommand{\MV}{\mathcal{V}}
	\newcommand{\rd}{{\rm d}}
	\newcommand{\dV}{{\rm d}_{\MV}}
	\newcommand{\dmv}{\rd_{\MV,\vartheta}}
	\newcommand{\dfv}{\rd_{\mathcal{F},[\vartheta]}}
	\newcommand{\dbf}{\bar{\partial}_{\mathcal{F},[\vartheta]}}
	\newcommand{\mfp}{\mathfrak{p}}
	\newcommand{\lmfp}{\scalebox{0.8}[1]{$\mathfrak{p}$}}

	\title[Levi flat structures via structure sheaves]{Levi flat structures via structure sheaves: differential complexes, convexity, and global solvability}
	\author{Qingchun Ji}
	\address{School of Mathematical Sciences, Fudan University, Shanghai 200433, China}
	\email{qingchunji@fudan.edu.cn}
	\author{Jun Yao}
	\address{School of Mathematical Sciences, University of Electronic Science and Technology of China, Chengdu 611731, China}
	\email{junyao@uestc.edu.cn}
	\thanks{This work was partially supported by National Natural Science Foundation of China, No. 12431005, No. 12571086.}
	\date{}

\begin{abstract}
	This paper investigates Levi flat structures from the perspective of structure sheaves. We employ formal integrability to construct a class of differential complexes, thereby providing a resolution of the structure sheaf and a global realization of the Treves complex. Drawing inspiration from Morse theory and Grauert's convexity, we introduce notions of convexity and positivity that fully exploit Levi flatness, which yield global exactness of the differential complex and Sobolev regularity in the compact case. As applications, we establish the global solvability of the Treves complex for Levi flat structures, together with results on singular cohomology and the extension problem for canonical forms in the elliptic case.
\end{abstract}
	\maketitle
\setcounter{tocdepth}{3}
\tableofcontents

\section{Introduction}
In \cite{Tf81}, F. Treves introduced a differential complex associated with a formally integrable structure, which naturally generalizes the de Rham, Dolbeault, and tangential Cauchy-Riemann complexes. The local solvability of the Treves complex has since attracted considerable attention, yet a complete characterization remains limited to a few cases: (i) $\mathcal{V}$ is elliptic (\cite{T2} and \cite{BCH08}), (ii) $\mathcal{V}$ has a nondegenerate Levi form (\cite{AH72} and \cite{Tf82}), (iii) $\mathcal{V}$ has corank one (\cite{CH}), or (iv) $\MV$ defines a \emph{complex Frobenius structure}, i.e., its Levi form vanishes (\cite{JYY22}). In the case of $\textrm{rank}_\mathbb{C}\MV=1$, the local solvability of the Treves complex is closely related to the local solvability of a single vector field. We review two criteria for local solvability. Let $X$ be a smooth complex vector field. A necessary condition for the local solvability of $X$, given by a special case of H\"ormander's theorem, is
\begin{align*}
	\mathcal{P}_{_X}^{-1}(0)\subseteq\mathcal{P}_{_{\left[X,\bar{X}\right]}}^{-1}(0),
\end{align*}
where $\mathcal{P}_\cdot$ denotes the principal symbol of a differential operator. Clearly, any smooth section of a complex Frobenius structure of rank $1$ satisfies Hörmander's condition. On the other hand, L. Nirenberg and F. Treves proposed a necessary and sufficient condition for the local solvability of $X$ (see \cite{NT63} and \cite{Hlv85}):
\begin{center}
	$\textrm{Im}\mathcal{P}_{_X}$ does not change sign along every null bicharacteristic curve of $\textrm{Re}\mathcal{P}_{_X}$.
\end{center}
In the $2$-dimensional case, write $X$ in the canonical form 
\begin{align*}
	X=\frac{\partial}{\partial x_1}+\sqrt{-1}a(x)\frac{\partial}{\partial x_2}, 
\end{align*}
where $a(x)\in C^\infty(\mathbb{R}^2)$ is a real-valued function. The Nirenberg-Treves condition amounts to
\begin{align*}
	x_1\mapsto a(x)\ \text{does not change sign}.
\end{align*}
Obviously, the smooth vector field $X=\frac{\partial}{\partial x_1}+\sqrt{-1}a(x)\frac{\partial}{\partial x_2}$ with
\begin{align*}
	a(x)=b(x_2)e^{c(x)}+d(x_2),\ b(x_2)d(x_2)\geq0,
\end{align*}
satisfies the Nirenberg-Treves condition, and, furthermore, generates a complex Frobenius structure  $\mathcal{V}$ if there exists a smooth real-valued
function $h$ such that $\partial_{x_1}a=ha$.

These observations suggest that local solvability may be inherent in complex Frobenius structures. 

A complex Frobenius structure $\MV$ is said to be a \emph{Levi flat structure} if $\MV+\overline{\MV}$ has constant rank (\cite{T2}). This paper is dedicated to the study of Levi flat structures from the perspective of structure sheaves. We formulate a geometric condition for the global solvability of the Morse-Novikov-Treves complex (see (\ref{new complex})) associated with a Levi flat structure, which is a natural generalization of the Treves complex. This condition is determined by the Levi flat structure itself, and local solvability follows immediately since the condition holds locally.

We begin by constructing a class of differential complexes, which provides a global realization of the Morse-Novikov-Treves complex. Let $\MV$ be a formally integrable structure with its dual bundle $\MV^*$ over an $(m+n)$-dimensional smooth manifold $M$, where
\begin{align*}
	n=\text{rank}_{\mathbb{C}}\MV,\ m=\text{corank}_{\mathbb{C}}\MV.
\end{align*}
For $p,q\geq 0$, set
\begin{align}\label{e12}
	\Lambda_\MV^{p,q}=\Lambda^{p} N^*\MV\wedge\Lambda^{q}\mathbb{C}T^*M\subseteq\Lambda^{p+q}\mathbb{C}T^*M,
\end{align}
then
\begin{align*}
	\Lambda^{\ell}\mathbb{C}T^*M=\Lambda_\MV^{0,\ell}\supseteq\Lambda_\MV^{1,\ell-1}\supseteq\cdots\supseteq\Lambda_\MV^{\ell,0}=\Lambda^\ell N^*\MV
\end{align*}
is a filtration of $\Lambda^\ell\mathbb{C}T^*M$, where $N^*\MV\subseteq\mathbb{C}T^*M$ is the conormal bundle of $\MV$. In particular, the following line bundle
\begin{align*}
	K_{\MV}:=\Lambda_\MV^{m,0}=\det N^*\MV
\end{align*}
is called the \emph{canonical bundle} of the formally integrable structure $\MV$.

In this paper, we assume that $\MV$ is a formally integrable structure whose canonical bundle $K_{\MV}$ is basic (Definition \ref{blb.}). We now state our results and describe the organization of the paper. Given a basic vector bundle $E$ of rank $r$ over $(M,\MV)$ and for $0\leq q\leq n$, define
\begin{align}\label{bdl.}
	\Lambda_\MV^{m,q}(E)=\Lambda_\MV^{m,q}\otimes E.
\end{align}
Let $\vartheta$ be a smooth $1$-form on $M$ satisfying 
\begin{align}\label{mnf}
	\rd\vartheta\equiv0\ {\rm mod}\ C^\infty(M,N^*\MV).
\end{align}
We will construct a differential complex over $(M,\MV)$ in Section \ref{constr.}:
\begin{align}\label{complex1}
	\mathcal{L}_\MV^{m,0}(E)\stackrel{\dmv}{\longrightarrow}\mathcal{L}_\MV^{m,1}(E)
	\stackrel{\dmv}{\longrightarrow}\mathcal{L}_\MV^{m,2}(E)\stackrel{\dmv }{\longrightarrow}\cdots\stackrel{\dmv }{\longrightarrow}\mathcal{L}_\MV^{m,n}(E)\longrightarrow0,
\end{align}
where $\mathcal{L}_\MV^{m,*}(E)$ is a sheaf defined by (\ref{subs}).




The well-known complex Frobenius theorem (\cite{Nl57}) asserts that every Levi flat structure is locally integrable. 
Since the canonical bundle of any locally integrable structure is basic (see (iii) in Example \ref{exa}), it follows that each Levi flat structure possesses a basic canonical bundle. Our global exactness theorem for the complex (\ref{complex1}) associated with a Levi flat structure is stated as follows.
\begin{thm}\label{t1}
	Suppose that $M$ admits a $q$-convex exhaustion function $\varphi\in C^\infty(M)$ with respect to a Levi flat structure $\MV$, $\vartheta$ is a smooth $1$-form on $M$ satisfying $(\ref{mnf})$, and $E$ is a basic vector bundle over $(M,\MV)$. Then for any $f\in L^2_{loc}(M,\Lambda_\MV^{m,q'}(E))$ with $ \dmv f=0 $, there exists some $ u\in L^2_{loc}(M,\Lambda_\MV^{m,q'-1}(E)) $ such that $ \dmv u=f $ for $q'\geq q$. 
\end{thm}

The notion of $q$-convexity for $\varphi$ (Definition \ref{qpqc}) in Theorem \ref{t1} is weaker than the Grauert-type convexity when $\MV$ is a complex structure, and encodes information about the Morse indices of $\varphi$ if $\MV$ is essentially real, i.e., $\MV=\overline{\MV}$. We will give some examples of smooth $q$-convex exhaustion functions with respect to elliptic structures and Levi flat CR structures (Example \ref{qc.}). 
It is noteworthy that the convexity condition requires the positivity of the quadratic form (Definition \ref{quadratic.}) only on a subset $\mathcal{K}_\varphi\subseteq M$. However, the Levi flatness of $\MV$ enables us to construct a vector field that compensates for the loss of positivity outside $\mathcal{K}_\varphi$ (see Section \ref{existe.} for details).

Theorem \ref{t1} yields a local existence result (Corollary \ref{local existence}), since any point in $M$ possesses an open neighborhood with a smooth $1$-convex exhaustion function with respect to a Levi flat structure. Thus, the complex (\ref{complex1}) provides a fine resolution of the structure sheaf $\mathcal{S}_\MV (E)$ (Corollary \ref{ros}).

On a compact manifold, any real-valued smooth function $\phi$ must attain its maximum, and the Hessian of $\phi$ is negative
semidefinite at maximum points, which implies that there is no $q$-convex function. For compact manifolds, we employ the positivity of a basic line bundle (Definition \ref{qpqc}) to establish a theorem on global existence and Sobolev regularity.
\begin{thm}\label{ct}
	Let $\MV$ be a Levi flat structure on a compact manifold $M$, $E$ a basic vector bundle over $(M,\MV)$, and $\vartheta$ a smooth $1$-form on $M$ fulfilling $(\ref{mnf})$. Assume that there exists a $q$-positive basic line bundle $(L,h_{_L})$ with respect to $\MV$. Then for every $s\in\mathbb{Z}_{\geq0}$, there exists a positive integer $\tau_s$ such that, for any $\tau\geq\tau_s$, $q'\geq q$ and $f\in W^s\big(M,\Lambda_\MV^{m,q'}(L^{\tau}\otimes E)\big)$ with $\dmv f=0$, there exists some $u\in W^{s}\big(M,\Lambda_\MV^{m,q'-1}(L^{\tau}\otimes E)\big)$ satisfying $ \dmv u=f $ and $\|u\|_s\leq C_{s,\tau}\|f\|_s$, where $C_{s,\tau}>0$ is independent of $f$, and $W^{s}\big(M,\Lambda_\MV^{m,q'}(L^{\tau}\otimes E)\big)$ is the completion of $C^\infty\big(M,\Lambda_\MV^{m,q'}(L^{\tau}\otimes E)\big)$ under the Sobolev norm of order $s$.
\end{thm}

\begin{remark}
	It should be noted that the finite-order regularity in Theorem \ref{ct} cannot be extended to infinite-order due to the example in \cite{Am14}. However, in the case of the tube structures, a very interesting characterization of the almost global hypoellipticity of the Treves complex is established in \cite{AFJR24}.
\end{remark}
Theorem \ref{t1} and Theorem \ref{ct}, together with Corollary \ref{ros} imply the vanishing of sheaf cohomologies by letting $\vartheta\equiv0$.
\begin{cor}\label{local existence1}
	Let $\MV$ be a Levi flat structure on a smooth manifold $M$, and $E$ a basic vector bundle over $(M,\MV)$. 
	\begin{enumerate}
		\item[$(i)$] If there exists a smooth $q$-convex exhaustion function with respect to $\MV$ on $M$, then $H^{q'}(M,\mathcal{S}_\MV (E))=0$ for $q'\geq q$.
		\item[$(ii)$] If $M$ is compact and $(L,h_{_L})$ is a $q$-positive basic line bundle with respect to $\MV$, then there exists an integer $\tau_0>0$ such that $H^{q'}(M,\mathcal{S}_\MV (L^{\tau}\otimes E))=0$ for $q'\geq q$ and $\tau\geq \tau_0$.
	\end{enumerate}
\end{cor}
\begin{remark}
	$(i)$ If $\MV$ defines an elliptic structure $($i.e., $N^*\MV\cap\overline{N^*\MV}=0)$, then $\mathcal{S}_\MV(E)=\mathcal{O}_\MV(E)$ $($Definition \ref{blb.}$)$ by the standard elliptic regularity analysis.
	
	$(ii)$ If $\MV$ is a complex structure, then Corollary \ref{local existence1} recovers some classical vanishing theorems in several complex variables $($see \cite{Djp12C}$)$.
	
	$(iii)$ For $\MV=\mathbb{C}TM$, let $E$ be a trivial line bundle over $(M,\MV)$. Corollary \ref{local existence1} $(i)$ states that $H^{q'}(M,\mathbb{C})=0$ for $q'\geq q$, provided that $M$ admits a smooth $q$-convex exhaustion function with respect to $\MV$. In the special case where $\varphi$ is also a Morse function, since
	\begin{center}
		$\mathcal{K}_{\varphi}=$ the critical locus of $\varphi$, $Q_{\varphi,{\mfp}}={\rm Hess}_\varphi(\lmfp)$ $($Definition \ref{quadratic.}$)$,
	\end{center}
	the $q$-convexity of $\varphi$ with respect to $\MV$ $($Definition \ref{qpqc}$)$ is equivalent to requiring its indices to be at most $q-1$. The fundamental theorem of Morse theory $($\cite{Mj63}$)$ yields that $M$ is homotopy equivalent to a $CW$-complex with exactly one $q''$-cell for each critical point of $\varphi$ of index $q''$ $(\leq q-1)$, in particular, $H^{q'}(M,\mathbb{C})=0$ for $q'\geq q$.
\end{remark}

Recently, significant progress has been made on the global solvability of the Treves complexes associated with tube structures on compact manifolds (see \cite{HZ17}, \cite{HZ19}, \cite{BdZ21}, \cite{ADd23}, \cite{AFJR24} and references therein). In Section \ref{global.}, we show that the complex (\ref{complex1}) gives a global realization of the Morse-Novikov-Treves complex associated with $\MV$ if $E$ is chosen to be the dual bundle of $K_\MV$ (Proposition \ref{gr1}).
Moreover, our work demonstrates the global solvability of the Morse-Novikov-Treves complexes associated with Levi flat structures on non-compact manifolds, which follows directly from Proposition \ref{gr1} and Theorem \ref{t1}.
\begin{cor}\label{treves'}
	For any Levi flat structure $\MV$, let $\vartheta$ be a global smooth $1$-form satisfying $(\ref{mnf})$. If $M$ has a smooth $q$-convex exhaustion function with respect to $\MV$, then the Morse-Novikov-Treves complex $($see $(\ref{new complex}))$ is globally $L_{loc}^2$-solvable in degree $q'\geq q$, i.e., for every $f\in L_{loc}^2(M,\Lambda^{q'}\mathbb{C}T^*M/\Lambda_\MV^{1,q'-1})$ with $\rd_\vartheta'f=0$, there is a $u\in L_{loc}^2(M,\Lambda^{q'-1}\mathbb{C}T^*M/\Lambda_\MV^{1,q'-2})$ such that $\rd_\vartheta'u=f$. In particular, this global $L_{loc}^2$-solvability can be strengthened to global smooth solvability when $\MV$ is elliptic.
\end{cor}

For any Levi flat structure $\MV$, the regularity result in \cite{T2} yields a resolution of $\mathcal{O}_\MV(E)$, where $E$ is a basic vector bundle (Corollary \ref{sros}). We also establish vanishing results for the leafwise $L_{loc}^2$-cohomology (Corollary \ref{vlmnc1} and Corollary \ref{vlmnc2}), the latter of which provides a partial affirmative answer to an open question posed by A. El Kacimi Alaoui (Question 2.10.4 in \cite{EKAa14}) under a suitable geometric assumption; see Section \ref{iso} for details.

In Section \ref{logari.}, we introduce a notion of logarithmic forms along a basic hypersurface $D$ (Definition \ref{basichyper}) of $M$ for any elliptic structure, which, combined with the global solvability of the complex (\ref{complex1}), allows us to compute the singular cohomology of $M$ and $M\setminus D$, and to obtain an extension result. 
\begin{cor}\label{sc}
	Let $\MV$ be an elliptic structure over a smooth manifold $M$, $p\geq 0$. If $M$ admits a smooth $1$-convex exhaustion function with respect to $\MV$, then we have
	\begin{enumerate}
		\item[$(i)$] $H^p(M,\mathbb{C}) =\frac{\mathrm{Ker}\left(\Gamma(M,\Omega_\MV^p)\overset{\rd}{\longrightarrow}\Gamma(M,\Omega_\MV^{p+1})\right)}{\mathrm{Im}\left(\Gamma(M,\Omega_\MV^{p-1})\overset{\rd}{\longrightarrow}\Gamma(M,\Omega_\MV^{p})\right)};$
		
		\item[$(ii)$] for any normal crossing basic hypersurface $D\subseteq M$,
		\begin{align*}
			H^p(M\setminus D,\mathbb{C}) =\frac{\mathrm{Ker}\left(\Gamma\big(M,\Omega_\MV^p(\log D)\big)\overset{\rd}{\longrightarrow}\Gamma\big(M,\Omega_\MV^{p+1}(\log D)\big)\right)}{\mathrm{Im}\left(\Gamma\big(M,\Omega_\MV^{p-1}(\log D)\big)\overset{\rd}{\longrightarrow}\Gamma\big(M,\Omega_\MV^{p}(\log D)\big)\right)},
		\end{align*}
		in particular, for $p>{\rm rank}_{\mathbb{C}}N^*\MV$
		\begin{align*}
			H^p(M,\mathbb{C})=H^p(M\setminus D,\mathbb{C})=0;
		\end{align*}
		\item[$(iii)$] if $D$ is a smooth basic hypersurface with the induced elliptic structure $\MV_D$, then the restriction homomorphism ${R}$ $($induced by $(\ref{sfm})$$)$
		\begin{align*}
			{R}:\Gamma(M,\Omega_\MV^m\otimes[D])\rightarrow \Gamma(D,\Omega_{\MV_D}^{m-1})\ \text{is surjective}.
		\end{align*}
	\end{enumerate}
\end{cor}

\section{Preliminaries}
In this section, we construct a class of differential complexes for basic vector bundles of a formally integrable structure. 
Then we introduce notions of convexity and positivity determined by Levi flat structures. Finally, we establish a Bochner-type formula and its variant by making an additional integration by parts.

\subsection{Construction of the differential complex}\label{constr.}
Let $\MV$ be a formally integrable structure of rank $n$ and corank $m$ on a smooth manifold $M$. We first introduce the notion of basic vector bundles over $(M,\MV)$.
\begin{definition}\label{blb.}
	A vector bundle $E$ over $(M,\MV)$ is called a basic vector bundle if there exist trivializations $\{(U_\alpha,\zeta_\alpha)\}_\alpha$ of $E$ such that the transition matrices $\zeta_\alpha\circ\zeta_\beta^{-1}$ are annihilated by $\MV$. A section $e_\alpha$ over $U_\alpha$ is said to be basic if the components $\zeta_\alpha( e_\alpha)$ are annihilated by $\MV$. We denote by $\mathcal{S}_\MV(E)$ $($or $\mathcal{O}_\MV(E)$$)$ the sheaf of germs of $L^2_{\text{loc}}$ $($smooth respectively$)$ basic sections of $E$, and call $\mathcal{S}_\MV(E)$ the structure sheaf of $E$.
\end{definition}

The concept of a basic function on $(M,\MV)$ is defined analogously. We give some examples of basic vector bundles as follows.

\begin{example}\label{exa}
	$(i)$ If $\MV$ is a complex structure over $M$, then a basic vector bundle $E$ is exactly a holomorphic vector bundle.
	
	$(ii)$ $N^*\MV$ is a basic vector bundle provided that $\MV$ is locally integrable. In fact, it follows from local integrability of $\MV$ that for any ${\lmfp}_0\in M$, there exist an open neighborhood $U_0$ of ${\lmfp}_0$ and basic functions $v_1,\cdots,v_m\in  C^\infty(U_0)$ such that
	\begin{align*}
		{\rm span}\{{{\rm d}_{\mfp}v_1},\cdots,{{\rm d}_{\mfp}v_m}\}={N_{\mfp}^*\MV},\ {\lmfp}\in U_0.
	\end{align*}
	Let $(U,\{\rd v_\ell\}_{\ell=1}^m)$ and $(V,\{\rd w_\ell\}_{\ell=1}^m)$ be such neighborhoods with non-empty intersection. Then for any $1\leq\jmath\leq m$, there are functions $g_{\jmath\ell}\in C^\infty(U\cap V)$ such that
	\begin{align}\label{tf}
		\rd v_\jmath=\sum_{\ell=1}^{m}g_{\jmath\ell}\rd w_\ell
	\end{align}
	holds on $U\cap V$. Applying $\mathscr{L}_X$ to $(\ref{tf})$ yields that $X(g_{\jmath\ell})=0$ for all $1\leq\jmath,\ell\leq m$, where $X\in C^\infty(U\cap V,\MV)$ and $\mathscr{L}_X$ is the Lie derivative with respect to $X$.
	
	$(iii)$ $\Lambda^p N^*\MV$ is a basic vector bundle for any $1\leq p\leq m$ if $\MV$ is locally integrable.
	
	$(iv)$ If $\MV=\mathbb{C}TM$, then basic vector bundles are precisely flat vector bundles.
\end{example}

We now assume that $\MV$ is a formally integrable structure with the basic canonical bundle $K_{\MV}$. Let $\vartheta$ be a smooth $1$-form on $M$ satisfying
\begin{align}\label{mnf'}
	\rd\vartheta\equiv0\ {\rm mod}\ C^\infty(M,N^*\MV).
\end{align}
We define the differential operator $\dmv $ in the following sequence:
\begin{align}\label{complex11}
	\mathcal{A}_\MV^{m,0}(E)\stackrel{\dmv }{\longrightarrow}\mathcal{A}_\MV^{m,1}(E)
	\stackrel{\dmv }{\longrightarrow}\mathcal{A}_\MV^{m,2}(E)\stackrel{\dmv }{\longrightarrow}\cdots\stackrel{\dmv }{\longrightarrow}\mathcal{A}_\MV^{m,n}(E)\longrightarrow0,
\end{align}
where $E$ is a basic vector bundle of rank $r$ over $(M,\MV)$ and $\mathcal{A}_\MV^{m,q}(E)$ is the sheaf of germs of smooth sections of (\ref{bdl.}). For an open subset $U\subseteq M$, let
\begin{align}\label{lf.}
	\{X_1,\cdots,X_n,P_1,\cdots,P_m\}
\end{align}
be a smooth frame of $\mathbb{C}TM$ over $U$ with the dual frame
\begin{align}\label{df.}
	\{\omega^1,\cdots,\omega^n,\theta^1,\cdots,\theta^m\}
\end{align}
of $\mathbb{C}T^*M$ over $U$, where $\{X_1,\cdots,X_n\}$ and $\{\theta^1,\cdots,\theta^m\}$ span $\MV|_{_U}$ and $N^*\MV|_{_U}$ respectively. The integrability of $\MV$ implies that there are smooth 1-forms $\theta_\ell^\jmath$ on $U$ such that
\begin{align}\label{inte}
	\rd \theta^\jmath=\sum_{\ell=1}^{m}\theta_\ell^\jmath\wedge\theta^\ell,\ \jmath=1,\cdots,m,
\end{align}
and therefore the exterior derivative $\rd$ satisfies
\begin{align}\label{ic}
	\rd (\mathcal{A}_\MV^{p,q})\subseteq \mathcal{A}_\MV^{p,q+1},\ \forall p,q\geq0,
\end{align}
where $\mathcal{A}_\MV^{p,*}$ is the sheaf of germs of smooth sections of $\Lambda_\MV^{p,*}$ (see (\ref{e12})).

The fact that $K_\MV $ is basic allows us to define a morphism
\begin{align*}
	\Xi_\MV:\mathcal{A}_\MV^{m,q}\rightarrow\mathcal{A}_\MV^{m,q+1},
\end{align*}
which admits the following local representation on $U$
\begin{align}\label{tr}
	\Xi_\MV|_{_U}:=\bigg(\sum_{\ell=1}^m\theta_\ell^\ell\bigg)\wedge,
\end{align}
where 
the forms $\theta_\ell^\ell$ are given by (\ref{inte}), in which the local frame $\{\theta^1,\cdots,\theta^m\}$ of $N^*\MV$ is chosen such that $\theta^1\wedge\cdots\wedge\theta^m$ is a local smooth basic section of $K_\MV$. Indeed, if $\{\eta^1,\cdots,\eta^m\}$ is another smooth frame of $N^*\MV$ over $U$ such that
\begin{align*}
	\theta^\jmath=\sum_{\ell=1}^{m}f_\ell^\jmath\eta^\ell,\ \jmath=1,\cdots,m,
\end{align*}
where the functions $f_\ell^\jmath\in C^\infty(U)$ satisfy the condition that $\det(f_\ell^\jmath)$ is a basic function on $U$. By the integrability of $\MV$, there are smooth 1-forms $\eta_\ell^\jmath$ fulfilling
\begin{align*}
	\rd \eta^\jmath=\sum_{\ell=1}^{m}\eta_\ell^\jmath\wedge\eta^\ell,\ \jmath=1,\cdots,m.
\end{align*}
Therefore,
\begin{align}\label{dt1}
	\rd \theta^\jmath=\sum_{\ell=1}^{m}\rd f_\ell^\jmath\wedge\eta^\ell+\sum_{\ell=1}^{m}f_\ell^\jmath\rd \eta^\ell=\sum_{\ell=1}^{m}\rd f_\ell^\jmath\wedge\eta^\ell+\sum_{\imath,\ell=1}^{m}f_\ell^\jmath\eta_\imath^\ell\wedge\eta^\imath.
\end{align}
On the other hand,
\begin{align*}
	\rd \theta^\jmath=\sum_{\ell=1}^{m}\theta_\ell^\jmath\wedge\theta^\ell=\sum_{\imath,\ell=1}^{m}f_\imath^\ell\theta_\ell^\jmath\wedge\eta^\imath,
\end{align*}
which in conjunction with (\ref{dt1}) implies
\begin{align*}
	\sum_{\ell=1}^{m}f_\imath^\ell\theta_\ell^\jmath\equiv\bigg(\rd f_\imath^\jmath+\sum_{\ell=1}^{m}f_\ell^\jmath\eta_\imath^\ell\bigg)\ {\rm mod}\ C^\infty(U,N^*\MV).
\end{align*}
Thus,
\begin{align*}
	\theta_{\ell'}^\jmath\equiv \bigg(\sum_{\imath=1}^mh_{\ell'}^\imath\rd f_\imath^\jmath+\sum_{\imath,\ell=1}^{m}h_{\ell'}^\imath f_\ell^\jmath\eta_\imath^\ell\bigg)\ {\rm mod}\ C^\infty(U,N^*\MV),
\end{align*}
where $(h_{\ell'}^\imath)_{1\leq\imath,\ell'\leq m}$ is the inverse matrix of $(f_\imath^\ell)_{1\leq\imath,\ell\leq m}$. It turns out that
\begin{align*}
	\sum_{\jmath=1}^m\theta_{\jmath}^\jmath&\equiv\bigg(\sum_{\imath,\jmath=1}^mh_{\jmath}^\imath\rd f_\imath^\jmath+\sum_{\imath,\jmath,\ell=1}^{m}h_{\jmath}^\imath f_\ell^\jmath\eta_\imath^\ell\bigg)\ {\rm mod}\ C^\infty(U,N^*\MV)\\
	&=\bigg(\rd (\log\det(f_\imath^\jmath))+\sum_{\ell=1}^{m}\eta_\ell^\ell\bigg)\ {\rm mod}\ C^\infty(U,N^*\MV)\\
	&\equiv\bigg(\sum_{\ell=1}^{m}\eta_\ell^\ell\bigg)\ {\rm mod}\ C^\infty(U,N^*\MV),
\end{align*}
where the last line holds since $\det(f_\imath^\jmath)$ is a basic function. This confirms that (\ref{tr}) gives a well-defined morphism $\Xi_\MV:\mathcal{A}_\MV^{m,q}\rightarrow\mathcal{A}_\MV^{m,q+1}$. In particular, $\Xi_\MV$ can be chosen to vanish if $\MV$ is locally integrable.

Let $\{U_\alpha\}_\alpha$ be an open covering of $M$, and let
\begin{align*}
	\{(U_\alpha\cap U_\beta,g_{\alpha\beta})\}_{\alpha,\beta}
\end{align*}
be a basic cocycle defining $E$. For
\begin{align*}
	u=\{u_\alpha\otimes  e_\alpha\}_\alpha\in C^\infty(M,\Lambda_\MV^{m,q}(E)),
\end{align*}
where $u_\alpha:=(u_\alpha^1,\cdots,u_\alpha^r)$ in which $u_\alpha^a\in C^\infty(U_\alpha,\Lambda_\MV^{m,q})$ for all $1\leq a\leq r$, and $ e_\alpha:=(e_\alpha^1,\cdots,e_\alpha^r)^T$ is a smooth basic frame of $E$ over $U_\alpha$ satisfying $ e_\beta=g_{\alpha\beta}\cdot e_\alpha$ on $U_\alpha\cap U_\beta$. Then $u_\alpha=u_\beta\cdot g_{\alpha\beta}$, taking into account the fact that $g_{\alpha\beta}$ is annihilated by $\MV$ yields
\begin{align*}
	(\rd -\Xi_\MV)u_\alpha=\rd u_\beta\cdot g_{\alpha\beta}+(-1)^{q+m}u_\beta\wedge\rd g_{\alpha\beta}-\Xi_\MV(u_\beta\cdot g_{\alpha\beta})=(\rd -\Xi_\MV)u_\beta\cdot g_{\alpha\beta},
\end{align*}
where $(\rd -\Xi_\MV)$ acts componentwise on forms. The relation (\ref{ic}) immediately implies that $\rd u_\alpha^a\in C^\infty(U_\alpha,\Lambda_\MV^{m,q+1})$ for $1\leq a\leq r$. Hence, we can define the operator $\dmv:C^\infty(M,\Lambda_\MV^{m,q}(E))\rightarrow C^\infty(M,\Lambda_\MV^{m,q+1}(E))$ as follows:
\begin{align}\label{bo}
	\dmv u=\{(\rd -\Xi_\MV-\vartheta\wedge)u_\alpha\otimes  e_\alpha\}_\alpha,
\end{align}
where $\Xi_\MV$ is given by (\ref{tr}) and $\vartheta$ is a global smooth $1$-form satisfying (\ref{mnf'}). Henceforth, we will denote $\rd_{\MV,0}$ by $\dV$ for simplicity.

For later use, we calculate the local expression of the operator $\dmv$ on $U_\alpha$ in terms of arbitrary local frame $\tilde{e}_\alpha$ of $E$. Let $\tilde{f}_\alpha:U_\alpha\rightarrow GL(r,\mathbb{C})$ be a smooth map such that $e_\alpha=\tilde{f}_\alpha\cdot\tilde{e}_\alpha$, then for $u=u_\alpha\otimes e_\alpha=\tilde{u}_\alpha\otimes\tilde{e}_\alpha\in C^\infty(U_\alpha,\Lambda_\MV^{m,q}(E))$
\begin{align}\label{lefaf}
	\dmv u&=(\rd -\Xi_\MV-\vartheta\wedge)u_\alpha\otimes e_\alpha\nonumber\\
	&=(\rd -\Xi_\MV-\vartheta\wedge)(\tilde{u}_\alpha\cdot\tilde{f}_\alpha^{-1})\otimes\tilde{f}_\alpha\cdot\tilde{e}_\alpha\nonumber\\
	&=(\rd -\Xi_\MV-\vartheta\wedge)\tilde{u}_\alpha\otimes\tilde{e}_\alpha+(-1)^{q+m}\tilde{u}_\alpha\wedge\rd\tilde{f}_\alpha^{-1}\cdot\tilde{f}_\alpha\otimes\tilde{e}_\alpha.
\end{align}
The operator $\dmv$ may a priori depend on the choice of a basic cocycle. However, the difference between two such operators (arising from different basic cocycles) can be explicitly characterized.
\begin{prop}
	Let $\dmv$ and $\dmv^{'}$ denote the operators determined by two basic cocycles $\{(U_\alpha\cap U_\beta,g_{\alpha\beta})\}_{\alpha,\beta}$ and $\{(U_\alpha'\cap U_\beta',g_{\alpha\beta}')\}_{\alpha,\beta}$ respectively. Then $\dmv-\dmv^{'}$ is a well-defined zero-order operator. Moreover, the operator $\dmv$ is uniquely determined by the class $[\{g_{\alpha\beta}\}_{\alpha,\beta}]\in H^1(M,GL(r,\mathcal{O}_\MV ))$, where $\mathcal{O}_\MV $ is the sheaf of germs of smooth basic functions.
\end{prop}
\begin{proof}
	By taking a common refinement, we may assume that both cocycles are defined over the same covering $\{U_\alpha\}_{\alpha}$. Let $\{(U_\alpha, e_\alpha)\}_\alpha$ and $\{(U_\alpha, e_\alpha')\}_\alpha$ be the corresponding frames of $E$ with respect to $\{(U_\alpha\cap U_\beta,g_{\alpha\beta})\}_{\alpha,\beta}$ and $\{(U_\alpha\cap U_\beta,g_{\alpha\beta}')\}_{\alpha,\beta}$ respectively. There exist smooth maps $f_\alpha:U_\alpha\rightarrow GL(r,\mathbb{C})$ such that
	\begin{align*}
		e_\alpha=f_\alpha\cdot e_\alpha'\ \text{on}\ U_\alpha.
	\end{align*}
	Then, we have
	\begin{align*}
		f_\beta\cdot g_{\alpha\beta}'\cdot e_\alpha'=f_\beta \cdot e_\beta'=e_\beta=g_{\alpha\beta}\cdot f_\alpha\cdot e_\alpha'\ \text{on}\ U_\alpha\cap U_\beta,
	\end{align*}
	which gives
	\begin{align}\label{ffb}
		f_\beta\cdot g_{\alpha\beta}'=g_{\alpha\beta}\cdot f_\alpha\ \text{on}\ U_\alpha\cap U_\beta.
	\end{align}
	For $u=\{u_\alpha\otimes e_\alpha\}_\alpha=\{u_\alpha'\otimes e_\alpha'\}_\alpha\in C^\infty(M,\Lambda_\MV^{m,q}(E))$, by (\ref{lefaf}) we obtain
	\begin{align*}
		\dmv u-\dmv^{'}u=(-1)^{q+m}u_\alpha'\wedge\rd f_\alpha^{-1}\cdot f_\alpha\otimes e_\alpha'=:B_\MV u,
	\end{align*}
	where $B_\MV:\Lambda_\MV^{m,q}(E)\rightarrow\Lambda_\MV^{m,q+1}(E)$ is a well-defined bundle homomorphism, since by (\ref{ffb}) we know that
	\begin{align*}
		u_\alpha'\wedge\rd f_\alpha^{-1}\cdot f_\alpha\otimes e_\alpha'&=u_\alpha'\wedge\rd(g_{\alpha\beta}'^{-1}\cdot f_\beta^{-1}\cdot g_{\alpha\beta})\cdot f_\alpha\otimes e_\alpha'\\
		&=u_\alpha'\wedge(g_{\alpha\beta}'^{-1}\cdot\rd f_\beta^{-1}\cdot g_{\alpha\beta})\cdot f_\alpha\otimes e_\alpha'\\
		&=(u_\alpha'\cdot g_{\alpha\beta}'^{-1})\wedge\rd f_\beta^{-1}\cdot(g_{\alpha\beta}\cdot f_\alpha)\otimes e_\alpha'\\
		&=u_\beta'\wedge\rd f_\beta^{-1}\cdot(f_\beta\cdot g_{\alpha\beta}')\otimes e_\alpha'\\
		&=u_\beta'\wedge\rd f_\beta^{-1}\cdot f_\beta\otimes e_\beta'.
	\end{align*}
	In particular, $B_\MV=0$ if the matrix-valued functions $f_\alpha$ are chosen to be basic. Thus, the operator $\dmv$ is determined by the class $[\{g_{\alpha\beta}\}_{\alpha,\beta}]\in H^1(M,GL(r,\mathcal{O}_\MV ))$.
\end{proof}

It remains to prove that the sequence (\ref{complex11}) is actually a differential complex as follows.
\begin{prop}
	$\dmv^2=0$.
\end{prop}
\begin{proof}
	Applying ${\rm d}$ to (\ref{inte}) we obtain that for $1\leq\jmath\leq m$
	\begin{align*}
		0=\sum_{\ell=1}^{m}\rd \theta_\ell^\jmath\wedge\theta^\ell-\sum_{\ell=1}^{m}\theta_\ell^\jmath\wedge\rd \theta^\ell=\sum_{\ell=1}^{m}\rd \theta_\ell^\jmath\wedge\theta^\ell-\sum_{\imath,\ell=1}^{m}\theta_\ell^\jmath\wedge\theta_\imath^\ell\wedge\theta^\imath,
	\end{align*}
	which yields
	\begin{align}\label{inte1}
		\rd \theta_\ell^\jmath\equiv\bigg(\sum_{\imath=1}^{m}\theta_\imath^\jmath\wedge\theta_\ell^\imath\bigg)\ {\rm mod}\ C^\infty(U,N^*\MV),\ \jmath,\ell=1,\cdots,m.
	\end{align}
	Hence, for any open subset $U\subseteq M$ with local frame (\ref{df.}) of $\mathbb{C}T^*M$, we have
	\begin{align*}
		\dmv^2&=(\rd -\Xi_\MV-\vartheta\wedge)\circ(\rd -\Xi_\MV-\vartheta\wedge)\\
		&=-\sum_{\ell=1}^m\rd (\theta_{\ell}^\ell\wedge)-\rd(\vartheta\wedge)-\sum_{\ell=1}^m\theta_{\ell}^\ell\wedge\rd +\sum_{\jmath,\ell=1}^m\theta_{\jmath}^\jmath\wedge\theta_{\ell}^\ell\wedge-\vartheta\wedge\rd\\
		&=-\sum_{\ell=1}^m\rd \theta_{\ell}^\ell\wedge-\rd\vartheta\wedge\\
		&\equiv-\bigg(\sum_{\imath,\ell=1}^m\theta_\imath^\ell\wedge\theta_\ell^\imath\wedge\bigg)\ {\rm mod}\ C^\infty(U,N^*\MV)\\
		&=-\bigg(\sum_{\imath<\ell}(\theta_\imath^\ell\wedge\theta_\ell^\imath-\theta_\imath^\ell\wedge\theta_\ell^\imath)\wedge\bigg)\ {\rm mod}\ C^\infty(U,N^*\MV)\\
		&=0,
	\end{align*}
	where the fourth line is valid by (\ref{inte1}) and (\ref{mnf'}).
\end{proof}

For $0\leq q\leq n$, we introduce the presheaf $\mathcal{L}_\MV^{m,q}(E)$ as follows:
\begin{align}\label{subs}
	\big(\mathcal{L}_\MV^{m,q}(E)\big)(U):=\left\{u\in L_{loc}^2(U,\Lambda_\MV^{m,q}(E))\ |\ \dmv u\in L_{loc}^2(U,\Lambda_\MV^{m,q+1}(E))\right\},
\end{align}
where $U$ is any open subset of $M$. Obviously, (\ref{subs}) defines a complete presheaf over $M$, and we denote by the same symbol $\mathcal{L}_\MV^{m,q}(E)$ the sheaf generated by it. The complex (\ref{complex11}) gives rise naturally to the following complex of sheaf morphisms over $M$
\begin{align}\label{complex12}
	\mathcal{L}_\MV^{m,0}(E)\stackrel{\dmv }{\longrightarrow}\mathcal{L}_\MV^{m,1}(E)
	\stackrel{\dmv }{\longrightarrow}\mathcal{L}_\MV^{m,2}(E)\stackrel{\dmv }{\longrightarrow}\cdots\stackrel{\dmv }{\longrightarrow}\mathcal{L}_\MV^{m,n}(E)\longrightarrow0.
\end{align}
The complex (\ref{complex11}) is thus the subcomplex of the complex (\ref{complex12}).

\subsection{Notions of \texorpdfstring{$q$}{}-convexity and \texorpdfstring{$q$}{}-positivity}
Let $L$ be a basic line bundle over $(M,\MV)$, and fix a Hermitian metric $h_{_L}$ on $L$. Let $\{U_\alpha\}_\alpha$ be an open covering of $M$ and $\sigma_\alpha$ a smooth basic frame of $L$ over $U_\alpha$. We may assume that
\begin{align*}
	\sigma_\beta=\sigma_\alpha\phi_{\alpha\beta}\ \text{on}\ U_\alpha\cap U_\beta,
\end{align*}
where $\phi_{\alpha\beta}$ is the basic transition function of $L$ for each pair of $(\alpha,\beta)$. Write 
\begin{align*}
	h_{_L}(\sigma_\alpha,\sigma_\alpha)=e^{-\phi_\alpha}.
\end{align*}
Then
\begin{align}\label{lm.}
	\phi_\alpha=\phi_\beta+\log|\phi_{\alpha\beta}|^2\ \text{on}\ U_\alpha\cap U_\beta.
\end{align}
Since the transition functions $\phi_{\alpha\beta}$ are annihilated by $\MV$, there exists a well-defined subset
\begin{align*}
	\mathcal{K}_{h_L}\subseteq M
\end{align*}
such that for each $\alpha$
\begin{align*}
	\mathcal{K}_{h_L}\cap U_\alpha=\{{\lmfp}\in U_\alpha\ |\ X_{\mfp}(\phi_\alpha)=0,\ \forall X_{\mfp}\in\MV_{\mfp}\cap\overline{\MV}_{\mfp}\}.
\end{align*}

From now on, we assume $\MV$ is a complex Frobenius structure, i.e., for any smooth local frame $\{X_1,\cdots,X_n\}$ of $\MV$, there are smooth local functions $d_{jk}^l, e_{jk}^l$ such that
\begin{equation}\label{cbl}
	\left[X_j,\bar{X}_k\right]=\sum_{l=1}^nd_{jk}^lX_l-\sum_{l=1}^ne_{jk}^l\bar{X}_l.
\end{equation}

\begin{remark}
	The coefficients $ d_{jk}^l $ and $ e_{jk}^l $ need not be uniquely determined.
\end{remark} 

The discrepancy relation (\ref{lm.}) allows us to introduce a quadratic form on $\MV_{\mfp}$ at $\lmfp\in\mathcal{K}_{h_L}$ for a complex Frobenius structure $\MV$.
\begin{definition}\label{quadratic.}
	Let $\MV$ be a complex Frobenius structure, and let ${\lmfp}\in\mathcal{K}_{h_L}\cap U_\alpha$. The quadratic form for the metric $h_{_L}$ on $\MV_{\mfp}$ is defined by
	\begin{equation}\label{qf.}
		Q_{h_L,{\mfp}}(\xi,\xi)={\rm Re}\sum_{j,k=1}^n\bigg(X_j\bar{X}_k\phi_\alpha({\lmfp})+\sum_{l=1}^ne_{jk}^l({\lmfp})\bar{X}_l\phi_\alpha({\lmfp})\bigg)\xi_{j}\bar{\xi}_{k},
	\end{equation}
	where $\{X_j\}_{j=1}^n$ is a smooth frame of $\MV$ over $U_\alpha$ and $ \xi=\sum_{j=1}^n\xi_jX_j|_{\mfp}\in\MV_{\mfp} $. In particular, one can define the quadratic form $Q_{\varphi,{\mfp}}$ like $(\ref{qf.})$ at ${\lmfp}\in\mathcal{K}_\varphi$ for a real-valued function $\varphi\in C^2(M)$ since $e^{-\varphi}$ can be viewed as a metric on the trivial bundle.
\end{definition}

This definition for a real-valued function $\varphi\in C^2(M)$ is essentially due to L. H{\"o}rmander (see \cite{Hlv65JT}). The definition of $Q_{h_L,{\mfp}}$ is unambiguous.
\begin{lemma}
	The quadratic form $Q_{h_L,\mfp}$ is well-defined.
\end{lemma}
\begin{proof}
	We will deduce the conclusion in the following three steps.
	
	\emph{Step 1.} We first show that on each $U_\alpha\cap\mathcal{K}_{h_L}$ the quadratic form (\ref{qf.}) is independent of the choice of the functions $e_{jk}^l$. Indeed, if $\tilde{e}_{jk}^l$ and $\tilde{d}_{jk}^l$ satisfy (\ref{cbl}), then
	\begin{align*}
		\sum_{l=1}^n(\tilde{e}_{jk}^l-e_{jk}^l)\bar{X}_l=\sum_{l=1}^n(\tilde{d}_{jk}^l-d_{jk}^l)X_l,
	\end{align*}
	it follows from ${\lmfp}\in\mathcal{K}_{h_L}\cap U_\alpha$ that
	\begin{align*}
		\sum_{l=1}^n\tilde{e}_{jk}^l({\lmfp})\bar{X}_l\phi_\alpha({\lmfp})=\sum_{l=1}^ne_{jk}^l({\lmfp})\bar{X}_l\phi_\alpha({\lmfp}).
	\end{align*}
	
	\emph{Step 2.} Let $\{Y_j\}_{j=1}^n$ be another frame of $\MV$ over $U_\alpha$, then $Y_j=\sum_{k=1}^na_{jk}X_k$ for $1\leq j\leq n$ where $(a_{jk})_{1\leq j,k\leq n}$ is a $GL(n,\mathbb{C})$-valued smooth function. If
	\begin{align*}
		\left[Y_j,\bar{Y}_k\right]=\sum_{l=1}^n\hat{d}_{jk}^lY_l-\sum_{l=1}^n\hat{e}_{jk}^l\bar{Y}_l,
	\end{align*}
	then
	\begin{align*}
		\sum_{l,j'=1}^{n}\left[a_{jl}X_l,\bar{a}_{kj'}\bar{X}_{j'}\right]=\left[Y_j,\bar{Y}_k\right]=\sum_{l,j'=1}^n\hat{d}_{jk}^la_{lj'}X_{l'}-\sum_{l,j'=1}^n\hat{e}_{jk}^l\bar{a}_{lj'}\bar{X}_{l'},
	\end{align*}
	combined with (\ref{cbl}) implies
	\begin{align*}
		&\sum_{j'=1}^n\bigg(\sum_{l,k'=1}^na_{jl}\bar{a}_{k{k'}}d_{l{k'}}^{j'}-\sum_{l'=1}^n\hat{d}_{jk}^{l'}a_{{l'}{j'}}-\sum_{k'=1}^n\bar{a}_{k{k'}}\bar{X}_{k'}(a_{j{j'}})\bigg)X_{j'}\\
		=&\sum_{j'=1}^n\bigg(\sum_{l,k'=1}^na_{jl}\bar{a}_{k{k'}}e_{l{k'}}^{j'}-\sum_{l'=1}^n\hat{e}_{jk}^{l'}\bar{a}_{{l'}{j'}}-\sum_{l=1}^na_{jl}X_l(\bar{a}_{k{j'}})\bigg)\bar{X}_{j'}.
	\end{align*}
	Since ${\lmfp}\in\mathcal{K}_{h_L}\cap U_\alpha$, we have 
	\begin{align*}
		\sum_{l,j',k'=1}^na_{jl}({\lmfp})\bar{a}_{k{k'}}({\lmfp})e_{l{k'}}^{j'}({\lmfp})\bar{X}_{j'}\phi_\alpha({\lmfp})-\sum_{l,j'=1}^na_{jl}({\lmfp})X_l(\bar{a}_{k{j'}})({\lmfp})\bar{X}_{j'}\phi_\alpha({\lmfp})=\sum_{j',l'=1}^n\hat{e}_{jk}^{l'}({\lmfp})\bar{a}_{{l'}{j'}}({\lmfp})\bar{X}_{j'}\phi_\alpha({\lmfp}),
	\end{align*}
	which gives
	\begin{align*}
		Y_j\bar{Y}_k\phi_\alpha({\lmfp})+\sum_{l=1}^n\hat{e}_{jk}^l({\lmfp})\bar{Y}_l\phi_\alpha({\lmfp})=\sum_{j',k'=1}^na_{jj'}({\lmfp})\bar{a}_{k{k'}}({\lmfp})\bigg(X_{j'}\bar{X}_{k'}\phi_\alpha({\lmfp})+\sum_{l=1}^ne_{j'k'}^l({\lmfp})\bar{X}_l\phi_\alpha({\lmfp})\bigg).
	\end{align*}
	The quadratic form (\ref{qf.}) is thereby invariant under non-singular linear transformations of local frames of $\MV$.
	
	\emph{Step 3.} For any ${\lmfp}\in\mathcal{K}_{h_L}\cap U_\alpha\cap U_\beta$, the fact that $\phi_{\alpha\beta}$'s are annihilated by $\MV$, together with the relation (\ref{lm.}), leads to
	\begin{align*}
		&X_j\bar{X}_k\phi_\alpha({\lmfp})+\sum_{l=1}^ne_{jk}^l({\lmfp})\bar{X}_l\phi_\alpha({\lmfp})\\
		=&X_j\bar{X}_k\phi_\beta({\lmfp})+\sum_{l=1}^ne_{jk}^l({\lmfp})\bar{X}_l\phi_\beta({\lmfp})+X_j\bar{X}_k\log|\phi_{\alpha\beta}|^2({\lmfp})+\sum_{l=1}^ne_{jk}^l({\lmfp})\bar{X}_l\log|\phi_{\alpha\beta}|^2({\lmfp})\\
		=&X_j\bar{X}_k\phi_\beta({\lmfp})+\sum_{l=1}^ne_{jk}^l({\lmfp})\bar{X}_l\phi_\beta({\lmfp})+\frac{X_j\bar{X}_k\phi_{\alpha\beta}({\lmfp})}{\phi_{\alpha\beta}({\lmfp})}+\sum_{l=1}^ne_{jk}^l({\lmfp})\frac{\bar{X}_l\phi_{\alpha\beta}({\lmfp})}{\phi_{\alpha\beta}({\lmfp})}\\
		=&X_j\bar{X}_k\phi_\beta({\lmfp})+\sum_{l=1}^ne_{jk}^l({\lmfp})\bar{X}_l\phi_\beta({\lmfp})+\frac{\big[X_j,\bar{X}_k\big]\phi_{\alpha\beta}({\lmfp})}{\phi_{\alpha\beta}({\lmfp})}+\sum_{l=1}^ne_{jk}^l({\lmfp})\frac{\bar{X}_l\phi_{\alpha\beta}({\lmfp})}{\phi_{\alpha\beta}({\lmfp})}\\
		=&X_j\bar{X}_k\phi_\beta({\lmfp})+\sum_{l=1}^ne_{jk}^l({\lmfp})\bar{X}_l\phi_\beta({\lmfp})+\frac{\sum_{l=1}^n(d_{jk}^l(\lmfp)X_l-e_{jk}^l(\lmfp)\bar{X}_l)\phi_{\alpha\beta}({\lmfp})}{\phi_{\alpha\beta}({\lmfp})}+\sum_{l=1}^ne_{jk}^l({\lmfp})\frac{\bar{X}_l\phi_{\alpha\beta}({\lmfp})}{\phi_{\alpha\beta}({\lmfp})}\\
		=&X_j\bar{X}_k\phi_\beta({\lmfp})+\sum_{l=1}^ne_{jk}^l({\lmfp})\bar{X}_l\phi_\beta({\lmfp}),
	\end{align*}
	it turns out that the quadratic form $Q_{h_L,{\mfp}}$ is well-defined at ${\lmfp}\in\mathcal{K}_{h_L}$.
\end{proof}

Assume that $\mathcal{V}$ is a Levi flat structure. From the perspective of Morse theory, the quadratic form $Q_{h_L,\mfp}$ arises naturally. By the complex Frobenius theorem (see \cite{Nl57}), for any $\lmfp\in M$, there exists a coordinate chart
\begin{align*}
	(U;z,y,t)=(U;x_1+\sqrt{-1}x_{1+d},\cdots, x_{d}+\sqrt{-1}x_{2d},y_1,\cdots,y_{m-d},t_{1},\cdots,t_{n-d})
\end{align*}
centered at $\lmfp$, where $d:=m-{\rm rank}_\mathbb{C}(N^*\MV\cap\overline{N^*\MV})$, such that
\begin{align*}
	\MV|_{_U}\ \text{is spanned by}\ \left\{\frac{\partial}{\partial\bar{z}_\varrho}, \frac{\partial}{\partial t_\tau}\right\}_{\begin{subarray}{}1\leq\varrho\leq d\\1\leq\tau\leq n-d\end{subarray}}.
\end{align*}
In terms of this local frame of $\MV$, the quadratic form $Q_{h_L,\mfp}$ at $\lmfp\in\mathcal{K}_{h_L}\cap U_\alpha$ takes the form
\begin{align}\label{qoe.}
	\sum_{\varrho_{_1},\varrho_{_2}=1}^d\frac{\partial^2\phi_\alpha(\lmfp)}{\partial z_{\varrho_{_1}}\partial \bar{z}_{\varrho_{_2}}}\rd z_{\varrho_{_1}}\rd\bar{z}_{\varrho_{_2}}+2{\rm Re}\sum_{\varrho=1}^{d}\sum_{\tau=1}^{n-d}\frac{\partial^2\phi_\alpha(\lmfp)}{\partial z_\varrho\partial t_\tau}\rd z_\varrho\rd t_\tau+\sum_{\tau_{_1},\tau_{_2}=1}^{n-d}\frac{\partial^2\phi_\alpha(\lmfp)}{\partial t_{\tau_{_1}}\partial t_{\tau_{_2}}}\rd t_{\tau_{_1}}\rd t_{\tau_{_2}}.
\end{align}
If $\MV$ defines a Levi flat CR structure (i.e., $n=d$), then the restriction of $L$ to a Levi leaf $F$ (a complex submanifold of real codimension $m-d$) of $M$ is exactly a holomorphic line bundle. In this case, \eqref{qoe.} corresponds precisely to the curvature of the Chern connection of $h_{_L}|_{_F}$. On the other hand, if $\MV$ is an essentially real structure (i.e., $d\equiv0$), there exists a foliation $\{F_\alpha\}_\alpha$ of $M$ defined by $\MV$. For a real-valued function $\varphi$, we have
\begin{equation}\label{ers}
	\mathcal{K}_\varphi=\bigcup_{\alpha}\mathcal{C}_{\varphi_\alpha},\ \text{the form (\ref{qoe.}) for}\ \varphi\ \text{is the Hessian of}\ \varphi_\alpha\ \text{at any}\ \lmfp\in\mathcal{C}_{\varphi_\alpha},
\end{equation}
where $\varphi_\alpha:=\varphi|_{_{F_\alpha}}$ and $\mathcal{C}_{\varphi_\alpha}$ is the critical locus of $\varphi_\alpha$. Thus, the quadratic form $Q_{h_L,\mfp}$ connects the curvature form of the Chern connection in complex geometry with the Hessian of a function at its critical points in Morse theory.

We introduce the following concepts concerning the positivity of $Q_{h_L,{\mfp}}$.
\begin{definition}\label{qpqc}
	Let $\MV$ be a complex Frobenius structure, $L$ a basic line bundle over $(M,\MV)$ with a Hermitian metric $h_{_L}$, and $1\leq q\leq n$.
	\begin{enumerate}
		\item[$(i)$] The Hermitian line bundle $(L,h_{_L})$ is said to be $q$-positive with respect to $\MV$, provided that the quadratic form $Q_{h_L,{\mfp}}$ has at least $n-q+1$ positive eigenvalues on $\MV_{\mfp}$ at each point ${\lmfp}\in\mathcal{K}_{h_L}$.
		\item[$(ii)$] A real-valued function $\varphi\in C^2(M)$ is called $q$-convex with respect to $\MV$, if $Q_{\varphi,{\mfp}}$ has at least $\dim_\mathbb{C}(\MV_{\mfp}\cap{\rm Ker}(\rd\varphi))-q+1$ positive eigenvalues in the subspace $\MV_{\mfp}\cap{\rm Ker}(\rd\varphi)$ at each point ${\lmfp}\in\mathcal{K}_\varphi$.
	\end{enumerate}
\end{definition}

\begin{example}\label{qc.}
	$(i)$ If $\MV$ is a complex structure, then the minimum-maximum principle for the eigenvalues of the quadratic form $Q_{\varphi,\mfp}\ (=\sqrt{-1}\partial\bar{\partial}\varphi(\lmfp))$ ensures that any $q$-convex function in the sense of Grauert $($i.e., the form $\sqrt{-1}\partial\bar{\partial}\varphi(\lmfp)$ has at least $n-q+1$ positive eigenvalues at $\lmfp\in M$$)$ is $q$-convex in our sense.
	
	$(ii)$ When $\MV$ is essentially real, then $\MV$ yields a foliation $\{F_\alpha\}_\alpha$ on $M$. By $(\ref{ers})$ we know that if $\varphi_\alpha:=\varphi|_{_{F_\alpha}}$ is a Morse function for every $\alpha$, then $\varphi$ is $q$-convex with respect to $\MV$ if and only if each $\varphi_\alpha$ has index $\leq q-1$.
	
	$(iii)$ Let $M$ be an $m$-dimensional complex manifold, and $\mathscr{M}$ an $(m+n)$-dimensional smooth manifold $(n\geq m)$. Assume that there exists a proper submersion $\pi:\mathscr{M}\rightarrow M$, then $\mathscr{M}$ has a natural elliptic structure $\MV$ of rank $n$.
	\begin{enumerate}
		\item[($iii_a$)] If $M$ is $q$-complete in the sense of Grauert, then there is a smooth $(n-m+q)$-convex exhaustion function with respect to $\MV$ on $\mathscr M$.
		\item[($iii_b$)] If there is a Hermitian line bundle $(L,h_{_L})$ over $M$ such that the curvature of the Chern connection on $(L,h_{_L})$ has at least $m-q+1$ positive eigenvalues on $M$ $(1\leq q\leq m)$, then the pull back $(\pi^*L,\pi^*h_{_L})$ of $(L,h_{_L})$ by $\pi$ over $\mathscr{M}$ is $(n-m+q)$-positive with respect to $\MV$.
	\end{enumerate}
	
	$(iv)$ Let $M$ be an $m$-dimensional complex manifold which is $q$-complete in the sense of Grauert, let $G$ be a Lie group, and $F$ an $(n-m)$-dimensional $G$-manifold. Suppose that
	\begin{align*}
		F\hookrightarrow\mathscr{M}\stackrel{\pi}{\longrightarrow}M
	\end{align*}
	is a fiber bundle  over $M$ with structure group $G$. If there exists a $G$-invariant Morse exhaustion function $\phi$ on $F$ with index $\leq p$, then $\mathscr{M}$ admits an elliptic structure $\MV$ of rank $n$ and a smooth $(p+q)$-convex exhaustion function with respect to $\MV$.
	
	Indeed, let $\{V_\alpha\}_\alpha$ be an open covering of $M$. On each nonempty intersection $V_\alpha\cap V_\beta$ we have
	\begin{align*}
		(\mfp,\mathfrak{f})\sim(\mfp,\mathfrak{g}_{\alpha\beta}(\mfp)\cdot \mathfrak{f}),
	\end{align*}
	where $(\mfp,\mathfrak{f})\in(V_\alpha\cap V_\beta)\times F$ and $\mathfrak{g}_{\alpha\beta}(\mfp)\in G$ acts smoothly on the fiber $F$. It follows that $\mathscr{M}$ has an elliptic structure $\MV$ of rank $n$, which is locally spanned by
	\begin{align*}
		\left\{\frac{\partial}{\partial\bar{z}_\varrho},\frac{\partial}{\partial t_\tau}\right\}_{\begin{subarray}{}1\leq \varrho\leq m\\1\leq \tau\leq n-m\end{subarray}},
	\end{align*}
	where $(z_1,\cdots,z_m)$ are local holomorphic coordinates on $M$ and $(t_1,\cdots,t_{n-m})$ are local smooth coordinates on $F$. Since $\phi$ is $G$-invariant, there is a function $\Phi\in C^\infty(\mathscr{M})$ such that for any small open subset $U\subseteq M$
	\begin{align*}
		\Phi|_{{\pi^{-1}(U)}}={\rm Pr}_{_U}^*\phi,
	\end{align*}
	where
	\begin{align*}
		{\rm Pr}_{_U}:\pi^{-1}(U)\stackrel{trivialization}{\longrightarrow}U\times F\stackrel{{\rm Pr}_{_2}}{\longrightarrow}F,
	\end{align*}
	here ${\rm Pr}_{_2}$ denotes the canonical projection map. Let $\psi\in C^\infty(M)$ be a smooth $q$-convex exhaustion function in the sense of Grauert, then
	\begin{align*}
		\varphi:=\Phi+\pi^*\psi
	\end{align*}
	is a smooth exhaustion function on $\mathscr{M}$. Moreover, $Q_{\varphi,\mfp}$ has at least $n-p-q+1$ positive eigenvalues on $\MV_\mfp$ at $\lmfp\in\mathcal{K}_\varphi$, i.e., $\varphi$ is $(p+q)$-convex with respect to $\MV$.
	
	In particular, we have the following special cases:
	\begin{enumerate}
		\item[($iv_a$)] Let $M$ be an $m$-dimensional complex manifold which is $q$-complete in the sense of Grauert, and let $G$ be a normal subgroup of the fundamental group $\pi_1(M)$. For any $(n-m)$-dimensional $(\pi_1(M)/G)$-manifold $F$, we consider the following flat bundle over $M$
		\begin{align*}
			F\hookrightarrow(\widetilde{M}/G)\times_{(\pi_1(M)/G)}F\rightarrow M,
		\end{align*}
		where $\widetilde{M}$ is the universal cover of $M$. If there exists a $(\pi_1(M)/G)$-invariant Morse exhaustion function on $F$ with index $\leq p$, then $(\widetilde{M}/G)\times_{(\pi_1(M)/G)}F$ has a smooth $(p+q)$-convex exhaustion function with respect to $\MV$.
		
		\item[($iv_b$)] Let $M\subseteq\mathbb{C}^{m}$ be a bounded homogeneous domain, then $M=G/K$, where $G$ is a Lie group and $K$ is a compact subgroup of $G$. Let $F$ be an $(n-m)$-dimensional $K$-manifold, then
		\begin{align}\label{bdl}
			F\hookrightarrow G\times_{K}F\rightarrow M
		\end{align}
		is a fiber bundle over $M$. According to Density Lemma 4.8 in \cite{Was69}, there is a $K$-invariant Morse exhaustion function $\phi$ on $F$. As any bounded homogeneous domain is pseudoconvex, if the index of $\phi$ is at most $p$, then $G\times_{K}F$ possesses a smooth $(p+1)$-convex exhaustion function with respect to $\MV$.

		Specifically, let $F$ be an $(n-m)$-dimensional real vector space, for any linear representation $K\rightarrow GL(F)$, we choose some $K$-invariant inner product $\left<\cdot,\cdot\right>$ on $F$. The bundle $(\ref{bdl})$ is then a vector bundle over $M$, and $\phi(t):=|t|^2$ is a $K$-invariant convex function on $F$. Thus, $G\times_{K}F$ has a smooth $1$-convex exhaustion function with respect to $\MV$ $($i.e., $p=0$ and $q=1)$.
		
		\begin{table}[htbp]
			\centering
			\caption{Bounded Symmetric Domains}
			\vspace{-6pt}
			\begin{adjustbox}{width=\textwidth}
				\begin{tabular}{ccccc}
					\toprule
					Type & Group $G$ & Subgroup $K$ & Vector Space $F$ & Linear Action\\
					\midrule
					I & ${\rm SU}(p,q)$ & ${\rm S}({\rm U}(p)\times {\rm U}(q))$ & $M_{p\times q}(\mathbb{C})$ & $(U,V)\cdot Z=UZV^*$\\
					\midrule
					II & ${\rm SO}^*(2n)$ & ${\rm U}(n)$ & $\Lambda^2\mathbb{C}^n$ & $U\cdot Z=UZU^T$\\
					\midrule
					III & ${\rm Sp}(n,\mathbb{R})$ & ${\rm U}(n)$ & ${\rm Sym}^2(\mathbb{C}^n)$ & $U\cdot Z=UZU^T$\\
					\midrule
					IV & ${\rm SO}_0(n,2)$ & ${\rm SO}(n)\times{\rm SO}(2)$ & $\mathbb{C}^n$ & $(R,\theta)\cdot z=e^{i\theta}Rz$\\
					\midrule
					V & ${\rm E}_{6(-14)}$ & ${\rm Spin}(10)\times{\rm U}(1)$ & $\mathbb{C}^{16}$ & Spinor representation\\
					\midrule
					VI & ${\rm E}_{7(-25)}$ & ${\rm E}_6\times{\rm U}(1)$ & $\mathbb{C}^{27}$ & Fundamental representation of ${\rm E}_6$\\
					\bottomrule
				\end{tabular}
			\end{adjustbox}
		\end{table}
	\end{enumerate}
	
	\vspace{-6pt}
%
	$(v)$ One can construct $q$-convex exhaustion functions and $q$-positive basic line bundles with respect to Levi flat CR structures in a similar way.
\end{example}

\subsection{Setup for \texorpdfstring{$L^2$}{}-theory}
Let $\MV$ be a Levi flat structure of rank $n$ and corank $m$ on $M$, and let $E$ and $L$ be a basic vector bundle of rank $r$ and a basic line bundle over $(M,\MV)$ respectively. Fix Hermitian metrics $\left<\cdot,\cdot\right>$ on $\mathbb{C}TM$ (whose real part induces a Riemannian metric on $M$), $h_{_E}$ on $E$ and $h_{_L}$ on $L$.

For $ 1\leq q\leq n $, we denote by $ L^2_{loc}(M,\Lambda_\MV^{m,q}(E)) $ the space of all $E$-valued $ (m,q) $-forms with locally $ L^2 $-integrable coefficients on $M$. We introduce the following Hilbert space $L^2_\phi(M,\Lambda_\MV^{m,q}(E))$ for a real-valued Borel function $\phi$ on $M$ which is locally bounded from below
\begin{align*}
	L^2_\phi(M,\Lambda_\MV^{m,q}(E))=\bigg\{u\in L^2_{loc}(M,\Lambda_\MV^{m,q}(E))\ \big|\ \int_M |u|_{h_{_E}}^2e^{-\phi}<+\infty\bigg\},
\end{align*}
here and hereafter we denote by $|u|_{h_{_E}}$ the pointwise norm of $u$ with respect to the metric $\left<\cdot,\cdot\right>\cdot h_{_E}$ for simplicity. Let $(\cdot,\cdot)_{\phi}$ denote the inner product on $ L^2_\phi(M,\Lambda_\MV^{m,q}(E)) $, and in slight abuse of notation we still denote  by $ \dmv  $ the maximal extension of $\dmv $ from $ L^2_\phi(M,\Lambda_\MV^{m,q}(E)) $ to $ L^2_\phi(M,\Lambda_\MV^{m,q+1}(E)) $, whose Hilbert adjoint and formal adjoint are denoted by $\rd_{\MV,\vartheta}^{\phi*}$ and $ \delta_{\MV,\vartheta}^\phi $ respectively.

\begin{lemma}[\cite{H3}]\label{functional analysis}
	Let $ T:H_1\rightarrow H_2 $ and $ S:H_2\rightarrow H_3 $ be closed, densely defined linear operators such that $ {\rm Im}(T)\subseteq {\rm Ker}(S) $. If there exists a constant $C>0$ such that
	\begin{equation}\label{priori estimate}
		\|g\|_{H_2}^2\leq C^2(\|T^*g\|_{H_1}^2+\|Sg\|_{H_3}^2),\quad  g\in {\rm Dom}(T^*)\cap{\rm Dom}(S).
	\end{equation}
	Then for any $f\in{\rm Ker}(S)$, one can find $u\in H_1$ such that $Tu=f$ and $\|u\|_{H_1}\leq C\|f\|_{H_2}$.
\end{lemma}

We will apply Lemma \ref{functional analysis} to the following Hilbert spaces
\begin{equation}\label{Hil.}
	H_1=L^2_\phi(M,\Lambda_\MV^{m,q-1}(E)), \ H_2=L^2_\phi(M,\Lambda_\MV^{m,q}(E)), \ H_3 =L^2_\phi(M,\Lambda_\MV^{m,q+1}(E)),
\end{equation}
and operators
\begin{equation}\label{Ope.}
	T,S=\text{the maximal extensions of}\ \dmv,
\end{equation}
where $\phi$ will be determined in the sequel. 

We fix a sequence $\lbrace\eta_\nu\rbrace_{\nu=1}^{\infty}$ of real-valued functions in $C_c^\infty(M)$ such that $ 0\leq\eta_\nu\leq1 $ for all $ \nu $, and $ \eta_\nu=1 $ on any compact subset of $ M $ when $ \nu $ is large. There exists a Hermitian metric on $\mathbb{C}TM$ (e.g., obtained by multiplying a given Hermitian metric on $\mathbb{C}TM$ by a sufficiently large positive smooth function) such that
\begin{equation}\label{cme}
	|\rd \eta_\nu|^2\leq1\ \text{on}\ M,\ \nu=1,2,\cdots.
\end{equation}
By the same argument as in Lemma 5.2.1 in \cite{Hlv90}, we obtain
\begin{lemma}\label{dl.}
	The space $ C^\infty_c(M,\Lambda_\MV^{m,q}(E))$ is dense in $ {\rm Dom}(T^*)\cap{\rm Dom}(S) $ with respect to the graph norm
	\begin{equation*}
		u\rightarrow\|u\|_{H_2}+\|T^*u\|_{H_1}+\|Su\|_{H_3},
	\end{equation*} 
	where $H_1,H_2,H_3$ and $S,T$ are defined by $ (\ref{Hil.}) $ and $ (\ref{Ope.}) $ respectively.
\end{lemma}

To conclude this section, we establish a Bochner-type formula and its variant by making another integration by parts. For any open set $U\subseteq M$, we adopt the same notations (somewhat abusively, for the sake of simplicity) to denote the orthonormalization of the frames (\ref{lf.}) of $\mathbb{C}TM$ and (\ref{df.}) of $\mathbb{C}T^*M$ over $U$ with respect to a Hermitian metric $\left<\cdot,\cdot\right>$ on $\mathbb{C}TM$. We fix a smooth local orthonormal frame $ e_{_U}:=(e_{_U}^1,\cdots,e_{_U}^r)$ of $E$ with respect to $h_{_E}$ and a smooth local frame $\sigma_{_U}$ of $L$. For 
\begin{align*}
	g=g_{_U}\otimes\sigma_{_U}\otimes  e_{_U}=\sum_{a=1}^{r}\sum_{|J|=q}g_{J}^a\Theta_{_U}\wedge\omega^J\otimes \sigma_{_U}\otimes e_{_U}^a\in C^\infty(U,\Lambda_\MV^{m,q}(L\otimes E)),
\end{align*}
here and hereafter, sums over multi-indices range over all strictly increasing multi-indices in $\{1,\cdots,n\}$ and $\Theta_{_U}:=\theta^1\wedge\cdots\wedge\theta^m$, we know from (\ref{lefaf}) that
\begin{align}\label{ed}
	\dmv g&=(\rd-\Xi_\MV-\vartheta\wedge)g_{_U}\otimes\sigma_{_U}\otimes e_{_U}+\cdots\nonumber\\
	&=\sum_{a=1}^{r}\sum_{|J|=q}\sum_{j=1}^n(-1)^m(X_jg_{J}^a)\Theta_{_U}\wedge\omega^j\wedge\omega^J\otimes\sigma_{_U}\otimes e_{_U}^a+\cdots,
\end{align}
where the dots indicate terms of order zero. Let $ \delta_{\MV,\vartheta} $ be the formal adjoint operator of $ \dmv $ with respect to the $L^2$-inner product $(\cdot,\cdot)$, we have
\begin{align}\label{ea}
	\delta_{\MV,\vartheta} g=\sum_{a=1}^{r}\sum_{|K|=q-1}\sum_{j=1}^n(-1)^m\big((X_j)_{\phi_{_U}}^*g_{jK}^a\big)\Theta_{_U}\wedge\omega^K\otimes \sigma_{_U}\otimes e_{_U}^a+\cdots,
\end{align}
where $\phi_{_U}$ is given by $e^{-\phi_{_U}}=h_{_L}(\sigma_{_U},\sigma_{_U})$,
\begin{align*}
	(X_j)_{\phi_{_U}}^*=-\bar{X}_j+\bar{X}_j\phi_{_U},
\end{align*}
and the dots indicate terms in which no $g_{jK}^a$ is differentiated and do not involve $\phi_{_U}$.

\begin{prop}
	Let $\{X_j\}_{j=1}^n$ and $e_U$ be the local orthonormal frames fixed above, and let $e_{jk}^l$ be as in $(\ref{cbl})$ on $U$. Let $0\le q\le n$, for $g\in C_c^\infty(U,\Lambda_\MV^{m,q}(L\otimes E))$, we have
	\begin{align}\label{be.}
		\|\dmv g\|^2+\|\delta_{\MV,\vartheta} g\|^2=&{\rm Re}\sum_{a=1}^{r}\sum_{|K|=q-1}\sum_{j,k=1}^{n}\int_{U}\bigg<\bigg(X_j\bar{X}_k\phi_{_U}+\sum_{l=1}^ne_{jk}^l\bar{X}_l\phi_{_U}\bigg)g_{kK}^a,g_{jK}^a\bigg>e^{-\phi_{_U}}\nonumber\\
		&+\sum_{a=1}^{r}\sum_{|J|=q}\sum_{j=1}^{n}\int_{U}|X_jg_{J}^a|^2e^{-\phi_{_U}}+O_{_U}(G_{_U}(g)\|g\|),
	\end{align}
	where $G_{_U}(g)^2:=\sum_{a=1}^{r}\sum_{|J|=q}\sum_{j=1}^{n}\int_{U}|X_jg_{J}^a|^2e^{-\phi_{_U}}+\|g\|^2$ and $O_{_U}(G_{_U}(g)\|g\|)$ denotes a term whose modulus is bounded by $G_{_U}(g)\|C_{_U}g\|$, in which $C_{_U}\in C^0(U)$ is independent of $g,\phi_{_U}$.
\end{prop}

\begin{proof}
	Consider $g=\sum_{a=1}^{r}\sum_{|J|=q}g_{J}^a\omega^J\wedge\Theta_{_U}\otimes \sigma_{_U}\otimes e_{_U}^a\in  C^\infty_c(U,\Lambda_\MV^{m,q}(L\otimes E))$, by (\ref{ed}) we have
	\begin{align}\label{ed1}
		(\dmv g,\dmv g)=&\sum_{a=1}^{r}\sum_{|I|=q,|J|=q}\sum_{i,j=1}^n\int_U{\rm sgn}\binom{iI}{jJ}\left<X_ig_{I}^a,X_jg_{J}^a\right>e^{-\phi_{_U}}+O_{_U}(G_{_U}(g)\|g\|)\nonumber\\
		=&\sum_{a=1}^{r}\sum_{|J|=q}\sum_{j=1}^{n}\int_{U}|X_jg_{J}^a|^2e^{-\phi_{_U}}-\sum_{a=1}^{r}\sum_{|K|=q-1}\sum_{j,k=1}^{n}\int_{U}\left<X_j g_{kK}^a,X_k g_{jK}^a\right>e^{-\phi_{_U}}\nonumber\\
		&+O_{_U}(G_{_U}(g)\|g\|),
	\end{align}
	where ${\rm sgn}\binom{iI}{jJ}=0$ unless $i\notin I, j\notin J$ and $\{i\}\cup I=\{j\}\cup J$, in which case ${\rm sgn}\binom{iI}{jJ}$ is the sign of the
	permutation $\binom{iI}{jJ}$. On the other hand, due to (\ref{ea})
	\begin{align}\label{ed2}
		(\delta_{\MV,\vartheta} g,\delta_{\MV,\vartheta} g)=\sum_{a=1}^{r}\sum_{|K|=q-1}\sum_{j,k=1}^{n}\int_{U}\left<(X_k)_{\phi_{_U}}^*g_{kK}^a,(X_j)_{\phi_{_U}}^*g_{jK}^a\right>e^{-\phi_{_U}}+O_{_U}(G_{_U}(g)\|g\|).
	\end{align}
	For the first term on the right-hand side of (\ref{ed2}),
	\begin{align}\label{ed3}
		&\int_{U}\left<(X_k)_{\phi_{_U}}^*g_{kK}^a,(X_j)_{\phi_{_U}}^*g_{jK}^a\right>e^{-\phi_{_U}}\nonumber\\
		=&\int_{U}\left<X_j(X_k)_{\phi_{_U}}^*g_{kK}^a,g_{jK}^a\right>e^{-\phi_{_U}}+O_{_U}(G_{_U}(g)\|g\|)\nonumber\\
		=&\int_{U}\left<X_jg_{kK}^a,X_kg_{jK}^a\right>e^{-\phi_{_U}}+\int_{U}\left<\left[X_j,(X_k)_{\phi_{_U}}^*\right]g_{kK}^a,g_{jK}^a\right>e^{-\phi_{_U}}+O_{_U}(G_{_U}(g)\|g\|).
	\end{align}
	
	Combining $(\ref{ed1})\sim(\ref{ed3})$ implies
	\begin{align}\label{ed4}
		\|\dmv g\|^2+\|\delta_{\MV,\vartheta} g\|^2=&\sum_{a=1}^{r}\sum_{|J|=q}\sum_{j=1}^{n}\int_{U}|X_jg_{J}^a|^2e^{-\phi_{_U}}+O_{_U}(G_{_U}(g)\|g\|)\nonumber\\
		&+{\rm Re}\sum_{a=1}^{r}\sum_{|K|=q-1}\sum_{j,k=1}^{n}\int_{U}\left<\left[X_j,(X_k)_{\phi_{_U}}^*\right]g_{kK}^a,g_{jK}^a\right>e^{-\phi_{_U}}.
	\end{align}
	It follows from (\ref{cbl}) that
	\begin{align*}
		\left[X_j,(X_k)_{\phi_{_U}}^*\right]&=-\big[X_j,\bar{X}_k\big]+X_j\bar{X}_k\phi_{_U}\\
		&=-\sum_{l=1}^nd_{jk}^lX_l-\sum_{l=1}^ne_{jk}^l(X_l)_{\phi_{_U}}^*+\sum_{l=1}^ne_{jk}^l\bar{X}_l\phi_{_U}+X_j\bar{X}_k\phi_{_U}.
	\end{align*}
	Thus,
	\begin{align}\label{ed5}
		\int_{U}\left<\left[X_j,(X_k)_{\phi_{_U}}^*\right]g_{kK}^a,g_{jK}^a\right>e^{-\phi_{_U}}=&\int_{U}\bigg<\bigg(X_j\bar{X}_k\phi_{_U}+\sum_{l=1}^ne_{jk}^l\bar{X}_l\phi_{_U}\bigg)g_{kK}^a,g_{jK}^a\bigg>e^{-\phi_{_U}}\nonumber\\
		&+O_{_U}(G_{_U}(g)\|g\|).
	\end{align}
	The desired formula (\ref{be.}) is then obtained by $(\ref{ed4})$ and $(\ref{ed5})$.
\end{proof}

We can exploit the negativity of $Q_{\varphi,\mfp}$ by applying an additional integration by parts (\cite{H3}, \cite{Kjj79}). Recently, Siu employed this approach (page 10 in \cite{Syt22}), in a very different context, to establish a Thullen-type extension theorem for holomorphic vector bundles.
\begin{prop}\label{ibp}
	Given the above orthonormal frame $\{X_1,\cdots,X_n\}$ of $\MV$ and the functions $e_{jk}^l$ in $(\ref{cbl})$ on $U$. Then for any $w\in C_c^\infty(U)$ and $1\leq j\leq n$ we have
	\begin{align*}
		\int_U|X_jw|^2e^{-\phi_{_U}}\geq-\int_U{\rm Re}\bigg(X_j\bar{X}_j\phi_{_U}+\sum_{l=1}^ne_{jj}^l\bar{X}_l\phi_{_U}\bigg)|w|^2e^{-\phi_{_U}}+O_{_U}(G_{_U}(w)\|w\|),
	\end{align*}
	where $\|w\|^2=\int_U|w|^2e^{-\phi_{_U}}$, $G_{_U}(w)^2=\sum_{k=1}^{n}\int_{U}|X_kw|^2e^{-\phi_{_U}}+\|w\|^2$ and $O_{_U}(G_{_U}(w)\|w\|)$ denotes a term whose modulus is bounded by $G_{_U}(w)\|C_{_U}w\|$ for some $C_{_U}\in C^0(U)$ independent of $w$ and $\phi_{_U}$.
\end{prop}

\begin{proof}
	For any function $w\in C_c^\infty(U)$,
	\begin{align*}
		&\int_U|X_jw|^2e^{-\phi_{_U}}\\
		=&\int_U\left<(X_j)_{\phi_{_U}}^*X_jw,w\right>e^{-\phi_{_U}}+O_{_U}(G_{_U}(w)\|w\|)\\
		=&\int_{U}\left<(X_j)_{\phi_{_U}}^*w,(X_j)_{\phi_{_U}}^*w\right>e^{-\phi_{_U}}-\int_{U}\left<\left[X_j,(X_j)_{\phi_{_U}}^*\right]w,w\right>e^{-\phi_{_U}}+O_{_U}(G_{_U}(w)\|w\|)\\
		\geq&-\int_{U}{\rm Re}\left<\left[X_j,(X_j)_{\phi_{_U}}^*\right]w,w\right>e^{-\phi_{_U}}+O_{_U}(G_{_U}(w)\|w\|).
	\end{align*}
	Again by
	\begin{align*}
		\left[X_j,(X_j)_{\phi_{_U}}^*\right]=-\sum_{l=1}^nd_{jj}^lX_l-\sum_{l=1}^ne_{jj}^l(X_l)_{\phi_{_U}}^*+\sum_{l=1}^ne_{jj}^l\bar{X}_l\phi_{_U}+X_j\bar{X}_j\phi_{_U},
	\end{align*}
	we derive the desired estimate.
\end{proof}

\section{Global existence and regularity of the differential complex}
Let $\MV$ be a Levi flat structure of rank $n$ and corank $m$ over $M$. This section is devoted to establishing global existence theorems for the complex (\ref{complex1}) on non-compact and compact manifolds, together with a local existence result. We also derive a Sobolev regularity result for the canonical solution of the equation $\dmv u=f$ on compact manifolds.

\subsection{Global existence theorems}\label{existe.}
We begin with the following lemma, which is originally due to A. Andreotti and E. Vesentini in \cite{AV65} (see also \cite{Djp12C}).

\begin{lemma}\label{mm}
	$(i)$ Let $\varphi\in C^\infty(M)$ be a $q$-convex function with respect to $\MV$. There exists a Hermitian metric $\omega$ on $\mathbb{C}TM$ satisfying $(\ref{cme})$ such that for every $q'\geq q$, the sum of any $q'$ eigenvalues of the quadratic form $Q_{\varphi,{\mfp}}$ with respect to $\omega$ at ${\lmfp}\in\mathcal{C}_\varphi$ is positive, where
	\begin{align}\label{cpe}
		\mathcal{C}_\varphi:=\{{\lmfp}\in M\ |\ X_{\mfp}(\varphi)=0,\ \forall X_{\mfp}\in\MV_{\mfp}\}.
	\end{align}
	
	$(ii)$ Let $(L,h_{_L})$ be a $q$-positive basic line bundle with respect to $\MV$, then there exists a Hermitian metric $\tilde{\omega}$ on $\mathbb{C}TM$ satisfying $(\ref{cme})$ such that for every $q'\geq q$, the sum of any $q'$ eigenvalues of the quadratic form $Q_{h_L,{\mfp}}$ with respect to $\tilde{\omega}$ at ${\lmfp}\in\mathcal{K}_{h_L}$ is positive.
\end{lemma}
\begin{proof}
	The proof for the case of the $q$-positive basic line bundle $(L,h_{_L})$ is analogous to that of the $q$-convex function $\varphi$, provided we observe that the $q$-positivity of $(L,h_{_L})$ imposes the eigenvalue conditions of $Q_{h_L,{\mfp}}$ on the whole space $\MV_{\mfp}$ at ${\lmfp}\in\mathcal{K}_{h_L}$. Therefore, we will focus exclusively on the first case. 
	
	We first prove that for every $\delta>0$, there exists a Hermitian metric $\hat{\omega}$ on $\mathbb{C}TM$ such that the eigenvalues $\hat{\lambda}_1,\cdots,\hat{\lambda}_{n}$ of the quadratic form $Q_{\varphi,{\mfp}}$ with respect to $\hat{\omega}$ at ${\lmfp}\in\mathcal{C}_\varphi$ satisfy $-\delta\leq\hat{\lambda}_1,\cdots,\hat{\lambda}_{q-1}\leq1$ and $\hat{\lambda}_q=\cdots=\hat{\lambda}_{n}=1$.
	
	Given a Hermitian metric $\omega^0$ on $\mathbb{C}TM$, then $\mathbb{C}TM$ splits orthogonally as $\mathbb{C}TM=\MV\oplus\MV ^\perp$, which gives a direct sum
	\begin{align*}
		\omega^0=\omega_\MV^0\oplus \omega_{\MV^\perp}^0.
	\end{align*}
	Let $A_{\MV,0}\in C^0({\rm End}(\MV|_{_{\mathcal{K}_\varphi}}))$ be the Hermitian endomorphism associated to $Q_{\varphi,{\mfp}}$ defined by
	\begin{align*}
		Q_{\varphi,{\mfp}}(X,Y)=\omega_\MV ^0(A_{\MV,0}X,Y)\ \text{for any}\ X,Y\in \MV_{\mfp}\ \text{at}\ {\lmfp}\in\mathcal{K}_\varphi,
	\end{align*}
	and $\lambda_1^0\leq\cdots\leq\lambda_{n}^0$ the eigenvalues of $A_{\MV,0}$. Since $\varphi$ is $q$-convex with respect to $\MV$ and 
	\begin{align*}
		\MV_{\mfp}\cap{\rm Ker}({\rm d}\varphi)=\MV_{\mfp}\ \text{on}\ \mathcal{C}_\varphi\subseteq\mathcal{K}_\varphi,
	\end{align*}
	we know that $\lambda_{q}^0({\lmfp})>0$ at ${\lmfp}\in\mathcal{C}_\varphi$. The fact that $\mathcal{C}_\varphi$ is closed on $M$ enables us to obtain a real-valued function $\tilde{\lambda}_q^0\in C^0(M)$ with $\tilde{\lambda}_q^0=\lambda_{q}^0$ on $\mathcal{C}_\varphi$, then there exists a neighborhood $\Omega$ of $\mathcal{C}_\varphi$ such that $\tilde{\lambda}_q^0>0$ on $\Omega$. By means of the Whitney approximation theorem, one can find a function $0<\eta'\in C^\infty(\Omega)$ such that $\eta'\leq\tilde{\lambda}_q^0$ on $\Omega$. Let $\{\psi,1-\psi\}$ be a smooth partition of unity subordinate to the covering $\{\Omega,M\setminus\mathcal{C}_\varphi\}$, the function $0<\eta:=\psi\eta'+1-\psi\in C^\infty(M)$ satisfies
	\begin{align*}
		\eta\leq\lambda_q^0\ \text{on}\ \mathcal{C}_\varphi.
	\end{align*}
	
	For any small open subset $V\subseteq M$, consider a smooth frame $\{X_1,\cdots,X_n\}$ of $\MV$ over $V$ and functions $e_{jk}^l\in C^\infty(V)$ in (\ref{cbl}), we define
	\begin{align*}
		Q_{\varphi,{\mfp}}^V(\xi,\xi):={\rm Re}\sum_{j,k=1}^{n}\bigg(X_j\bar{X}_k\varphi({\lmfp})+\sum_{l=1}^{n}e_{jk}^l({\lmfp})\bar{X}_l\varphi({\lmfp})\bigg)\xi_{j}\bar{\xi}_{k},
	\end{align*}
	where $\xi:=\sum_{j=1}^n\xi_jX_j\in\MV_{\mfp}$. Let $A_{\MV,0}^V\in C^\infty({\rm End}(\MV|_{_V}))$ be the Hermitian endomorphism associated to $Q_{\varphi,{\mfp}}^V$ with respect to $\omega_\MV ^0$, and $A_{\MV,0}^V$ has eigenvalues $\lambda_1^V\leq\cdots\leq\lambda_{n}^V$. The well-definedness of $Q_{\varphi,\mfp}$ yields that
	\begin{align*}
		A_{\MV,0}^V=A_{\MV,0}\ \text{on}\ V\cap\mathcal{C}_\varphi,
	\end{align*}
	in particular,
	\begin{align*}
		\lambda_j^V({\lmfp})=\lambda_j^0({\lmfp})\ \text{at}\ {\lmfp}\in V\cap\mathcal{C}_\varphi\ \text{for}\ 1\leq j\leq n.
	\end{align*}
	We then choose a positive function $\theta\in C^\infty(\mathbb{R})$ such that
	\begin{align*}
		\theta(t)\geq|t|/\delta\ \text{for}\ t\leq0;\quad \theta(t)\geq t\ \text{for}\ t\geq0;\quad \theta(t)=t\ \text{for}\ t\geq1.
	\end{align*}
	Let $\omega_\MV ^{_V}$ be the Hermitian metric on $\MV |_{_V}$ defined by the following Hermitian endomorphism
	\begin{align*}
		A_\MV ^V({\lmfp}):=\eta({\lmfp})\theta\big[(\eta({\lmfp}))^{-1}A_{\MV,0}^V({\lmfp})\big]\in C^\infty({\rm End}(\MV|_{_V})).
	\end{align*}
	Thus, the eigenvalues of $A_\MV ^V({\lmfp})$ are $\gamma_{j}^V({\lmfp})=\eta({\lmfp})\theta(\lambda_j^V({\lmfp})/\eta({\lmfp}))>0$ ($1\leq j\leq n$). Moreover, for ${\lmfp}\in V\cap\mathcal{C}_\varphi$ we have
	\begin{align}\label{mep}
		\left\{
		\begin{aligned}	
			\gamma_{j}^V({\lmfp})&\geq|\lambda_j^V({\lmfp})|/\delta=|\lambda_j^0({\lmfp})|/\delta,&&\text{if}\ \lambda_j^0({\lmfp})\leq0,\\
			\gamma_{j}^V({\lmfp})&\geq\lambda_j^V({\lmfp})=\lambda_j^0({\lmfp}),&&\text{if}\  \lambda_j^0({\lmfp})\geq0,\\
			\gamma_{j}^V({\lmfp})&=\lambda_j^V({\lmfp})=\lambda_j^0({\lmfp}),&&\text{if}\ j\geq q.
		\end{aligned}
		\right.
	\end{align}
	
	Let $\{V_\alpha\}_\alpha$ be a locally finite open covering of $M$ where $Q_{\varphi,{\mfp}}^{V_\alpha}$ can be defined on each $V_\alpha$, and $\{\psi_\alpha\}_\alpha$ a smooth partition of unity subordinate to $\{V_\alpha\}_\alpha$. Define
	\begin{align*}
		\omega_\MV :=\sum_{\alpha}\psi_\alpha \omega_\MV ^{_{V_\alpha}}.
	\end{align*}
	Since $A_{\MV,0}^{V_\alpha}=A_{\MV,0}$ on $V_\alpha\cap\mathcal{C}_\varphi$ for any $\alpha$, we obtain that $A_\MV ^{V_\alpha}({\lmfp})=A_\MV ^{V_\beta}({\lmfp})$ at ${\lmfp}\in V_\alpha\cap V_\beta\cap\mathcal{C}_\varphi$, from which
	\begin{align*}
		\gamma_j^{V_\alpha}({\lmfp})=\gamma_j^{V_\beta}({\lmfp})=:\gamma_j^{0}({\lmfp})\ \text{at}\ {\lmfp}\in V_\alpha\cap V_\beta\cap\mathcal{C}_\varphi\ \text{for}\ 1\leq j\leq n
	\end{align*}
	follows naturally. By construction, the eigenvalues of the quadratic form $Q_{\varphi,{\mfp}}$ at ${\lmfp}\in\mathcal{C}_\varphi$ with respect to $\omega_\MV $ are
	\begin{align*}
		\hat{\lambda}_j({\lmfp}):=\frac{\lambda_j^0({\lmfp})}{\sum_{\alpha}\psi_\alpha({\lmfp})\gamma_j^{V_\alpha}({\lmfp})}=\frac{\lambda_j^0({\lmfp})}{\gamma_j^{0}({\lmfp})},\ 1\leq j\leq n,
	\end{align*}
	and they have the required properties by (\ref{mep}). Hence, the metric $\hat{\omega}:=\omega_\MV \oplus\omega_{\MV^\perp}^0$ is the desired metric.
	
	Choose $\delta\leq1/q$ and set
	\begin{align*}
		\omega:=e^\rho\hat{\omega},
	\end{align*}
	where $\rho\in C^\infty(M)$ is an exhaustion function such that the metric $\omega$ satisfies (\ref{cme}). The eigenvalues $\lambda_j$ of $Q_{\varphi,{\mfp}}$ with respect to $\omega$ are then $e^{-\rho}\hat{\lambda}_j$ for $1\leq j\leq n$. For any ${\lmfp}\in\mathcal{C}_\varphi$,
	\begin{align*}
		\lambda_1+\cdots+\lambda_{q'}=e^{-\rho}(\hat{\lambda}_1+\cdots+\hat{\lambda}_{q-1}+\hat{\lambda}_q+\cdots+\hat{\lambda}_{q'})\geq e^{-\rho}(1-(q-1)\delta)>0.
	\end{align*}
	The proof is thus complete.
\end{proof}

In what follows, we will work with the metrics given by Lemma \ref{mm}. To prove Theorem \ref{t1} and  Theorem \ref{ct}, we also need the following lemma.
\begin{lemma}\label{evl}
	$(i)$ Let $\varphi\in C^\infty(M)$ be a $q$-convex function with respect to $\MV$. For each ${\lmfp}_0\in(M\setminus\mathcal{K}_\varphi)\cup\mathcal{C}_\varphi$, there is an open neighborhood $U$ of $\lmfp_0$ with the following property: given a smooth frame $\{X_j\}_{j=1}^n$ of $\MV$ over $U$, the coefficients $e_{jk}^l$ in $(\ref{cbl})$ can be appropriately chosen such that for all $q'\geq q$, the sum of any $q'$ eigenvalues of the quadratic form
	\begin{align}\label{Qu}
		Q_{\varphi,{\mfp}}^U(\xi,\xi):={\rm Re}\sum_{j,k=1}^{n}\bigg(X_j\bar{X}_k\varphi({\lmfp})+\sum_{l=1}^{n}e_{jk}^l({\lmfp})\bar{X}_l\varphi({\lmfp})\bigg)\xi_{j}\bar{\xi}_{k}
	\end{align}
	with respect to $\omega$ is positive for $\xi\in\MV_{\mfp}$ at every ${\lmfp}\in U$, where $\mathcal{C}_\varphi$ is defined by $(\ref{cpe})$.
	
	$(ii)$ Let $(L,h_{_L})$ be a $q$-positive basic line bundle with respect to $\MV$, where $h_{_L}$ is given by $\{(U_\alpha,\phi_\alpha)\}_\alpha$ satisfying $(\ref{lm.})$. For each $U_\alpha$, fix a smooth frame $\{X_j\}_{j=1}^n$ of $\MV$. The sum of any $q'$ $(\geq q)$ eigenvalues of the quadratic form $Q_{h_L,{\mfp}}^{\alpha}$ $($obtained from $(\ref{Qu})$ by setting $U=U_\alpha, \varphi=\phi_\alpha$$)$ with respect to $\tilde{\omega}$ is positive at every ${\lmfp}\in U_\alpha$ for suitable coefficients $e_{jk}^l$ in $(\ref{cbl})$.
\end{lemma}
\begin{proof}
	(i) According to (i) in Lemma \ref{mm}, the sum of any $q'$ eigenvalues of the quadratic form $Q_{\varphi,{\mfp}}$ with respect to $\omega$ is positive at ${\lmfp}\in\mathcal{C}_\varphi$. Hence, for any $\lmfp_0\in\mathcal{C}_\varphi$, there exists an open neighborhood $U$ of ${\lmfp}_0$ with a smooth frame $\{X_j\}_{j=1}^n$ so that, $Q_{\varphi,{\mfp}}^U$ satisfies the above positivity condition at $\lmfp\in U$ with respect to $\omega$ for arbitrarily given $e_{jk}^l$'s in (\ref{cbl}).
	
	For each ${\lmfp}_0\in M\setminus\mathcal{K}_\varphi$, let $U\subseteq M\setminus\mathcal{K}_\varphi$ be an open neighborhood of ${\lmfp}_0$. By the constant rank assumption on $\MV\cap\overline{\MV}$, it follows that for any ${\lmfp}\in U$, there are a neighborhood $V_{\mfp}$ and a real vector field $X_{({\mfp})}\in C^\infty(V_{\mfp},\MV\cap\bar{\MV})$ such that $X_{({\mfp})}(\varphi)|_{_{V_{\mfp}}}>0$. Then, by using a partition of unity, we obtain a real vector field $X=\sum_{l=1}^ne^lX_l\in C^\infty(U,\MV\cap\bar{\MV})$ fulfilling $X(\varphi)>0$ on $U$. Let $\psi\in  C^\infty(U)$ be a function to be determined. Replace $e_{jk}^l$ in $Q_{\varphi,{\mfp}}^U$ by
	\begin{align*}
		e_{jk}^l+\psi\omega(X_j,X_k)\bar{e}^l.
	\end{align*}
	Accordingly, $d_{jk}^l$ in (\ref{cbl}) is replaced by $d_{jk}^l+\psi\omega(X_j,X_k){e}^l$, and the quadratic form $Q_{\varphi,{\mfp}}^U(\xi,\xi)$ results in
	\begin{align}\label{nq}
		{\rm Re}\sum_{j,k=1}^{n}\bigg(X_j\bar{X}_k\varphi({\lmfp})+\sum_{l=1}^{n}e_{jk}^l({\lmfp})\bar{X}_l\varphi({\lmfp})\bigg)\xi_{j}\bar{\xi}_{k}+\psi({\lmfp}) X(\varphi)({\lmfp})|\xi|_\omega^2.
	\end{align}
	We choose $\psi\in C^\infty(U)$ satisfying
	\begin{align*}
		\psi({\lmfp})>\frac{-\lambda_{\min}({\lmfp})+1}{X(\varphi)({\lmfp})},
	\end{align*}
	where $\lambda_{\min}({\lmfp})$ is the smallest eigenvalue of the quadratic form
	\begin{align*}
		{\rm Re}\sum_{j,k=1}^{n}\bigg(X_j\bar{X}_k\varphi({\lmfp})+\sum_{l=1}^{n}e_{jk}^l({\lmfp})\bar{X}_l\varphi({\lmfp})\bigg)\xi_{j}\bar{\xi}_{k}
	\end{align*}
	with respect to $\omega$ on $U$. The quadratic form ${Q}_{\varphi,{\mfp}}^U(\xi,\xi)$ in (\ref{nq}) fulfills the desired positivity condition on $U$, since 
	\begin{align*}
		{Q}_{\varphi,{\mfp}}^U(\xi,\xi)\geq\big(\lambda_{\min}({\lmfp})+\psi({\lmfp}) X(\varphi)({\lmfp})\big)|\xi|_\omega^2>|\xi|_\omega^2.
	\end{align*}
	
	(ii) By (ii) in Lemma \ref{mm}, for every $\lmfp\in \mathcal{K}_{h_L}$, the sum of any $q'$ eigenvalues of $Q_{h_L,\mfp}$ with respect to $\tilde{\omega}$ is positive. Thus, there is an open neighborhood $\Omega\subseteq U_\alpha$ of $\mathcal{K}_{h_L}\cap U_\alpha$ ($\neq\emptyset$ without loss of generality) such that $Q_{h_L,\mfp}^\alpha$ meets the above positivity condition on $\Omega$ with respect to $\tilde{\omega}$ for any given $e_{jk}^l$'s in (\ref{cbl}).
	
	Again by the assumption that $\MV\cap\overline{\MV}$ has constant rank and using a partition of unity, one can find a real vector field $X=\sum_{l=1}^{n}e^lX_l\in C^\infty(U_\alpha,\mathcal{V}\cap\bar{\mathcal{V}})$ fulfilling $X(\phi_\alpha)\geq0$ on $U_\alpha$ and $X(\phi_\alpha)|_{_{U_\alpha\setminus\Omega}}>0$. Let $\psi\in C^\infty(U_\alpha)$ be chosen below. Replacing $e_{jk}^l$ in $Q_{h_L,\mfp}^\alpha$ by $e_{jk}^l+\psi\tilde{\omega}(X_j,X_k)\bar{e}^l$, the quadratic form $Q_{h_L,\mfp}^\alpha(\xi,\xi)$ becomes
	\begin{align*}
		{\rm Re}\sum_{j,k=1}^{n}\bigg(X_j\bar{X}_k\phi_\alpha(\lmfp)+\sum_{l=1}^{n}e_{jk}^l(\lmfp)\bar{X}_l\phi_\alpha(\lmfp)\bigg)\xi_{j}\bar{\xi}_{k}+\psi(\lmfp)X(\phi_\alpha)(\lmfp)|\xi|_{\tilde{\omega}}^2.
	\end{align*}
	Thanks to the Whitney approximation theorem, we can select $0\leq\psi\in C^\infty(U_\alpha)$ such that
	\begin{align*}
		\psi|_{_{U_\alpha\setminus\Omega}}>\frac{-\lambda_{\min}^\alpha+1}{X(\phi_\alpha)}\bigg|_{U_\alpha\setminus\Omega},
	\end{align*}
	where $\lambda_{\min}^\alpha$ is the smallest eigenvalue of 
	\begin{align*}
		{\rm Re}\sum_{j,k=1}^{n}\bigg(X_j\bar{X}_k\phi_\alpha(\lmfp)+\sum_{l=1}^{n}e_{jk}^l(\lmfp)\bar{X}_l\phi_\alpha(\lmfp)\bigg)\xi_{j}\bar{\xi}_{k}
	\end{align*}
	with respect to $\tilde{\omega}$ on $U_\alpha$.
	Hence, the quadratic form ${Q}_{h_L,\mfp}^\alpha(\xi,\xi)$ satisfies the desired positivity condition on $U_\alpha$, since 
	\begin{align*}
		{Q}_{h_L,\mfp}^\alpha(\xi,\xi)\geq
		\left\{
		\begin{aligned}
			&{\rm Re}\sum_{j,k=1}^{n}\bigg(X_j\bar{X}_k\phi_\alpha(\lmfp)+\sum_{l=1}^{n}e_{jk}^l(\lmfp)\bar{X}_l\phi_\alpha(\lmfp)\bigg)\xi_{j}\bar{\xi}_{k}\ \text{on}\ \Omega,\\
			&\big(\lambda_{\min}^\alpha(\lmfp)+\psi(\lmfp) X(\phi_\alpha)(\lmfp)\big)|\xi|_{\tilde{\omega}}^2>|\xi|_{\tilde{\omega}}^2\ \text{on}\ U_\alpha\setminus\Omega.
		\end{aligned}
		\right.
	\end{align*}
	The proof is thus complete.
\end{proof}

From now on, we will adopt those coefficients $e_{jk}^l$ given by Lemma \ref{evl}. We can now proceed to prove the global existence theorem on non-compact manifolds.

\begin{thm}[=Theorem \ref{t1}]\label{t1'}
	Suppose that $M$ admits a $q$-convex exhaustion function $\varphi\in C^\infty(M)$ with respect to a Levi flat structure $\MV$, $\vartheta$ is a smooth $1$-form on $M$ satisfying $(\ref{mnf})$, and $E$ is a basic vector bundle over $(M,\MV)$. Then for any $f\in L^2_{loc}(M,\Lambda_\MV^{m,q'}(E))$ with $ \dmv f=0 $, there exists some $ u\in L^2_{loc}(M,\Lambda_\MV^{m,q'-1}(E)) $ such that $ \dmv u=f $ for $q'\geq q$. 
\end{thm}
\begin{proof}
	We first recall that
	\begin{align*}
		\mathcal{C}_\varphi=\{{\lmfp}\in M\ |\ X_{\mfp}(\varphi)=0,\ \forall X_{\mfp}\in\MV_{\mfp}\}\subseteq\mathcal{K}_\varphi
	\end{align*}
	is a closed subset. We will consider the Bochner formula (\ref{be.}) locally around ${\lmfp}\in \mathcal{K}_\varphi\setminus\mathcal{C}_\varphi$ and ${\lmfp}\in(M\setminus\mathcal{K}_\varphi)\cup\mathcal{C}_\varphi$ separately, then derive the global estimate (\ref{priori estimate}) by using a partition of unity and replacing $\varphi$ by $\chi(\varphi)$ with a suitable convex increasing function $\chi\in C^\infty(\mathbb{R})$.
	
	We begin by considering ${\lmfp}\in\mathcal{K}_\varphi\setminus\mathcal{C}_\varphi$. Let $\pi:\mathbb{C}TM\rightarrow\MV $ be the orthogonal projection and define
	\begin{align*}
		X_{\varphi}:=\pi({\rm grad}\varphi)/|\pi({\rm grad}\varphi)|\in C^\infty(M\setminus  \mathcal{C}_\varphi,\MV),
	\end{align*}
	whose dual $1$-form is denoted by $\omega^\varphi$. Here ${\rm grad}\varphi$ denotes the Hermitian gradient with respect to $\omega$. Given ${\lmfp}_0\in\mathcal{K}_\varphi\setminus\mathcal{C}_\varphi$, we choose an open neighborhood $V\subset\subset M\setminus \mathcal{C}_\varphi$ of ${\lmfp}_0$ and extend $X_\varphi|_{_V}$ to an orthonormal frame of $\MV$
	\begin{align*}
		\left\{X_1,\cdots,X_{n-1},X_n:=X_{\varphi}|_{_V}\right\}
	\end{align*}
	over $V$ with the dual frame
	\begin{align*}
		\left\{\omega^1,\cdots,\omega^{n-1},\omega^n:=\omega^{\varphi}|_{_V}\right\}.
	\end{align*}
	This leads to $X_1,\cdots, X_{n-1}\in C^\infty(V,\MV)\cap{\rm Ker}(\rd\varphi)$. By a unitary transformation of $X_j$ (with $\omega^j$ transforming accordingly) for $j<n$ (retaining the same notations for simplicity), we can achieve that for $\xi=\sum_{j=1}^{n-1}\xi_j X_j|_{\mfp_{_0}}\in\MV_{\mfp_{_0}}\cap{\rm Ker(\rd\varphi)}$
	\begin{align*}
		Q_{\varphi,\mfp_{_0}}(\xi,\xi)={\rm Re}\sum_{j,k=1}^{n-1}\bigg(X_j\bar{X}_k\varphi({\lmfp}_0)+\sum_{l=1}^{n}e_{jk}^l({\lmfp}_0)\bar{X}_l\varphi({\lmfp}_0)\bigg)\xi_{j}\bar{\xi}_{k}=\sum_{j=1}^{n-1}\mu_j({\lmfp}_0)|\xi_{j}|^2,
	\end{align*}
	where $\mu_j({\lmfp}_0)<0$ for $j\leq\kappa$ and $\mu_j({\lmfp}_0)\geq0$ for $j>\kappa$. For $0<\epsilon<\frac{1}{2}$, from the Bochner formula (\ref{be.}) in terms of $g=g_{_V}\otimes e_{_V}=\sum_{a=1}^{r}\sum_{|J|=q'}g_{J}^a\Theta_{_V}\wedge\omega^J\otimes e_{_V}^a\in C_c^\infty(V,\Lambda_\MV^{m,q'}(E))$ and the metric $e^{-\varphi}h_{_E}$ we read off
	\begin{align}\label{vt1}
		&\|\dmv g\|_\varphi^2+\|\delta_{\MV,\vartheta}^\varphi g\|_\varphi^2\nonumber\\
		\geq&(1-\epsilon)\sum_{a=1}^{r}\sum_{|J|=q'}\sum_{j=1}^n\int_{V}|X_jg_{J}^a|^2e^{-\varphi}+\int_{V}Q_{\varphi}^V(g_{_V},g_{_V})e^{-\varphi}+\int_{V}C_{_V,\epsilon}|g_{_V}|^2e^{-\varphi},
	\end{align}
	here and hereafter $C_{_V,\epsilon}$ denotes various continuous functions which are independent of $g$,
	\begin{align}\label{QU}
		Q_{\varphi}^V(g_{_V},g_{_V})=&\sum_{a=1}^{r}\sum_{|K|=q'-1}\sum_{j,k=1}^{n}\varphi_{j\bar{k}}^{_V}g_{kK}^a\bar{g}_{jK}^a\nonumber\\
		:=&{\rm Re}\sum_{a=1}^{r}\sum_{|K|=q'-1}\sum_{j,k=1}^{n}\bigg(X_j\bar{X}_k\varphi+\sum_{l=1}^ne_{jk}^l\bar{X}_l\varphi\bigg)g_{kK}^a\bar{g}_{jK}^a.
	\end{align}
	Applying Proposition \ref{ibp} to those terms in the first sum in (\ref{vt1}) where $j\leq\kappa$, we have
	\begin{align}\label{vt1'}
		\|\dmv g\|_\varphi^2+\|\delta_{\MV,\vartheta}^\varphi g\|_\varphi^2\geq\int_VP_{\varphi,_V}^\epsilon({\lmfp};g_{_V},g_{_V})e^{-\varphi}+\int_{V}C_{_V,\epsilon}|g_{_V}|^2e^{-\varphi},
	\end{align}
	where
	\begin{align*}
		P_{\varphi,_V}^\epsilon({\lmfp};g_{_V},g_{_V}):=\sum_{a=1}^{r}\sum_{|J|=q'}\sum_{j=1}^{\kappa}-(1-2\epsilon)\varphi_{j\bar{j}}^{_V}|g_{J}^a|^2+\sum_{a=1}^{r}\sum_{|K|=q'-1}\sum_{j,k=1}^{n}\varphi_{j\bar{k}}^{_V}g_{kK}^a\bar{g}_{jK}^a.
	\end{align*}
	Let $\chi\in C^\infty(\mathbb{R})$ be a strictly increasing convex function with $\chi'(t)>0$, then 
	\begin{align*}
		\mathcal{K}_{\chi(\varphi)}=\mathcal{K}_{\varphi},\ \mathcal{C}_{\chi(\varphi)}=\mathcal{C}_{\varphi},\ X_{\chi(\varphi)}=X_{\varphi}
	\end{align*}
	and
	\begin{align}\label{vt3}
		X_j\bar{X}_k\chi(\varphi)+\sum_{l=1}^ne_{jk}^l\bar{X}_l\chi(\varphi)=\chi'(\varphi)\bigg(X_j\bar{X}_k\varphi+\sum_{l=1}^ne_{jk}^l\bar{X}_l\varphi\bigg)+\chi''(\varphi)X_j\varphi\bar{X}_k\varphi.
	\end{align}
	Taking into account
	(\ref{vt1'}) and (\ref{vt3}) gives
	\begin{align}\label{e1}
		\|\dmv g\|_{\chi(\varphi)}^2+\|\delta_{\MV,\vartheta}^{\chi(\varphi)} g\|_{\chi(\varphi)}^2\geq&\int_V\bigg(\chi'(\varphi)P_{\varphi,_V}^\epsilon({\lmfp};g_{_V},g_{_V})+\sum_{a=1}^{r}\sum_{|K|=q'-1}\chi''(\varphi)|X_{\varphi}(\varphi)g_{nK}^a|^2\bigg)e^{-\chi(\varphi)}\nonumber\\
		&+\int_{V}C_{_V,\epsilon}|g_{_V}|^2e^{-\chi(\varphi)}.
	\end{align}
	We write
	\begin{align}\label{g1g2}
		g=g^1+g^2=&(g_{_V}^{1}+g_{_V}^{2})\otimes e_{_V}\nonumber\\
		:=&(X_{\varphi}\lrcorner\omega^\varphi\wedge g_{_V}+\omega^\varphi\wedge X_{\varphi}\lrcorner g_{_V})\otimes e_{_V}\nonumber\\
		=&\sum_{a=1}^{r}\bigg(\sideset{}{'}\sum_{|J|=q'}g_{J}^{a,1}\Theta_{_V}\wedge\omega^J+\sideset{}{''}\sum_{|J|=q'}g_J^{a,2}\Theta_{_V}\wedge\omega^J\bigg)\otimes e_{_V}^a,
	\end{align}
	where the notation $\sum'$ (or $\sum''$) means that the summation extends only over strictly increasing multi-indices $J$ with $n\notin J$ (respectively $n\in J$). From (\ref{g1g2}) we know that $g_{_V}^1=0$ when $q=n$. For any positive number $\varepsilon$, the item $P_{\varphi,_V}^\epsilon({\lmfp};g_{_V},g_{_V})$ can be bounded from below by the sum of the following two terms:
	\begin{align*}
		P_{\varphi,_V}^{\epsilon,\varepsilon}({\lmfp};g_{_V}^1,g_{_V}^1):=\sum_{a=1}^{r}\bigg(\sideset{}{'}\sum_{|J|=q'}\sum_{j=1}^{\kappa}-(1-2\epsilon)\varphi_{j\bar{j}}^{_V}|g_{J}^{a,1}|^2+\sideset{}{'}\sum_{|K|=q'-1}\sum_{j,k=1}^{n-1}\varphi_{j\bar{k}}^{_V}g_{kK}^{a,1}\bar{g}_{jK}^{a,1}-\varepsilon\sideset{}{'}\sum_{|J|=q'}|g_{J}^{a,1}|^2\bigg)
	\end{align*}
	and
	\begin{align*}
		P_{\varphi,_V}^{\epsilon,\varepsilon}({\lmfp};g_{_V}^2,g_{_V}^2):=&\sum_{a=1}^{r}\sideset{}{''}\sum_{|J|=q'}\sum_{j=1}^{\kappa}-(1-2\epsilon)\varphi_{j\bar{j}}^{_V}|g_{J}^{a,2}|^2+\sum_{a=1}^{r}\sideset{}{''}\sum_{|J|=q'}\varphi_{n\bar{n}}^{_V}|g_{J}^{a,2}|^2\\
		&+\sum_{a=1}^{r}\sideset{}{''}\sum_{|K|=q'-1}\sum_{j,k=1}^{n-1}\varphi_{j\bar{k}}^{_V}g_{kK}^{a,2}\overline{g}_{jK}^{a,2}-C_{\varepsilon}\sum_{a=1}^{r}\sideset{}{''}\sum_{|J|=q'}\sum_{j=1}^{n-1}|\varphi_{j\bar{n}}^{_V}g_{J}^{a,2}|^2,
	\end{align*}
	where $C_{\varepsilon}$ is a positive constant. It follows from the $q$-convexity of $\varphi$ at ${\lmfp}_0\in\mathcal{K}_\varphi\setminus\mathcal{C}_\varphi$ that
	\begin{align*}
		P_{\varphi,_V}^{0,0}({\lmfp}_0;g_{_V}^1,g_{_V}^1)=\sum_{a=1}^{r}\sideset{}{'}\sum_{|J|=q'}\bigg(\sum_{j=1}^{n-1}\mu_{j}^-({\lmfp}_0)+\sum_{j\in J}\mu_j({\lmfp}_0)\bigg)|g_{J}^{a,1}|^2
	\end{align*}
	is positive definite, where $\mu_j^-({\lmfp}_0):=\max\{-\mu_j({\lmfp}_0),0\}$. Thus, there is a positive function $\mu_{_V}^{\epsilon,\varepsilon}\in C^0(V)$ such that for sufficiently small $\epsilon$ and $\varepsilon$ (shrinking $V$ if necessary),
	\begin{align*}
		P_{\varphi,_V}^{\epsilon,\varepsilon}({\lmfp};g_{_V}^1,g_{_V}^1)\geq\mu_{_V}^{\epsilon,\varepsilon}|g_{_V}^1|^2.
	\end{align*}
	Then (\ref{e1}) implies
	\begin{align}\label{vt2}
		\|\dmv g\|_{\chi(\varphi)}^2+\|\delta_{\MV,\vartheta}^{\chi(\varphi)} g\|_{\chi(\varphi)}^2\geq&\int_V\chi'(\varphi)\big(\mu_{_V}^{\epsilon,\varepsilon}|g^1|_{h_\chi}^2+R_{\varphi,_V}^{\epsilon,\varepsilon}|g^2|_{h_\chi}^2\big)\nonumber\\
		&+\int_V\chi''(\varphi)|X_{\varphi}(\varphi)g^2|_{h_\chi}^2+\int_{V}C_{_V,\epsilon}|g|_{h_\chi}^2,
	\end{align}
	where $h_\chi:=e^{-\chi(\varphi)}h_{_E}$, $C_{_V,\epsilon}\in C^0(V)$ is independent of $g,\chi$, and
	\begin{align*}
		R_{\varphi,_V}^{\epsilon,\varepsilon}:=
		\sum_{j=1}^{\kappa}-(1-2\epsilon)\varphi_{j\bar{j}}^{_V}+\varphi_{n\bar{n}}^{_V}-\sqrt{\sum_{j,k=1}^{n-1}|\varphi_{j\bar{k}}^{_V}|^2}-C_{\varepsilon}\sum_{j=1}^{n-1}|\varphi_{j\bar{n}}^{_V}|^2.
	\end{align*}
	
	For any ${\lmfp}_0\in(M\setminus\mathcal{K}_\varphi)\cup\mathcal{C}_\varphi$, we apply the local Bochner formula (\ref{be.}) to the basic vector bundle $(E,h_\chi)$ near $\lmfp_0$. By Lemma \ref{evl} (i), there exist an open neighborhood $W$ of $\lmfp_0$ and a positive function $\lambda_{_W}\in C^0(W)$ such that for all $g=g_{_W}\otimes \sigma_{_W}\in C^\infty_c(W,\Lambda_\MV^{m,q'}(E))$
	\begin{align}\label{e2}
		\|\dmv g\|_{\chi(\varphi)}^2+\|\delta_{\MV,\vartheta}^{\chi(\varphi)} g\|_{\chi(\varphi)}^2\geq&\int_WQ_{\chi(\varphi)}^W(g_{_W},g_{_W})e^{-\chi(\varphi)}+\int_{W}C_{_W}|g_{_W}|^2e^{-\chi(\varphi)}\nonumber\\
		\geq&\int_W\chi'(\varphi)Q_{\varphi}^W(g_{_W},g_{_W})e^{-\chi(\varphi)}+\int_{W}C_{_W}|g|_{h_\chi}^2\nonumber\\
		\geq&\int_W\chi'(\varphi)\lambda_{_W}|g|_{h_\chi}^2+\int_{W}C_{_W}|g|_{h_\chi}^2,
	\end{align}
	where $C_{_W}\in C^0(W)$ is independent of $g,\chi$, the notations $Q_{\chi(\varphi)}^W(g_{_W},g_{_W})$ and $Q_{\varphi}^W(g_{_W},g_{_W})$ are defined analogously to (\ref{QU}).
	
	Now we choose coordinate charts $\{V_\alpha\}_\alpha$ ($V_\alpha\subset\subset M\setminus \mathcal{C}_\varphi$) and $\{W_\alpha\}_\alpha$ covering $M$ where (\ref{vt2}) and (\ref{e2}) are applicable respectively, and $\{V_\alpha,W_\alpha\}_\alpha$ are locally finite. Choose locally finite families of smooth real-valued cutoff functions $\{\psi_\alpha\}_\alpha$ and $\{\tilde{\psi}_\alpha\}_\alpha$ subordinate to the covering $\{V_\alpha\}_\alpha\cup
	\{W_\alpha\}_\alpha$, such that $\psi_\alpha\in C_c^\infty(V_\alpha)$, $\tilde{\psi}_\alpha\in
	C_c^\infty(W_\alpha)$ and $\sum_\alpha \psi_\alpha^2+\sum_\alpha\tilde{\psi}_\alpha^2=1$ on $M$. For any $g\in  C^\infty_c(M,\Lambda_\MV^{m,q'}(E))$, employing (\ref{vt2}) and (\ref{e2}) to $\psi_\alpha g$ and $\tilde{\psi}_\alpha g$ respectively, and summing over $\alpha$ gives
	\begin{align}\label{e3}
		\|\dmv g\|_{\chi(\varphi)}^2+\|\delta_{\MV,\vartheta}^{\chi(\varphi)} g\|_{\chi(\varphi)}^2\geq&\int_M\bigg(\chi'(\varphi)\mu^{\epsilon,\varepsilon}|g^1|_{h_\chi}^2+\big(\chi'(\varphi)R_{\varphi}^{\epsilon,\varepsilon}|g^2|_{h_\chi}^2+\chi''(\varphi)\Psi|X_{\varphi}(\varphi)|^2|g^2|_{h_\chi}^2\big)\bigg)\nonumber\\
		&+\int_M\chi'(\varphi)\lambda|g|_{h_\chi}^2+\int_{M}C_\epsilon|g|_{h_\chi}^2,
	\end{align}
	where $g|_{_{M\setminus\mathcal{C}_\varphi}}=g^1+g^2$ are given by (\ref{g1g2}), $C_\epsilon\in C^0(M)$ is independent of $g,\chi$, and
	\begin{align*}
		\mu^{\epsilon,\varepsilon}:=\sum_\alpha\mu_{_{V_\alpha}}^{\epsilon,\varepsilon}\psi_\alpha^2,\ R_\varphi^{\epsilon,\varepsilon}:=\sum_\alpha R_{\varphi,_{V_\alpha}}^{\epsilon,\varepsilon}\psi_\alpha^2,\ \Psi:=\sum_\alpha \psi_\alpha^2,\ \lambda:=\sum_\alpha\lambda_{_{W_\alpha}}\tilde{\psi}_\alpha^2.
	\end{align*}
	In view of (\ref{e3}), it remains to show that we can select a convex function $\chi\in C^\infty(\mathbb{R})$ increasing so rapidly that
	\begin{align}\label{e4}
		\left\{
		\begin{aligned}
			&\chi'(\varphi)\mu^{\epsilon,\varepsilon}+C_{\epsilon}\Psi\geq\Psi,\\
			&\chi'(\varphi)R_\varphi^{\epsilon,\varepsilon}+\chi''(\varphi)\Psi|X_{\varphi}(\varphi)|^2+C_{\epsilon}\Psi\geq\Psi,\\
			&\chi'(\varphi)\lambda+C_{\epsilon}(1-\Psi)\geq1-\Psi.
		\end{aligned}
		\right.
	\end{align}
	
	Since $\varphi$ is an exhaustion function, $M_t:=\{\varphi< t\}\subset\subset M$ for any $t\in\mathbb{R}$. Let
	\begin{align*}
		t_0:=\inf_{\{\Psi\neq0\}}\varphi({\lmfp}),
	\end{align*}
	then $\overline{M}_{t_0}\cap{\rm supp}\Psi\neq\emptyset$, $\overline{M}_{t}\cap{\rm supp}\Psi=\emptyset$ if $t<t_0$ and ${M}_{t}\cap\{\Psi\neq0\}\neq\emptyset$ if $t>t_0$. Assuming without loss of generality that $t_0<+\infty$ (i.e., $\Psi\not\equiv0$), we can define the following functions on $t\geq t_0$:
	\begin{align*}
		\mu(t):=\sup\limits_{M_{t+1}\cap\{\Psi\neq0\}}\frac{(1-C_{\epsilon})\Psi}{\mu^{\epsilon,\varepsilon}},\ \ 
		R(t):=\sup\limits_{M_{t+1}\cap\{\Psi\neq0\}}\frac{1-R_\varphi^{\epsilon,\varepsilon}/\Psi}{|X_{\varphi}(\varphi)|^2},
		\ \ 
		C(t):=\max_{M_{t+1}}(1-C_{\epsilon}).
	\end{align*}
	Indeed, for any $t\geq t_0$, there exist only finitely many open sets $V_{\alpha_1},\cdots,V_{\alpha_s}\subset\subset M\setminus \mathcal{C}_\varphi$ which have non-empty intersection with $M_{t+1}$, then for ${\lmfp}\in M_{t+1}\cap\{\Psi\neq0\}$
	\begin{align}\label{Psi1}
		\frac{\Psi({\lmfp})}{\mu^{\epsilon,\varepsilon}({\lmfp})}=\frac{\sum_{\iota=1}^s\psi_{\alpha_\iota}^2({\lmfp})}{\sum_{\iota=1}^s\mu_{_{V_{\alpha_\iota}}}^{\epsilon,\varepsilon}\psi_{\alpha_\iota}^2({\lmfp})}\leq\frac{\sum_{\iota=1}^s\psi_{\alpha_\iota}^2({\lmfp})}{\sum_{\iota=1}^s\mu_{\min}^{\iota}\psi_{\alpha_\iota}^2({\lmfp})}\leq\frac{1}{\min\{\mu_{\min}^1,\cdots,\mu_{\min}^s\}},
	\end{align}
	where $\mu_{\min}^\iota:=\min_{{\rm supp}\psi_{\alpha_\iota}}{\mu}_{_{V_{\alpha_\iota}}}^{\epsilon,\varepsilon}>0$, which yields that $\mu(t)<+\infty$ for $t\geq t_0$.
	
	Similarly, there is a constant $c>0$ such that
	\begin{align}\label{R1}
		\left|\frac{R_\varphi^{\epsilon,\varepsilon}({\lmfp})}{\Psi({\lmfp})}\right|<c,\ {\lmfp}\in M_{t+1}\cap\{\Psi\neq0\}.
	\end{align}
	Since $V_{\alpha_1},\cdots,V_{\alpha_s}\subset\subset M\setminus \mathcal{C}_\varphi$, 
	\begin{align*}
		\inf_{M_{t+1}\cap\{\Psi\neq0\}}|X_{\varphi}(\varphi)|^2\geq\inf_{M_{t+1}\cap(\cup_{\iota=1}^sV_{\alpha_\iota})}|X_{\varphi}(\varphi)|^2>0,
	\end{align*}
	which combined with (\ref{R1}), leads to $R(t)<+\infty$ for any $t\geq t_0$.
	
	By exactly the same reasoning as in (\ref{Psi1}), we can define a function on $t\geq t_1$ as follows
	\begin{align*}
		\lambda(t):=\sup\limits_{M_{t+1}\cap\{1-\Psi\neq0\}}\frac{(1-C_{\epsilon})(1-\Psi)}{\lambda}<+\infty,
	\end{align*}
	where $t_1:=\inf_{\{1-\Psi\neq0\}}\varphi({\lmfp})$.
	
	We then choose $\chi\in C^\infty(\mathbb{R})$ increasing rapidly such that
	\begin{align*}
		\left\{
		\begin{aligned}
			&\chi'(t)\geq\max\{\mu(t),C(t)\},&\ t\geq t_0,\\
			&\chi''(t)/\chi'(t)\geq R(t), &\ t\geq t_0,\\
			&\chi'(t)\geq\lambda(t), &\ t\geq t_1,
		\end{aligned}
		\right.
	\end{align*}
	thereby resulting in (\ref{e4}). In fact, if ${\lmfp}\in\{\Psi\neq0\}$, then ${\lmfp}\in M_{t+1}\cap\{\Psi\neq0\}$ for $t:=\varphi({\lmfp})\geq t_0$. The inequality $\chi'(t)\geq\mu(t)$ for $t\geq t_0$ gives
	\begin{align*}
		\chi'(\varphi({\lmfp}))\geq\sup_{M_{t+1}\cap\{\Psi\neq0\}}\frac{(1-C_{\epsilon})\Psi}{\mu^{\epsilon,\varepsilon}}\geq\frac{(1-C_{\epsilon}({\lmfp}))\Psi({\lmfp})}{\mu^{\epsilon,\varepsilon}({\lmfp})}.
	\end{align*}
	Hence, 
 	\begin{align*}
		\chi'(\varphi({\lmfp}))\mu^{\epsilon,\varepsilon}({\lmfp})\geq(1-C_{\epsilon}({\lmfp}))\Psi({\lmfp}),\ {\lmfp}\in\{\Psi\neq0\}.
	\end{align*}
	This is exactly the first inequality in (\ref{e4}) since $\mu^{\epsilon,\varepsilon}(\mfp)=0$ if $\Psi(\mfp)=0$. The last inequality in (\ref{e4}) holds in the same manner. For the second one, if ${\lmfp}\in\{\Psi\neq0\}$, then ${\lmfp}\in M_{t+1}\cap\{\Psi\neq0\}$ for $t:=\varphi({\lmfp})\geq t_0$. By $\chi''(t)/\chi'(t)\geq R(t)$ for $t\geq t_0$ we have
	\begin{align*}
		\frac{\chi''(\varphi({\lmfp}))}{\chi'(\varphi({\lmfp}))}\geq\sup_{M_{t+1}\cap\{\Psi\neq0\}}\frac{1-R_\varphi^{\epsilon,\varepsilon}/\Psi}{|X_{\varphi}(\varphi)|^2}\geq\frac{1-R_\varphi^{\epsilon,\varepsilon}({\lmfp})/\Psi({\lmfp})}{|X_{\varphi}(\varphi)({\lmfp})|^2},
	\end{align*}
	which implies
	\begin{align}\label{c1}
		\frac{R_\varphi^{\epsilon,\varepsilon}({\lmfp})}{\Psi({\lmfp})}+\frac{\chi''(\varphi({\lmfp}))}{\chi'(\varphi({\lmfp}))}|X_{\varphi}(\varphi)({\lmfp})|^2\geq1,\ {\lmfp}\in\{\Psi\neq0\}.
	\end{align}
	In view of $\chi'(t)\geq C(t)$ for $t\geq t_0$,
	\begin{align}\label{c2}
		\chi'(\varphi({\lmfp}))\geq1-C_{\epsilon}({\lmfp}),\ {\lmfp}\in\{\Psi\neq0\}.
	\end{align}
	Multiplying (\ref{c1}) and (\ref{c2}) yields
	\begin{align*}
		\chi'(\varphi({\lmfp}))R_\varphi^{\epsilon,\varepsilon}({\lmfp})+\chi''(\varphi({\lmfp}))\Psi({\lmfp})|X_{\varphi}(\varphi)({\lmfp})|^2\geq(1-C_{\epsilon}({\lmfp}))\Psi({\lmfp}),\ {\lmfp}\in\{\Psi\neq0\},
	\end{align*}
	which gives the second inequality in (\ref{e4}) since $R_\varphi^{\epsilon,\varepsilon}({\lmfp})=0$ if $\Psi(\lmfp)=0$.
	
	Due to (\ref{e3}) and (\ref{e4}), for any $g\in  C^\infty_c(M,\Lambda_\MV^{m,q'}(E))$ we obtain
	\begin{align*}
		\|g\|_{\chi(\varphi)}^2\leq\|T^*g\|_{\chi(\varphi)}^2+\|Sg\|_{\chi(\varphi)}^2,
	\end{align*}
where $S,T$ are defined by (\ref{Ope.}). We may assume that $f\in L_{\chi(\varphi)}^2(M,\Lambda_\MV^{m,q'}(E))$, by Lemma \ref{dl.} and Lemma \ref{functional analysis}, there exists some $u\in L_{\chi(\varphi)}^2(M,\Lambda_\MV^{m,q'-1}(E))$ such that $\dmv u=f$. 
%
\end{proof}

If $\MV=\mathbb{C}TM$, then a basic vector bundle $E$ over $(M,\MV)$ is a flat vector bundle (see (iv) in Example \ref{exa}). Recently, Deng-Zhang (see \cite{DZ21}) established a fundamental relationship between $L^2$-estimates and the curvature positivity of Riemannian metrics on flat vector bundles. As an application of Theorem \ref{t1'}, we have the following local existence result.
\begin{cor}\label{local existence}
	Let $\MV$ be a Levi flat structure on a smooth manifold $M$, $E$ a basic vector bundle over $(M,\MV)$, and let $\vartheta$ be a global smooth $1$-form satisfying $(\ref{mnf})$. Given a point ${\lmfp}\in M$, there is a neighborhood $U$ around ${\lmfp}$ such that, for any $f\in L_{loc}^2(U,\Lambda_\MV^{m,q}(E))$ with $ \dmv f=0 $, there exists some $u\in L_{loc}^2(U,\Lambda_\MV^{m,q-1}(E))$ such that $ \dmv u=f $ for all $1\leq q\leq n$.
\end{cor}
\begin{proof}
	For any point ${\lmfp}\in M$, there is a local coordinate chart $(x,V)$ with $x({\lmfp})=0$. Define
	\begin{align*}
		\phi({\lmfp}'):=|x({\lmfp'})|^2\in C^\infty(V),
	\end{align*}
	which implies that ${\lmfp}\in\mathcal{K}_\phi$. Given the coefficients $e_{jk}^l$ in (\ref{cbl}) and a smooth frame $\{X_1,\cdots,X_n\}$ of $\MV$ over $V$, where $X_j:=\sum_{\nu=1}^{m+n}a_j^\nu\partial_{x_\nu}$ $(1\leq j\leq n)$, then the linear independence of $X_1,\cdots,X_n$ at ${\lmfp}$ implies that for every $0\neq\xi\in\MV_\mfp$,
	\begin{equation*}
		Q_{\phi,\mathfrak p}^V(\xi,\xi):={\rm Re}\sum_{j,k=1}^n\bigg(X_j\bar{X}_k\phi({\lmfp})+\sum_{l=1}^ne_{jk}^l({\lmfp})\bar{X}_l\phi({\lmfp})\bigg)\xi_{j}\bar{\xi}_{k}=2\sum_{\nu=1}^{m+n}\left|\sum_{j=1}^na_j^\nu({\lmfp})\xi_j\right|^2>0.
	\end{equation*}
	Consequently, by continuity, the form $Q_{\phi,\mathfrak p'}^V$ remains positive
	definite on an open ball $U\subseteq V$ centered at ${\lmfp}$ with radius $\epsilon>0$. By (\ref{vt3}) we know that
	\begin{align*}
		\varphi({\lmfp}'):=-\log(\epsilon^2-|x({\lmfp}')|^2)
	\end{align*}
	is a smooth exhaustion function on $U$ satisfying
	\begin{align*}
		{\rm Re}\sum_{j,k=1}^n\bigg(X_j\bar{X}_k\varphi({\lmfp}')+\sum_{l=1}^ne_{jk}^l({\lmfp}')\bar{X}_l\varphi({\lmfp}')\bigg)\xi_{j}\bar{\xi}_{k}>0,\ {\lmfp}'\in U.
	\end{align*}
	Hence, any point ${\lmfp}\in M$ possesses an open neighborhood with a smooth $1$-convex exhaustion function with respect to $\MV$. The proof is thus completed by Theorem \ref{t1'}.
\end{proof}

Clearly, the following conclusion is an immediate application of Corollary \ref{local existence}.
\begin{cor}\label{ros}
	Let $\MV$ be a Levi flat structure on a smooth manifold $M$, $E$ a basic vector bundle over $(M,\MV)$, and let $\vartheta$ be a global smooth $1$-form satisfying $(\ref{mnf})$. The complex $(\ref{complex1})$ is exact, and replacing $E$ with $K_\MV^{-1}\otimes E$ in $(\ref{complex1})$ yields a fine resolution of $\mathcal{S}_\MV(E)$ when $\vartheta\equiv0$, and for any $0\leq q\leq n$
	\begin{align*}
		H^q(M,\mathcal{S}_\MV (E))\cong\frac{{\rm{Ker}}\left(L_{loc}^2\big(M,\Lambda_\MV^{m,q}(K_\MV^{-1}\otimes E)\big)\stackrel{\dV}{\longrightarrow}L_{loc}^2\big(M,\Lambda_\MV^{m,q+1}(K_\MV^{-1}\otimes E)\big)\right)}{{\rm{Im}}\left(L_{loc}^2\big(M,\Lambda_\MV^{m,q-1}(K_\MV^{-1}\otimes E)\big)\stackrel{\dV}{\longrightarrow}L_{loc}^2\big(M,\Lambda_\MV^{m,q}(K_\MV^{-1}\otimes E)\big)\right)}.
	\end{align*}
\end{cor}

\begin{remark}
	When $E$ is a trivial line bundle and $\vartheta\equiv0$, the conclusion of Corollary \ref{ros} was established in \cite{JYY22} under the assumption that $\MV$ is a complex Frobenius structure.
\end{remark}

\begin{thm}[=Theorem \ref{ct} for $s=0$]\label{ct'}
	Let $\MV$ be a Levi flat structure on a compact manifold $M$, $E$ a basic vector bundle over $(M,\MV)$, and $\vartheta$ a smooth $1$-form on $M$ fulfilling $(\ref{mnf})$. Assume that there exists a $q$-positive basic line bundle $(L,h_{_L})$ with respect to $\MV$. Then there exists a positive integer $\tau_0$ such that, for any $\tau\geq\tau_0$, $q'\geq q$ and $f\in L^2\big(M,\Lambda_\MV^{m,q'}(L^{\tau}\otimes E)\big)$ with $ \dmv f=0 $, there exists some $u\in L^2\big(M,\Lambda_\MV^{m,q'-1}(L^{\tau}\otimes E)\big)$ satisfying $ \dmv u=f $ and $\|u\|\leq\|f\|$.
\end{thm}

\begin{proof}
	Applying the local Bochner formula (\ref{be.}) to $(L\otimes E,h_{_L}h_{_E})$, with $h_{_L}$ determined by $\{(U_\alpha,\phi_\alpha)\}_\alpha$ satisfying (\ref{lm.}), and using the $q$-positivity of $(L,h_{_L})$, we obtain from Lemma \ref{evl} (ii) a positive function $\lambda_{\alpha}^h\in C^0(U_\alpha)$ such that, for any $g=g_{\alpha}\otimes\sigma_\alpha\otimes e_{\alpha}\in C^\infty_c\big(U_\alpha,\Lambda_\MV^{m,q'}(L\otimes E)\big)$,
	\begin{align*}
		&\|\dmv g\|^2+\|\delta_{\MV,\vartheta} g\|^2\\
		\geq&\int_{U_\alpha}Q_{h_L}^\alpha(g_{\alpha},g_{\alpha})e^{-\phi_{\alpha}}+\int_{U_\alpha}C_{\alpha}|g_{\alpha}|^2e^{-\phi_{\alpha}}+\frac{1}{2}\sum_{a=1}^{r}\sum_{|J|=q'}\sum_{j=1}^{n}\int_{U_\alpha}|X_j^{\alpha} g_{\alpha,J}^a|^2e^{-\phi_{\alpha}}\nonumber\\
		\geq&\int_{U_\alpha}\lambda_{\alpha}^h|g|^2+\int_{U_\alpha}C_{\alpha}|g|^2+\frac{1}{2}\int_{U_\alpha}|\nabla_\MV^{\alpha}g|^2,
	\end{align*}
	where $C_{\alpha}$ is a continuous function on $U_\alpha$ independent of $g,\phi_\alpha$, the notation $Q_{h_L}^\alpha(g_{\alpha},g_{\alpha})$ is defined analogously to (\ref{QU}), and
	\begin{align*}
		|\nabla_\MV^{\alpha}g|^2:=\sum_{a=1}^{r}\sum_{|J|=q'}\sum_{j=1}^{n}|X_j^{\alpha} g_{\alpha,J}^a|^2e^{-\phi_{\alpha}}.
	\end{align*}
	For $\tau\in\mathbb{Z}_{>0}$, the Hermitian metric $h_{_L}^\tau$ on the basic line bundle $L^\tau$ is given by $\{(U_\alpha,\tau\phi_\alpha)\}_\alpha$, then
	\begin{align*}
		Q_{h_L^\tau}^\alpha(g_{\alpha},g_{\alpha})=\tau Q_{h_L}^\alpha(g_{\alpha},g_{\alpha})\ \text{on}\ U_\alpha.
	\end{align*}
	Thus, for any $g=g_{\alpha}\otimes\sigma_\alpha^\tau\otimes e_{\alpha}\in  C^\infty_c\big(U_\alpha,\Lambda_\MV^{m,q'}(L^{\tau}\otimes E)\big)$,
	\begin{align}\label{ce2}
		\|\dmv g\|^2+\|\delta_{\MV,\vartheta} g\|^2\geq\int_{U_\alpha}\tau\lambda_{\alpha}^h|g|^2+\int_{U_\alpha}C_{\alpha}|g|^2+\frac{1}{2}\int_{U_\alpha}|\nabla_\MV^{\alpha}g|^2.
	\end{align}
	
	Choose a family of smooth real-valued cutoff functions $\{\psi_\alpha\}_\alpha$ with $\psi_\alpha\in C_c^\infty(U_\alpha)$ such that $\sum_\alpha\psi_\alpha^2=1$ on $M$. For any $g\in  C^\infty\big(M,\Lambda_\MV^{m,q'}(L^{\tau}\otimes E)\big)$, applying (\ref{ce2}) to $\psi_\alpha g$ and summing over $\alpha$ yields
	\begin{align}\label{ep}
		\|\dmv g\|^2+\|\delta_{\MV,\vartheta} g\|^2\geq&\int_{M}(\tau\lambda-C)|g|^2+\frac{1}{2}\int_{M}|\nabla_\MV g|^2,
	\end{align}
	where $C$ is a positive constant, $\lambda$ is the minimum of the positive function $\sum_\alpha\psi_\alpha^2\lambda_{\alpha}^h$ on the compact manifold $M$, and
	\begin{align}\label{nabla.}
		 |\nabla_\MV g|^2:=\sum_\alpha|\nabla_\MV^{\alpha}\psi_\alpha g|^2.
	\end{align}
	We define
	\begin{align*}
		\tau_{0}=\max\left\{1,\left\lceil \frac{C+1}{\lambda}\right\rceil\right\},
	\end{align*}
	where $\lceil \cdot \rceil$ denotes the ceiling function. Lemma \ref{dl.} then  implies that
	\begin{align}\label{priori.}
		\|g\|^2\leq\|\dmv^{\tau*}g\|^2+\|\dmv g\|^2,\quad \forall g\in {\rm Dom}(\dmv^{\tau*})\cap{\rm Dom}(\dmv),
	\end{align}
	where $\dmv^{\tau*}$ is the Hilbert adjoint of $\dmv:L^2\big(M,\Lambda_\MV^{m,q'-1}(L^{\tau}\otimes E)\big)\rightarrow L^2(M,\Lambda_\MV^{m,q'}(L^{\tau}\otimes E))$ for $\tau\geq\tau_{0}$. The $L^2$-solvability of the equation $\dmv u=f$ and the estimate of the solution follow from Lemma \ref{functional analysis}.
\end{proof}

A formally integrable structure $\MV$ is called a \emph{CR structure} if $\MV\cap\overline{\MV}=0$. There have been remarkable developments and numerous significant applications in the study of CR structures (see \cite{BHR96}, \cite{FH18}, \cite{HY25},\cite{HZ25} and references therein). If $\MV$ further defines a Levi flat CR structure of corank $(n+1)$ over a smooth manifold $M$ of dimension $(2n+1)$, then $M$ is called a \emph{Levi flat CR manifold}, i.e., foliated by complex submanifolds of real codimension one. Ohsawa-Sibony in \cite{OS00} established $L^2$-existence theorems on compact oriented Levi flat CR manifolds.

Theorem \ref{t1'} and Theorem \ref{ct'}, together with Corollary \ref{ros} yield the vanishing of sheaf cohomologies by letting $\vartheta\equiv0$.
\begin{cor}[=Corollary \ref{local existence1}]
	Let $\MV$ be a Levi flat structure on a smooth manifold $M$, and $E$ a basic vector bundle over $(M,\MV)$. 
	\begin{enumerate}
		\item[$(i)$] If there exists a smooth $q$-convex exhaustion function with respect to $\MV$ on $M$, then $H^{q'}(M,\mathcal{S}_\MV (E))=0$ for $q'\geq q$.
		\item[$(ii)$] If $M$ is compact and $(L,h_{_L})$ is a $q$-positive basic line bundle with respect to $\MV$, then there exists an integer $\tau_0>0$ such that $H^{q'}(M,\mathcal{S}_\MV (L^{\tau}\otimes E))=0$ for $q'\geq q$ and $\tau\geq \tau_0$.
	\end{enumerate}
\end{cor}

By the complexification technique, S. Mongodi and G. Tomassini (see \cite{MT16}) established a vanishing theorem for $H^q(M,\mathcal{O}_\MV)$ for every $q\geq1$ under the following conditions:
\begin{enumerate}
	\item[$(i)$] $\mathcal{V}$ defines a real analytic Levi flat CR structure of corank $m=n+1$, i.e., $M$ is a real analytic manifold foliated by complex submanifolds of real codimension one.
	\item[$(ii)$] There exists a smooth exhaustion function which is strictly plurisubharmonic along the leaves.
	\item[$(iii)$] The transverse bundle $N_{\rm tr}$ to the leaves of $M$ is positive.
\end{enumerate}

\subsection{Sobolev regularity on compact manifolds}
For any $\tau\geq\tau_{0}$ ($\tau_{0}$ is the constant in Theorem \ref{ct'}) and $q'\geq q$, the a priori estimate (\ref{priori.}) implies the self-adjoint operator $\Delta_{\MV,\vartheta}^{\tau}:=\dmv\dmv^{\tau*}+\dmv^{\tau*}\dmv$ is injective and
\begin{align*}
	L^2\big(M,\Lambda_\MV^{m,q'}(L^{\tau}\otimes E)\big)={\rm Im}(\Delta_{\MV,\vartheta}^{\tau}).
\end{align*}
Hence, we can consider the Green operator
\begin{align*}
	N^\tau:L^2\big(M,\Lambda_\MV^{m,q'}(L^{\tau}\otimes E)\big)\rightarrow{\rm Dom}(\Delta_{\MV,\vartheta}^{\tau}),
\end{align*}
which is determined by
\begin{align*}
	N^\tau\Delta_{\MV,\vartheta}^{\tau}={\rm Id}\ \text{on}\ {\rm Dom}(\Delta_{\MV,\vartheta}^{\tau})\ \text{and}\ \Delta_{\MV,\vartheta}^{\tau}N^\tau={\rm Id}\ \text{on}\ L^2\big(M,\Lambda_\MV^{m,q'}(L^{\tau}\otimes E)\big).
\end{align*}
For any $f\in L^2\big(M,\Lambda_\MV^{m,q'}(L^{\tau}\otimes E)\big)$, we have
\begin{align}\label{boxe}
	\Delta_{\MV,\vartheta}^{\tau}N^\tau f=f.
\end{align}
Suppose that $\dmv f=0$, then (\ref{boxe}) reduces to
\begin{align*}
	\dmv\dmv^{\tau*}N^\tau f=f,
\end{align*}
where the solution
\begin{align*}
	u_{\rm can}:=\dmv^{\tau*}N^\tau f
\end{align*}
is called the \emph{canonical solution} of the equation $\dmv u=f$. 

We will apply the $q$-positivity to establish the Sobolev regularity of the canonical solution, via the elliptic regularization technique developed by Kohn-Nirenberg \cite{KN65}; for the case of Levi flat CR structures, see \cite{OS00} and \cite{HM17}. Denote by $W^{s}\big(M,\Lambda_\MV^{m,q'}(L^{\tau}\otimes E)\big)$  the completion of $C^\infty\big(M,\Lambda_\MV^{m,q'}(L^{\tau}\otimes E)\big)$ under the Sobolev norm $\|\cdot\|_s$.

\begin{thm}[=Theorem \ref{ct} for $s\in\mathbb{Z}_{>0}$]\label{ct''}
	Let $\MV$ be a Levi flat structure on a compact manifold $M$, $E$ a basic vector bundle over $(M,\MV)$, and $\vartheta$ a smooth $1$-form on $M$ fulfilling $(\ref{mnf})$. Assume that there exists a $q$-positive basic line bundle $(L,h_{_L})$ with respect to $\MV$. Then for every $s\in\mathbb{Z}_{>0}$, there exists a positive integer $\tau_s$ such that, for any $\tau\geq\tau_s$, $q'\geq q$ and $f\in W^s\big(M,\Lambda_\MV^{m,q'}(L^{\tau}\otimes E)\big)$ with $\dmv f=0$, the canonical solution $u_{\rm can}$ of $\dmv u=f$ satisfies $\|u_{\rm can}\|_s\leq C_{s,\tau}\|f\|_s$, where $C_{s,\tau}>0$ is independent of $f$.
\end{thm}

\begin{proof}
	We first show that for every $s\in\mathbb{Z}_{\geq0}$, there exists a positive integer $\tau_s$ such that the following Sobolev estimate
	\begin{align}\label{sob.}
		\|\dmv v\|_s^2+\|\dmv^{\tau*}v\|_s^2+\|v\|_s^2\leq C_{s,\tau}\|\Delta_{\MV,\vartheta}^{\tau}v\|_s^2,\quad\forall v\in C^\infty\big(M,\Lambda_\MV^{m,q'}(L^{\tau}\otimes E)\big),
	\end{align}
	holds if $\tau\geq\tau_{s}$. Throughout the proof, $C>0$ denotes a constant that may change from line to line, and subscripts indicate its dependence on the specified parameters. Clearly, $(\Delta_{\MV,\vartheta}^{\tau}v,v)=\|\dmv v\|^2+\|\dmv^{\tau*}v\|^2$, then for any $0<\epsilon<1$,
	\begin{align*}
		\|\Delta_{\MV,\vartheta}^{\tau}v\|^2\geq\frac{1}{\epsilon}\big(\|\dmv v\|^2+\|\dmv^{\tau*}v\|^2\big)-\frac{1}{4\epsilon^2}\|v\|^2.
	\end{align*}
	On the other hand, from (\ref{ep}) we obtain
	\begin{align}\label{es0}
		\|\dmv v\|^2+\|\dmv^{\tau*}v\|^2\geq(\lambda\tau-C)\|v\|^2+\frac{1}{2}\int_{M}|\nabla_\MV v|^2
	\end{align}
	where $\lambda>0$ and $|\nabla_\MV v|^2$ is defined by (\ref{nabla.}). Hence,
	\begin{align*}
		\|\Delta_{\MV,\vartheta}^{\tau}v\|^2\geq\frac{1}{2\epsilon}\big(\|\dmv v\|^2+\|\dmv^{\tau*}v\|^2\big)+\left(\frac{\lambda\tau-C}{2\epsilon}-\frac{1}{4\epsilon^2}\right)\|v\|^2,
	\end{align*}
	which gives the estimate (\ref{sob.}) for $s=0$ by choosing $\tau_0$ such that $\frac{\lambda\tau_0-C}{2\epsilon}-\frac{1}{4\epsilon^2}\geq0$. By induction, suppose that (\ref{sob.}) holds for $s-1$. 
	
	Let $\{(U_\alpha;x_1^\alpha,\cdots,x_{m+n}^\alpha)\}_\alpha$ be a finite covering of $M$ by local coordinate charts, which are given by the complex Frobenius theorem (see \cite{Nl57}). Then
	\begin{align*}
		\MV|_{_{U_\alpha}}\ \text{is spanned by}\ \left\{\frac{\partial}{\partial\bar{z}_1^\alpha},\cdots,\frac{\partial}{\partial\bar{z}_d^\alpha},\frac{\partial}{\partial x_{m+d+1}^\alpha},\cdots,\frac{\partial}{\partial x_{m+n}^\alpha}\right\},
	\end{align*}
	where $d:=m-{\rm rank}_\mathbb{C}(N^*\MV\cap\overline{N^*\MV})$, and
	\begin{align*}
		z^\alpha=(x_1^\alpha+\sqrt{-1}x_{1+d}^\alpha,\cdots, x_{d}^\alpha+\sqrt{-1}x_{2d}^\alpha). 
	\end{align*}
	Passing to a common finite refinement of the coordinate covering and the covering used to establish (\ref{ep}), we may assume that
	each $U_\alpha$ is equipped with fixed local frames. We retain the notation $\{U_\alpha\}_\alpha$ for the refined covering
	and choose smooth real-valued cutoff functions $\psi_\alpha\in C_c^\infty(U_\alpha)$ satisfying $\sum_\alpha\psi_\alpha^2=1$ on $M$. The estimate
	(\ref{ep}) retains the same form after this choice, possibly with different constants. Set
	\begin{align*}
		\|u\|_s^2:=\sum_{\alpha}\sum_{|I|\leq s}\|\psi_\alpha D_\alpha^Iu\|^2,\quad u\in C^\infty\big(M,\Lambda_\MV^{m,q'}(L^{\tau}\otimes E)\big),
	\end{align*}
	where $D_\alpha^I:=(-\sqrt{-1})^{|I|}\partial_{x_{1}^{_\alpha}}^{^{i_1}}\cdots\partial_{x_{m+n}^{_\alpha}}^{^{i_{m+n}}}$ and $\psi_\alpha D_\alpha^I$ acts on $u$ componentwise. If $I$ is an $(n+d)$-tuple, we define
	\begin{align*}
		D_{\MV,\alpha}^{I}:=(-\sqrt{-1})^{|I|}
		\partial_{z_1^\alpha}^{i_1}\cdots\partial_{z_d^\alpha}^{i_d}
		\partial_{\bar z_1^\alpha}^{i_{d+1}}\cdots\partial_{\bar z_d^\alpha}^{i_{2d}}
		\partial_{x_{m+d+1}^\alpha}^{i_{2d+1}}\cdots
		\partial_{x_{m+n}^\alpha}^{i_{n+d}}.
	\end{align*}
	
	For each multi-index $I$ with $|I|\leq s$, 
	\begin{align*}
		\|\dmv(\psi_\alpha D_\alpha^I v)\|^2
		&=\|\psi_\alpha\dmv D_\alpha^I v\|^2
		+2\operatorname{Re}\bigl(\psi_\alpha\dmv D_\alpha^I v,[\dmv,\psi_\alpha]D_\alpha^I v\bigr)+\|[\dmv,\psi_\alpha]D_\alpha^I v\|^2\\
		&={\rm Re}\bigl(\dmv D_\alpha^I v,\psi_\alpha^2\dmv D_\alpha^I v
		+2\psi_\alpha[\dmv,\psi_\alpha]D_\alpha^I v\bigr)+\|[\dmv,\psi_\alpha]D_\alpha^I v\|^2\\
		&={\rm Re}\bigl(\dmv D_\alpha^I v,\dmv(\psi_\alpha^2D_\alpha^I v)\bigr)+\|[\dmv,\psi_\alpha]D_\alpha^I v\|^2\\
		&={\rm Re}\bigl(\dmv^{\tau*}\dmv D_\alpha^I v,\psi_\alpha^2D_\alpha^I v\bigr)+\|[\dmv,\psi_\alpha]D_\alpha^I v\|^2,
	\end{align*}
	where the third line holds by the identity $[\dmv,\psi_\alpha^2]=2\psi_\alpha[\dmv,\psi_\alpha]$. Similarly,
	\begin{align*}
		\|\dmv^{\tau*}(\psi_\alpha D_\alpha^I v)\|^2
		=&{\rm Re}\bigl(\dmv\dmv^{\tau*}D_\alpha^I v,\psi_\alpha^2D_\alpha^I v\bigr)+\|[\dmv^{\tau*},\psi_\alpha]D_\alpha^I v\|^2.
	\end{align*}
	It follows that
	\begin{align*}
		\|\dmv(\psi_\alpha D_\alpha^Iv)\|^2+\|\dmv^{\tau*}(\psi_\alpha D_\alpha^Iv)\|^2=&{\rm Re}\bigl(D_\alpha^I\Delta_{\MV,\vartheta}^{\tau}v,\psi_\alpha^2D_\alpha^Iv\bigr)+{\rm Re}\bigl([\Delta_{\MV,\vartheta}^{\tau},D_\alpha^I]v,\psi_\alpha^2D_\alpha^Iv\bigr)\\
		&+\|[\dmv,\psi_\alpha]D_\alpha^Iv\|^2+\|[\dmv^{\tau*},\psi_\alpha]D_\alpha^Iv\|^2.
	\end{align*}
	Both $[\dmv,\psi_\alpha]$ and $[\dmv^{\tau*},\psi_\alpha]$ are zeroth-order operators with coefficients independent of $\tau$, since the principal coefficients of $\dmv$ and $\dmv^{\tau*}$ are independent of $\tau$. Therefore,
	\begin{align}\label{es'}
		&\sum_\alpha\sum_{|I|\le s}
		\left(\|\dmv(\psi_\alpha D_\alpha^I v)\|^2
		+\|\dmv^{\tau*}(\psi_\alpha D_\alpha^I v)\|^2\right)\nonumber\\
		&\ \ \ \leq C_s\|\Delta_{\MV,\vartheta}^{\tau}v\|_s\|v\|_s
		+C_s\|v\|_s^2+C_{s,\tau}\|v\|_{s-1}^2+\sum_\alpha\sum_{|I|\le s}
		\left|\operatorname{Re}\bigl([\Delta_{\MV,\vartheta}^{\tau},D_\alpha^I]v,\psi_\alpha^2D_\alpha^I v\bigr)\right|.
	\end{align}
	We apply (\ref{es0}) to $\psi_\alpha D_\alpha^I v$, summing over $\alpha$ and $|I|\le s$. It then follows from (\ref{es'}) that
	\begin{align}\label{es2}
		(\lambda&\tau-C)\|v\|_s^2+\frac12\sum_\alpha\sum_{|I|\le s}\int_M|\nabla_\MV^\alpha(\psi_\alpha D_\alpha^Iv)|^2\nonumber\\
		&\leq C_s\|\Delta_{\MV,\vartheta}^{\tau}v\|_s\|v\|_s+C_s\|v\|_s^2+C_{s,\tau}\|v\|_{s-1}^2+\sum_\alpha\sum_{|I|\le s}\big|{\rm Re}\big([\Delta_{\MV,\vartheta}^{\tau},D_\alpha^I]v,\psi_\alpha^2D_\alpha^Iv\big)\big|.
	\end{align}
	On the other hand,
	\begin{align*}
		\|\dmv(\psi_\alpha D_\alpha^Iv)\|^2&+\|\dmv^{\tau*}(\psi_\alpha D_\alpha^Iv)\|^2\\
		&\geq\frac{1}{2}\big(\|\psi_\alpha D_\alpha^I\dmv v\|^2+\|\psi_\alpha D_\alpha^I\dmv^{\tau*}v\|^2\big)-C_s\|v\|_s^2-C_{s,\tau}\|v\|_{s-1}^2,
	\end{align*}
	which again by (\ref{es'}) gives
	\begin{align}\label{es''}
		\|\dmv v\|_s^2+\|\dmv^{\tau*}v\|_s^2\leq& C_s\|\Delta_{\MV,\vartheta}^{\tau}v\|_s\|v\|_s
		+C_s\|v\|_s^2+C_{s,\tau}\|v\|_{s-1}^2\nonumber\\
		&+2\sum_\alpha\sum_{|I|\le s}
		\left|\operatorname{Re}\bigl([\Delta_{\MV,\vartheta}^{\tau},D_\alpha^I]v,\psi_\alpha^2D_\alpha^I v\bigr)\right|.
	\end{align}
	Combining (\ref{es2}) and (\ref{es''}), we know that
	\begin{align}\label{es'''}
		&\|\dmv v\|_s^2+\|\dmv^{\tau*}v\|_s^2+(\lambda\tau-C)\|v\|_s^2+\frac12\sum_\alpha\sum_{|I|\le s}\int_M|\nabla_\MV^\alpha(\psi_\alpha D_\alpha^Iv)|^2\nonumber\\
		&\ \leq C_s\|\Delta_{\MV,\vartheta}^{\tau}v\|_s\|v\|_s+C_s\|v\|_s^2+C_{s,\tau}\|v\|_{s-1}^2+3\sum_\alpha\sum_{|I|\le s}\big|{\rm Re}\big([\Delta_{\MV,\vartheta}^{\tau},D_\alpha^I]v,\psi_\alpha^2D_\alpha^Iv\big)\big|.
	\end{align}
	
	It remains to estimate the last term on the right-hand side of (\ref{es'''}). Observe that the principal part of $\Delta_{\MV,\vartheta}^{\tau}$ is independent of $\tau$ and consists of products of two leafwise derivatives. Since $D_\alpha^I$ has constant coefficients for every multi-index $I$ with $|I|\leq s$, the commutator is generated by
	\begin{align*}
		D_{\mathcal{V},\alpha}^{I'}D_\alpha^Jv\ (|I'|=2),\quad
		D_{\mathcal{V},\alpha}^{I'} D_\alpha^Jv\ (|I'|\leq1),
	\end{align*}
	where $|J|\le s-1$. 
	Pairing these terms with $\psi_\alpha^2D_\alpha^Iv$, integrating by parts once in a suitable leafwise direction, and estimating the resulting lower-order terms in terms of Sobolev norms, we obtain
	\begin{align*}
		\left|{\rm Re}\bigl([\Delta_{\MV,\vartheta}^{\tau},D_\alpha^I]v,\psi_\alpha^2D_\alpha^Iv\bigr)\right|\leq C_{s,\tau}\sum_{\substack{|J|\le s-1\\|I'|\le1}}\|\psi_\alpha D_{\MV,\alpha}^{I'}D_\alpha^Jv\|
		\bigg(\|v\|_s+\|\nabla_\MV^\alpha(\psi_\alpha D_\alpha^Iv)\|\bigg).
	\end{align*}
	
	To control the term $\|\psi_\alpha D_{\MV,\alpha}^{I'}D_\alpha^Jv\|$, we first estimate $\|D_{\MV,\alpha}^{I'}w\|^2\ (|I'|=1)$. By (\ref{es0}) and integration by parts in the $\bar{z}^\alpha$-variables, we note that for every compactly supported local section $w$ over $U_\alpha$,
	\begin{align*}
		\sum_{|I'|=1}\|D_{\MV,\alpha}^{I'}w\|^2
		\le C_{\tau}\big(\|\nabla_\MV^\alpha w\|^2+\|w\|^2\big)
		\le C_{\tau}\left(\|\dmv w\|^2+\|\dmv^{\tau*}w\|^2+\|w\|^2\right).
	\end{align*}
	Since $[\dmv,\psi_\alpha D_\alpha^J]$ and $[\dmv^{\tau*},\psi_\alpha D_\alpha^J]$ have order at most $s-1$, applying the preceding estimate to $\psi_\alpha D_\alpha^Jv$ implies
	\begin{align*}
		\sum_{\substack{|J|\le s-1\\|I'|\le1}}\|\psi_\alpha D_{\MV,\alpha}^{I'}D_\alpha^Jv\|^2&\leq2\sum_{\substack{|J|\le s-1\\|I'|\le1}}\|D_{\MV,\alpha}^{I'}(\psi_\alpha D_\alpha^Jv)\|^2+C_{s,\tau}\|v\|_{s-1}^2\nonumber\\
		&\le C_{s,\tau}\left(\|\dmv v\|_{s-1}^2+\|\dmv^{\tau*}v\|_{s-1}^2+\|v\|_{s-1}^2\right)\nonumber\\
		&\le C_{s,\tau}\|\Delta_{\MV,\vartheta}^{\tau}v\|_{s-1}^2,
	\end{align*}
	where the last inequality follows from the induction hypothesis. Hence, the Cauchy-Schwarz inequality gives, for every sufficiently small $\epsilon'>0$,
	\begin{align}\label{es2'}
		\sum_\alpha\sum_{|I|\le s}&\left|{\rm Re}\bigl([\Delta_{\MV,\vartheta}^{\tau},D_\alpha^I]v,\psi_\alpha^2D_\alpha^Iv\bigr)\right|\nonumber\\
		&\le\epsilon'\bigg(\|v\|_s^2+\sum_\alpha\sum_{|I|\le s}\|\nabla_\MV^\alpha(\psi_\alpha D_\alpha^Iv)\|^2\bigg)+C_{s,\epsilon',\tau}\|\Delta_{\MV,\vartheta}^{\tau}v\|_{s-1}^2.
	\end{align}
	Substituting (\ref{es2'}) into (\ref{es'''}) and applying the induction hypothesis, we have
	\begin{align*}
		\|\dmv v\|_s^2+\|\dmv^{\tau*}v\|_s^2+(\lambda\tau-C_{s,\epsilon'})&\|v\|_s^2+\left(\frac12-3\epsilon'\right)\sum_\alpha\sum_{|I|\le s}\int_M|\nabla_\MV^\alpha(\psi_\alpha D_\alpha^Iv)|^2\nonumber\\
		&\leq C_s\|\Delta_{\MV,\vartheta}^{\tau}v\|_s^2+C_{s,\tau}\|\Delta_{\MV,\vartheta}^{\tau}v\|_{s-1}^2+C_{s,\epsilon',\tau}\|\Delta_{\MV,\vartheta}^{\tau}v\|_{s-1}^2.
	\end{align*}
	Choose $\epsilon'>0$ such that $\frac12-3\epsilon'>0$. This completes the induction and establishes the estimate (\ref{sob.}) for every $\tau\geq\tau_s$ by choosing a positive integer $\tau_s\geq\tau_{s-1}$ sufficiently
	large such that $\lambda\tau_s-C_{s,\epsilon'}>0$.
	
	We now establish the regularity of the canonical solution of the equation $\dmv u=f$. By a standard smooth approximation argument, it suffices to show that, for any $s\in\mathbb{Z}_{>0}$,
	\begin{align*}
		\|\dmv^{\tau*}N^\tau g\|_s\leq C_{s,\tau}\|g\|_s,\quad\forall g\in C^\infty\big(M,\Lambda_\MV^{m,q'}(L^{\tau}\otimes E)\big)
	\end{align*}
	provided that $q'\geq q$ and $\tau\geq \tau_{s}$. For $\delta>0$, we consider the following elliptic differential operator
	\begin{align*}
		\Delta_{\MV,\vartheta}^{\tau,\delta}:=\Delta_{\MV,\vartheta}^{\tau}+\delta\sum_{\alpha}\sum_{\varrho=2d+1}^{m+d}(\partial_{x_{\varrho}^{_\alpha}}{\scriptstyle \circ}\psi_\alpha)_\tau^*\partial_{x_\varrho^{_\alpha}}{\scriptstyle \circ}\psi_\alpha,
	\end{align*}
	where $(\partial_{x_\varrho^{_\alpha}}{\scriptstyle \circ}\psi_\alpha)_\tau^*$ is the formal adjoint of $\partial_{x_\varrho^{_\alpha}}{\scriptstyle \circ}\psi_\alpha$. By the same argument as in the proof of (\ref{sob.}), the estimate (\ref{sob.}) remains valid uniformly for all sufficiently small $\delta>0$, with $\Delta_{\MV,\vartheta}^{\tau}$ replaced by $\Delta_{\MV,\vartheta}^{\tau,\delta}$. Thus, $\Delta_{\MV,\vartheta}^{\tau,\delta}$ is injective and we denote its Green operator by $N^{\tau,\delta}$. Since $g\in C^\infty\big(M,\Lambda_\MV^{m,q'}(L^{\tau}\otimes E)\big)$, the ellipticity of $\Delta_{\MV,\vartheta}^{\tau,\delta}$ leads to $N^{\tau,\delta} g\in C^\infty\big(M,\Lambda_\MV^{m,q'}(L^{\tau}\otimes E)\big)$, thereby
	\begin{align*}
		\|N^{\tau,\delta} g\|\leq C_{s,\tau}\|g\|,\ \|\dmv N^{\tau,\delta} g\|_s\leq C_{s,\tau}\|g\|_s\ \text{and}\ \|\dmv^{\tau*}N^{\tau,\delta} g\|_s\leq C_{s,\tau}\|g\|_s.
	\end{align*}
	We can then find a weak limit $\tilde{g}$ of $N^{\tau,\delta} g$ when $\delta\rightarrow0$ such that
	\begin{align*}
		\|\tilde{g}\|\leq C_{s,\tau}\|g\|,\ \|\dmv\tilde{g}\|_s\leq C_{s,\tau}\|g\|_s,\ \|\dmv^{\tau*}\tilde{g}\|_s\leq C_{s,\tau}\|g\|_s
	\end{align*}
	and $\Delta_{\MV,\vartheta}^{\tau}\tilde{g}=g$ in the sense of distributions. Thus, $\dmv\tilde{g}\in W^{1}\big(M,\Lambda_\MV^{m,q'+1}(L^{\tau}\otimes E)\big)$ and $\dmv^{\tau*}\tilde{g}\in W^{1}\big(M,\Lambda_\MV^{m,q'-1}(L^{\tau}\otimes E)\big)$ which gives $\tilde{g}\in{\rm Dom}(\Delta_{\MV,\vartheta}^{\tau})$. Since $\Delta_{\MV,\vartheta}^{\tau}\tilde{g}=g=\Delta_{\MV,\vartheta}^{\tau}N^\tau g$, it follows from the injectivity of $\Delta_{\MV,\vartheta}^{\tau}$ that $\tilde{g}=N^\tau g$. Therefore, $\|\dmv^{\tau*}N^\tau g\|_s=\|\dmv^{\tau*}\tilde{g}\|_s\leq C_{s,\tau}\|g\|_s$.
\end{proof}

\section{The Morse-Novikov-Treves complex}
Let $\MV$ be a formally integrable structure of rank $n$ over an $(m+n)$-dimensional smooth manifold $M$. We introduce the Morse-Novikov-Treves complex associated with $\MV$ and provide a global realization of this complex. The global solvability of this complex is established. We also obtain vanishing results for the leafwise $L_{loc}^2$-cohomology.

\subsection{Global solvability of the Morse-Novikov-Treves complex}\label{global.}
The motivation for introducing $\Xi_\MV$ in the definition of the differential operator $\dmv$ in the complex (\ref{complex1}) is twofold: first, to ensure that ${\rm Ker}\big(\mathcal{L}_\MV^{m,0}(E)\stackrel{\dV }{\longrightarrow}\mathcal{L}_\MV^{m,1}(E)\big)=\mathcal{S}_\MV(K_\MV\otimes E)$ for any basic vector bundle $E$ (when $\vartheta\equiv0$); second, to provide a global realization of the Morse-Novikov-Treves complex as follows.

Let $\vartheta$ be a global smooth $1$-form satisfying (\ref{mnf}). For $q\geq 0$, it follows from the integrability of $\MV$ that
\begin{equation*}
	\rd_\vartheta\mathcal{A}_\MV^{1,q}\subseteq \mathcal{A}_\MV^{1,q+1},
\end{equation*} 
where $\rd_\vartheta:=\rd-\vartheta\wedge$ and $\mathcal{A}_\MV^{1,*}$ is the sheaf of germs of smooth sections of $\Lambda_\MV^{1,*}$ (see (\ref{e12})). We introduce the Morse-Novikov-Treves complex as the quotient sequence of the Morse-Novikov type sequence ($\mathcal{A}_\MV^*,\rd_\vartheta$):
\begin{equation}\label{new complex}
	\mathfrak{U}_\MV ^0\stackrel{\rd_\vartheta'}{\longrightarrow}\mathfrak{U}_\MV ^1\stackrel{\rd_\vartheta'}{\longrightarrow}\mathfrak{U}_\MV ^2\stackrel{\rd_\vartheta'}{\longrightarrow}\mathfrak{U}_\MV ^3\stackrel{\rd_\vartheta'}{\longrightarrow}\cdots\stackrel{\rd_\vartheta'}{\longrightarrow}\mathfrak{U}_\MV^n\longrightarrow0,
\end{equation}
where $\mathfrak{U}_\MV^*$ is the sheaf of smooth sections of the quotient bundle ${\Lambda^*\mathbb{C}T^*M}/{\Lambda_\MV^{1,*-1}}$ (here $\Lambda_\MV^{1,-1}:=0$), or equivalently,  $\mathfrak{U}_\MV^*=\frac{\mathcal{A}^*}{\mathcal{A}_\MV^{1,*-1}}$ (here $\mathcal{A}^*$ is the sheaf of smooth forms), and
\begin{align}\label{dvt}
	\rd_\vartheta':\mathfrak{U}_\MV^q&\rightarrow\mathfrak{U}_\MV^{q+1}\nonumber\\
	u\ {\rm mod}\ \mathcal{A}_\MV^{1,q-1}&\mapsto(\rd u-\vartheta\wedge u)\ {\rm mod}\ \mathcal{A}_\MV^{1,q}.
\end{align}
It follows that $(\rd_\vartheta')^2=0$ by (\ref{mnf}). In particular, the complex (\ref{new complex}) is exactly the Treves complex in \cite{Tf81} if $\vartheta\equiv0$ where we will denote $\rd_0'$ by $\rd'$ in what follows. When $\MV$ arises from a Levi flat structure, both $\MV$ itself and $\MV+\overline{\MV}$ are formally integrable (the latter being essentially real), yielding two natural differential complexes. Recently, Paulo D. Cordaro and V. Novelli established an intriguing comparison principle between the differential complexes associated with $\MV$ and $\MV+\overline{\MV}$ on Levi flat CR manifolds.

The complex (\ref{new complex}) extends naturally to the following complex
\begin{equation}\label{new complex1}
	\mathfrak{L}_\MV ^0\stackrel{\rd_\vartheta'}{\longrightarrow}\mathfrak{L}_\MV ^1\stackrel{\rd_\vartheta'}{\longrightarrow}\mathfrak{L}_\MV ^2\stackrel{\rd_\vartheta'}{\longrightarrow}\mathfrak{L}_\MV ^3\stackrel{\rd_\vartheta'}{\longrightarrow}\cdots\stackrel{\rd_\vartheta'}{\longrightarrow}\mathfrak{L}_\MV^n\longrightarrow0,
\end{equation}
where $\mathfrak{L}_\MV^q$ for $0\leq q\leq n$ is the sheaf generated by the following complete presheaf
\begin{align*}
	\mathfrak{L}_\MV^q(U):=\left\{[u]\in L_{loc}^2(U,{\Lambda^q\mathbb{C}T^*M}/{\Lambda_\MV^{1,q-1}})\ |\ \rd_\vartheta'[u]\in L_{loc}^2(U,{\Lambda^{q+1}\mathbb{C}T^*M}/{\Lambda_\MV^{1,q}})\right\},
\end{align*}
where $U$ is any open subset of $M$.

Given a sufficiently small open subset $U\subseteq M$, let $L$ be a basic line bundle. We have the following commutative diagram over $U$ for $q\geq0$
\begin{align}\label{compare}
	\begin{CD}
		L_{loc}^2(U,\Lambda_\MV^{m,q}(L))@>\dmv >>  L_{loc}^2(U,\Lambda_\MV^{m,q+1}(L))\\
		@VV \Phi_q^{_U} V  @V \Phi_{q+1}^{_U}	 VV\\
		L_{loc}^2(U,{\Lambda^q\mathbb{C}T^*M}/{\Lambda_\MV^{1,q-1}}) @>\rd_\vartheta'>> L_{loc}^2(U,{\Lambda^{q+1}\mathbb{C}T^*M}/{\Lambda_\MV^{1,q}}),
	\end{CD}
\end{align}
where
\begin{align}\label{phi_q}
	\Phi_q^{_U}: L_{loc}^2(U,\Lambda_\MV^{m,q}(L))&\rightarrow L_{loc}^2(U,{\Lambda^q\mathbb{C}T^*M}/{\Lambda_\MV^{1,q-1}})\\
	\Theta_{_U}\wedge v_{_U}\otimes \sigma_{_U}&\mapsto (-1)^{qm}v_{_U}\ {\rm mod}\ L_{loc}^2(U,\Lambda_\MV^{1,q-1}),\nonumber
\end{align}
in which $\Theta_{_U}:=\theta^1\wedge\cdots\wedge\theta^m$ (here $\{\theta^1,\cdots,\theta^m\}$ is a frame of $N^*\MV$ over $U$ such that $\Theta_{_U}$ is basic) and $\sigma_{_U}$ is a basic frame of $L$ over $U$. In fact,
\begin{align*}
	\dmv(\Theta_{_U}\wedge v_{_U}\otimes \sigma_{_U})&=(\rd -\Xi_\MV-\vartheta\wedge)(\Theta_{_U}\wedge v_{_U})\otimes \sigma_{_U}\\
	&=\big(\rd \Theta_{_U}\wedge v_{_U}+(-1)^m\Theta_{_U}\wedge\rd v_{_U}-\Xi_\MV(\Theta_{_U}\wedge v_{_U})-\vartheta\wedge(\Theta_{_U}\wedge v_{_U})\big)\otimes \sigma_{_U}\\
	&=(-1)^m\Theta_{_U}\wedge\rd_\vartheta v_{_U}\otimes \sigma_{_U},
\end{align*}
then
\begin{align*}
	\Phi_{q+1}^{_U}\circ\dmv(\Theta_{_U}\wedge v_{_U}\otimes \sigma_{_U})&=\Phi_{q+1}^{_U}\big((-1)^m\Theta_{_U}\wedge\rd_\vartheta v_{_U}\otimes \sigma_{_U}\big)\\
	&=(-1)^{(q+2)m}\rd_\vartheta v_{_U}\ {\rm mod}\ L_{loc}^2(U,\Lambda_\MV^{1,q})\\
	&\equiv\rd_\vartheta'\big((-1)^{qm}v_{_U}\ {\rm mod}\ L_{loc}^2(U,\Lambda_\MV^{1,q-1})\big),
\end{align*}
which implies that (\ref{compare}) is commutative.

It is evident from (\ref{phi_q}) that each $\Phi_q^{_U}$ is an isomorphism depending on the choice of $\Theta_{_U}$. Furthermore, the restriction of $\Phi_q^{_U}$ induces an isomorphism from $C^\infty(U,\Lambda_\MV^{m,q}(L))$ to $C^\infty(U,{\Lambda^q\mathbb{C}T^*M}/{\Lambda_\MV^{1,q-1}})$. By taking $L$ to be the dual bundle $K_{\MV}^{-1}$ of the canonical bundle $K_\MV$, we obtain the following global realization of the complexes (\ref{new complex1}) and (\ref{new complex}).

\begin{prop}\label{gr1}
	If $\MV$ is a formally integrable structure with the basic canonical bundle $K_{\MV}$, then local isomorphisms $\Phi_q^{_U}$ in $(\ref{phi_q})$ can be patched together to form a sheaf-theoretic isomorphism $\Phi_q:\mathcal{L}_\MV^{m,q}(K_{\MV}^{-1})\rightarrow\mathfrak{L}_{\MV}^q$ such that
	\begin{align*}
		\begin{CD}
			\mathcal{L}_\MV^{m,q}(K_{\MV}^{-1})@>\dmv >> \mathcal{L}_\MV^{m,q+1}(K_{\MV}^{-1})\\
			@VV \Phi_q V  @V \Phi_{q+1}	VV\\
			\mathfrak{L}_{\MV}^q @>\rd_\vartheta'>> \mathfrak{L}_{\MV}^{q+1}
		\end{CD}
	\end{align*}
	is commutative for any $0\leq q\leq n$. In particular, $\Phi_q$ is well-defined provided that $\MV$ is locally integrable. Moreover, the collection $\{\Phi_q\}_{q=0}^n$ also induces an isomorphism from the complex $(\mathcal{A}_\MV^{m,*}(K_{\MV}^{-1}),\dmv)$ to the Morse-Novikov-Treves complex $(\mathfrak{U}_\MV^*,\rd_\vartheta')$.
\end{prop}
\begin{proof}
	Since $\Phi_q^{_U}$ is an isomorphism such that (\ref{compare}) is commutative for any sufficiently small open subset $U\subseteq M$, it suffices to prove that for any pair of sufficiently small open subsets $V\subseteq U\subseteq M$, we have the following commutative diagram
	\begin{align}\label{compare2}
		\begin{CD}
			L_{loc}^2\big(U,\Lambda_\MV^{m,q}(K_{\MV}^{-1})\big)@>\Phi_q^{_U}>>L_{loc}^2(U,{\Lambda^q\mathbb{C}T^*M}/{\Lambda_\MV^{1,q-1}})\\
			@VV\rho_{_V}^{_U}V@V\tilde{\rho}_{_V}^{_U}VV\\
			L_{loc}^2\big(V,\Lambda_\MV^{m,q}(K_{\MV}^{-1})\big)@>\Phi_q^{_V}>>L_{loc}^2(V,{\Lambda^q\mathbb{C}T^*M}/{\Lambda_\MV^{1,q-1}}),
		\end{CD}
	\end{align}
	where $\rho_{_V}^{_U}$ and $\tilde{\rho}_{_V}^{_U}$ are canonical restriction maps.
	
	Clearly, $K_{\MV}\otimes K_{\MV}^{-1}$ is a trivial bundle, we then always choose each
	basic frame $\sigma_{_U}$ of $K_\MV^{-1}$ to be dual to $\Theta_{_U}$, so that
	\begin{align*}
		(\Theta_{_U}\otimes \sigma_{_U})|_{_V}=\Theta_{_V}\otimes \sigma_{_V},
	\end{align*}
	where $\Theta_{_U}$ and $\sigma_{_U}$ are basic frames over the open set $U$ of $K_{\MV}$ and $K_{\MV}^{-1}$ respectively. For
	\begin{align*}
		\Theta_{_U}\wedge v_{_U}\otimes \sigma_{_U}\in L_{loc}^2\big(U,\Lambda_\MV^{m,q}(K_{\MV}^{-1})\big),
	\end{align*}
	we have
	\begin{align*}
		\rho_{_V}^{_U}(\Theta_{_U}\wedge v_{_U}\otimes \sigma_{_U})=(\Theta_{_U}\wedge v_{_U}\otimes \sigma_{_U})|_{_V}=\Theta_{_V}\wedge v_{_U}|_{_V}\otimes \sigma_{_V},
	\end{align*}
	which gives
	\begin{align*}
		\Phi_q^{_V}\circ\rho_{_V}^{_U}(\Theta_{_U}\wedge v_{_U}\otimes \sigma_{_U})&=(-1)^{qm}v_{_U}|_{_V}\ {\rm mod}\ L_{loc}^2(V,\Lambda_{\MV}^{1,q-1})\\
		&=\tilde{\rho}_{_V}^{_U}\big((-1)^{qm}v_{_U}\ {\rm mod}\ L_{loc}^2(U,\Lambda_{\MV}^{1,q-1})\big)\\
		&=\tilde{\rho}_{_V}^{_U}\circ\Phi_q^{_U}(\Theta_{_U}\wedge v_{_U}\otimes \sigma_{_U}),
	\end{align*}
	i.e., (\ref{compare2}) is commutative. Hence, $\{\Phi_q^{_U}\}_{_U}$ induces a sheaf-theoretic isomorphism $\Phi_q$. For locally integrable $\MV$, the conclusion follows from the fact that $K_{\MV}$ is a basic line bundle, as shown in Example \ref{exa} (iii). The final statement is immediate since $\Phi_q^{_U}$ preserves smoothness for all $q\geq0$ and sufficiently small open sets $U$.
\end{proof}

There has been considerable progress on the global hypoellipticity and global solvability of the Treves complex (see \cite{CC17}, \cite{CJ21}, \cite{ADd23a}, \cite{ADd23}, \cite{AFJR24} and references therein). Proposition \ref{gr1}, combined with Theorem \ref{t1'} gives the global solvability of the Morse-Novikov-Treves complex as follows.
\begin{cor}[=Corollary \ref{treves'}]\label{treves}
	For any Levi flat structure $\MV$, let $\vartheta$ be a global smooth $1$-form satisfying $(\ref{mnf})$. If $M$ has a smooth $q$-convex exhaustion function with respect to $\MV$, then the Morse-Novikov-Treves complex $($see $(\ref{new complex}))$ is globally $L_{loc}^2$-solvable in degree $q'\geq q$, i.e., for every $f\in L_{loc}^2(M,\Lambda^{q'}\mathbb{C}T^*M/\Lambda_\MV^{1,q'-1})$ with $\rd_\vartheta'f=0$, there is a $u\in L_{loc}^2(M,\Lambda^{q'-1}\mathbb{C}T^*M/\Lambda_\MV^{1,q'-2})$ such that $\rd_\vartheta'u=f$. In particular, this global $L_{loc}^2$-solvability can be strengthened to global smooth solvability when $\MV$ is elliptic.
\end{cor}

\begin{cor}\label{sros}
	Let $\MV$ be a Levi flat structure on a smooth manifold $M$, $E$ a basic vector bundle over $(M,\MV)$, and let $\vartheta$ be a global smooth $1$-form satisfying $(\ref{mnf})$. The complex $(\ref{complex11})$ is exact. In particular, for any $0\leq q\leq n$ we have
	\begin{align}\label{ISO1}
		H^q(M,\mathcal{O}_\MV (E))\cong\frac{{\rm{Ker}}\left(C^\infty\big(M,\Lambda_\MV^{m,q}(K_\MV^{-1}\otimes E)\big)\stackrel{\dV}{\longrightarrow}C^\infty\big(M,\Lambda_\MV^{m,q+1}(K_\MV^{-1}\otimes E)\big)\right)}{{\rm{Im}}\left(C^\infty\big(M,\Lambda_\MV^{m,q-1}(K_\MV^{-1}\otimes E)\big)\stackrel{\dV}{\longrightarrow}C^\infty\big(M,\Lambda_\MV^{m,q}(K_\MV^{-1}\otimes E)\big)\right)}.
	\end{align}
\end{cor}

\begin{proof}
	Since the complex $(\mathcal{L}_\MV^{m,*}(K_{\MV}^{-1}),\dV)$ is exact by Corollary \ref{ros}, Proposition \ref{gr1} and Theorem VIII.9.1 in \cite{T2} imply that the complex $(\mathfrak{U}_\MV^{*},\rd')$ is also exact. It follows from (\ref{mnf}) that $\rd'[\vartheta]=0$ where $[\vartheta]\in C^\infty(U,\mathbb{C}T^*M/\Lambda_\MV^{1,0})$ for any open subset $U\subseteq M$, then there exists an open subset $V\subset\subset U$ and $v\in C^\infty(V)$ such that
	\begin{align*}
		\rd v\equiv\vartheta|_{_V}\ \text{mod}\ C^\infty(V,N^*\MV).
	\end{align*}
	Hence, by (\ref{dvt}) we have the following commutative diagram for $q\geq0$:
	\begin{align*}
		\begin{CD}
			C^\infty(V,{\Lambda^q\mathbb{C}T^*M}/{\Lambda_\MV^{1,q-1}})@>\rd_\vartheta' >>  C^\infty(V,{\Lambda^{q+1}\mathbb{C}T^*M}/{\Lambda_\MV^{1,q}})\\
			@VV (e^{-v})\cdot V  @V (e^{-v})\cdot	 VV\\
			C^\infty(V,{\Lambda^q\mathbb{C}T^*M}/{\Lambda_\MV^{1,q-1}}) @>\rd'>> C^\infty(V,{\Lambda^{q+1}\mathbb{C}T^*M}/{\Lambda_\MV^{1,q}}),
		\end{CD}
	\end{align*}
	where $(e^{-v})\cdot$ denotes the multiplication operator by $e^{-v}$, which is evidently an isomorphism. Consequently, the complex $(\mathfrak{U}_\MV^{*},\rd_\vartheta')$ is exact, which yields that $(\mathcal{A}_\MV^{m,*}(K_{\MV}^{-1}),\dmv)$ is also exact according to Proposition \ref{gr1}.
	
	For $q\geq0$, the definition of $\dV$ in (\ref{bo}) gives a commutative diagram:
	\begin{align}\label{dd}
		\begin{CD}
			\mathcal{A}_\MV^{m,q}(K_{\MV}^{-1}\otimes E)@>\dmv>> \mathcal{A}_\MV^{m,q+1}(K_{\MV}^{-1}\otimes E)\\
			@VV \backsimeq V  @V \backsimeq	VV\\
			\mathcal{A}_\MV^{m,q}(K_{\MV}^{-1})\otimes_{_{\mathcal{O}_\MV}}\mathcal{O}_\MV(E) @>\dmv\otimes{\rm Id}_E>> \mathcal{A}_\MV^{m,q+1}(K_{\MV}^{-1})\otimes_{_{\mathcal{O}_\MV}}\mathcal{O}_\MV(E).
		\end{CD}
	\end{align}
	Clearly, $\mathcal{O}_\MV(E)$ is a locally free $\mathcal{O}_\MV$-module, then by (\ref{dd}) the complex of sheaves
	\begin{align}\label{com}
		\mathcal{A}_\MV^{m,0}(K_{\MV}^{-1}\otimes E)\stackrel{\dmv }{\longrightarrow}\mathcal{A}_\MV^{m,1}(K_{\MV}^{-1}\otimes E)\stackrel{\dmv }{\longrightarrow}\cdots\stackrel{\dmv }{\longrightarrow}\mathcal{A}_\MV^{m,n}(K_{\MV}^{-1}\otimes E)\longrightarrow0
	\end{align}
	is exact. In particular, the complex $(\mathcal{A}_\MV^{m,*}(K_{\MV}^{-1}\otimes E),\dV)$ is a fine resolution of $\mathcal{O}_\MV(E)$ and thereby (\ref{ISO1}) holds. The exactness of the complex (\ref{complex11}) follows immediately from replacing $E$ with $K_\MV\otimes E$ in (\ref{com}).
\end{proof}

\subsection{A reformulation of the Morse-Novikov-Treves complex}\label{iso}
The canonical quotient homomorphism $ \mathbb{C}T^*M\rightarrow\mathbb{C}T^*M/N^*\MV$, together with the isomorphism
\begin{align}\label{quo}
	\mathbb{C}T^*M/N^*\MV\cong\MV^*,
\end{align}
induces an epimorphism for any $q\geq0$,
\begin{align*}
	\pi_q:\Lambda^q\mathbb{C}T^*M\rightarrow\Lambda^q(\mathbb{C}T^*M/N^*\MV)\cong\Lambda^q\MV^*
\end{align*}
whose kernel contains $\Lambda_\MV^{1,q-1}$. Since for $1\leq q\leq n$,
\begin{align*}
	\dim_{\mathbb{C}}\Lambda^q\mathbb{C}T^*M/\Lambda_\MV^{1,q-1}&=\binom{m+n}{q}-\sum_{j=1}^{\min\{m,q\}}\binom{m}{j}\binom{n}{q-j}=\binom{n}{q}=\dim_{\mathbb{C}}\Lambda^q\MV^*,
\end{align*}
the epimorphism $\pi_q$ induces an isomorphism
\begin{align*}
	\tilde{\Phi}_q:\Lambda^q\mathbb{C}T^*M/\Lambda_\MV^{1,q-1}\rightarrow\Lambda^q\MV^*.
\end{align*}
Thus, we have the following sheaf-theoretic commutative diagram for $q\geq0$
\begin{align}\label{compare1}
	\begin{CD}
		\mathfrak{L}_{\MV}^q@>\rd_\vartheta'>>\mathfrak{L}_{\MV}^{q+1}\\
		@VV \tilde{\Phi}_q V  @V \tilde{\Phi}_{q+1}	 VV\\
		\mathfrak{L}^q\MV^*@>\rd_\vartheta''>>\mathfrak{L}^{q+1}\MV^*,
	\end{CD}
\end{align}
where $\rd_\vartheta'':=\tilde{\Phi}_{q+1}\circ\rd_\vartheta'\circ\tilde{\Phi}_q^{-1}$, and $\mathfrak{L}^q\MV^*$ for $0\leq q\leq n$ is the sheaf generated by the following complete presheaf
\begin{align*}
	\mathfrak{L}^q\MV^*(U):=\left\{[u]\in L_{loc}^2(U,{\Lambda^q\MV^*})\ |\ \rd_\vartheta''[u]\in L_{loc}^2(U,{\Lambda^{q+1}\MV^*})\right\},
\end{align*}
where $U\subseteq M$ is any open subset. It is straightforward to see that the restriction $\rd_\vartheta'':C^\infty(M,\Lambda^q\MV^*)\rightarrow C^\infty(M,\Lambda^{q+1}\MV^*)$ is determined by
\begin{align*}
	\rd_\vartheta'' u(X_{1},\cdots,X_{q+1})=&\sum_{j=1}^{q+1}(-1)^{j+1}X_{j}\big(u(X_{1},\cdots,\hat{X}_{j}\cdots,X_{{q+1}})\big)\nonumber\\
	&+\sum_{j<k}(-1)^{j+k}u\big([X_{j},X_{k}],X_{1},\cdots,\hat{X}_{j},\cdots,\hat{X}_{k},\cdots,X_{q+1}\big)\nonumber\\
	&-([\vartheta]\wedge u)(X_{1},\cdots,X_{q+1}),
\end{align*}
where $X_1,\cdots,X_{q+1}\in C^\infty(M,\MV)$, the notation $\hat{X}_j$ means $X_j$ is omitted and the form $[\vartheta]:=(\vartheta\ {\rm mod}\ C^\infty(M,N^*\MV))\in C^\infty(M,\MV^*)$ by the isomorphism (\ref{quo}).

The diagram (\ref{compare1}) then yields a complex $(\mathfrak{L}^*\MV^*,\rd_\vartheta'')$ for $\MV$ as follows
\begin{align}\label{lmnc}
	\mathfrak{L}^0\MV^*\stackrel{\rd_\vartheta''}{\longrightarrow}\mathfrak{L}^1\MV^*\stackrel{\rd_\vartheta''}{\longrightarrow}\mathfrak{L}^2\MV^*\stackrel{\rd_\vartheta''}{\longrightarrow}\mathfrak{L}^3\MV^*\stackrel{\rd_\vartheta''}{\longrightarrow}\cdots\stackrel{\rd_\vartheta''}{\longrightarrow}\mathfrak{L}^n\MV^*\stackrel{}{\longrightarrow}0.
\end{align}
In other words, we obtain
\begin{prop}\label{isop}
	Let $\MV$ be a formally integrable structure, and $\vartheta$ a global smooth $1$-form satisfying $(\ref{mnf})$. Then the complex $(\ref{lmnc})$ for $\MV$ is isomorphic to the complex $(\ref{new complex1})$ associated with $\MV$.
\end{prop}

We now discuss two special cases of the complex (\ref{lmnc}), where $\MV$ is an essentially real structure and $\MV$ defines a Levi flat CR structure respectively as follows.

\subsubsection{The \texorpdfstring{$L_{loc}^2$}{}-leafwise Morse-Novikov complex}
For any essentially real structure $\MV$ (i.e., $\MV=\overline{\MV}$), by the Frobenius theorem, there exists a foliation $\mathcal{F}$ on $M$ of dimension $n$, such that its tangent bundle $\mathbb CT\mathcal{F}=\MV$ and
\begin{align*}
	\mathbb CT^*\mathcal{F}=\MV^*\cong\mathbb{C}T^*M/N^*\MV.
\end{align*}
To align with standard conventions in foliation theory, we replace the notation $(\Lambda^*\MV^*,\rd_\vartheta'')$ with $(\Lambda^*\mathbb CT^*\mathcal{F},\dfv)$. The complex (\ref{lmnc}) gives the following complex
\begin{align}\label{lmnc1}
	L_{loc}^2(M,\Lambda^0\mathbb CT^*\mathcal{F})\stackrel{\dfv}{\longrightarrow}L_{loc}^2(M,\Lambda^1\mathbb CT^*\mathcal{F})\stackrel{\dfv}{\longrightarrow}L_{loc}^2(M,\Lambda^2\mathbb CT^*\mathcal{F})\stackrel{\dfv}{\longrightarrow}\cdots,
\end{align}
which is called \emph{the $L_{loc}^2$-leafwise Morse-Novikov complex} for $\MV$. The $q$-th cohomology group $L_{loc}^2H_{\mathcal{F},[\vartheta]}^{q}(M)$ of the complex (\ref{lmnc1}) is called the $q$-th \emph{leafwise Morse-Novikov $L_{loc}^2$-cohomology group}.

Proposition \ref{isop} and Corollary \ref{treves} yield a vanishing result for the leafwise Morse-Novikov $L_{loc}^2$-cohomology.
\begin{cor}\label{vlmnc1}
	Let $(M,\mathcal{F})$ be a foliated manifold, and let $[\vartheta]$ be a leafwise $1$-form which is $\rd_{\mathcal{F},0}$-closed. For $q\geq1$, assume that there exists an exhaustion function $\varphi\in C^\infty(M)$ such that for each leaf $F$,
	\begin{align*}
		{\rm Hess}_{\varphi|_{_F}}\ \text{has at least}\ n-q+1\ \text{positive eigenvalues at every critical point of}\ \varphi|_{_F}.
	\end{align*}
	Then $L_{loc}^2H_{\mathcal{F},[\vartheta]}^{q'}(M)=0$ for $q'\geq q$.
\end{cor}

\subsubsection{The \texorpdfstring{$L_{loc}^2$}{}-leafwise Morse-Novikov-Dolbeault complex}
Let $\MV$ be a Levi flat CR structure. According to the complex Frobenius theorem (see \cite{Nl57}), every $\lmfp\in M$ possesses a coordinate chart
\begin{align*}
	(U;z,y):=(U;x_1+\sqrt{-1}x_{1+n},\cdots, x_{n}+\sqrt{-1}x_{2n},y_1,\cdots,y_{m-n})
\end{align*}
centered at $\lmfp$ such that
\begin{align*}
	\MV|_{_U}\ \text{is spanned by}\ \left\{\frac{\partial}{\partial\bar{z}_1},\cdots,\frac{\partial}{\partial\bar{z}_n}\right\}.
\end{align*}
Hence, $M$ is foliated by complex leaves (locally determined by $y_1={\rm constant},\cdots,y_{m-n}={\rm constant}$) of complex dimension $n$. Such a foliation $\mathcal{F}$ is known as a \emph{complex foliation}. The complexified tangent bundle $\mathbb{C}T\mathcal{F}$ of $\mathcal{F}$ splits as
\begin{align*}
	\mathbb{C}T\mathcal{F}=T^{1,0}\mathcal{F}\oplus T^{0,1}\mathcal{F},
\end{align*}
where $T^{1,0}\mathcal{F}$ is the leafwise holomorphic tangent bundle of $\mathcal{F}$ and $T^{0,1}\mathcal{F}$ is the anti-holomorphic tangent bundle of $\mathcal{F}$. It's clear that $T^{0,1}\mathcal{F}=\MV$ and
\begin{align*}
	\Lambda^{0,1}\mathbb{C}T^*\mathcal{F}:=(T^{0,1}\mathcal{F})^*=\MV^*\cong\mathbb{C}T^*M/N^*\MV.
\end{align*}
Similarly, we replace the notation $(\Lambda^*\MV^*,\rd_\vartheta'')$ with $(\Lambda^{0,*}\mathbb{C}T^*\mathcal{F},\dbf)$, to be consistent with the notation commonly used in foliation theory. Then the complex (\ref{lmnc}) implies the following complex
\begin{align}\label{lmnc2}
	L_{loc}^2(M,\Lambda^{0,0}\mathbb{C}T^*\mathcal{F})\stackrel{\dbf}{\longrightarrow}L_{loc}^2(M,\Lambda^{0,1}\mathbb{C}T^*\mathcal{F})\stackrel{\dbf}{\longrightarrow}L_{loc}^2(M,\Lambda^{0,2}\mathbb{C}T^*\mathcal{F})\stackrel{\dbf}{\longrightarrow}\cdots,
\end{align}
which is called \emph{the $L_{loc}^2$-leafwise Morse-Novikov-Dolbeault complex}. The $q$-th cohomology group $L_{loc}^2H_{\mathcal{F},[\vartheta]}^{0,q}(M)$ of the complex (\ref{lmnc2}) is called the $q$-th \emph{leafwise Morse-Novikov-Dolbeault $L_{loc}^2$-cohomology group}.

Proposition \ref{isop} and Corollary \ref{treves} also imply a vanishing result for $L_{loc}^2H_{\mathcal{F},[\vartheta]}^{0,q}(M)$.
\begin{cor}\label{vlmnc2}
	Let $\mathcal{F}$ be a complex foliation on  $M$, and let $[\vartheta]$ be a leafwise $(0,1)$-form satisfying $\bar{\partial}_{\mathcal{F},0}[\vartheta]=0$. For $q\geq1$, suppose that there exists a smooth exhaustion function $\varphi$ on $M$ such that for each leaf $F$,
	\begin{align*}
		\text{the restriction of}\ \sqrt{-1}\partial\bar{\partial}(\varphi|_{_F})\ \text{to}\ T^{0,1}F\cap{\rm Ker}(\rd\varphi)
	\end{align*}
	has at least
	\begin{align*}
		\dim_{\mathbb{C}}(T^{0,1}F\cap{\rm Ker}(\rd\varphi))-q+1\ \text{positive eigenvalues}.
	\end{align*}
	Then $L_{loc}^2H_{\mathcal{F},[\vartheta]}^{0,q'}(M)=0$ for $q'\geq q$.
\end{cor}
When $\vartheta\equiv0$, the complex (\ref{lmnc2}) restricts to
\begin{align*}
	C^\infty(M,\Lambda^{0,0}\mathbb{C}T^*\mathcal{F})\stackrel{\bar{\partial}_\mathcal{F}}{\longrightarrow}C^\infty(M,\Lambda^{0,1}\mathbb{C}T^*\mathcal{F})\stackrel{\bar{\partial}_\mathcal{F}}{\longrightarrow}C^\infty(M,\Lambda^{0,2}\mathbb{C}T^*\mathcal{F})\stackrel{\bar{\partial}_\mathcal{F}}{\longrightarrow}\cdots,
\end{align*}
where $\bar{\partial}_\mathcal{F}:=\bar{\partial}_{\mathcal{F},0}$. This complex is known as the \emph{leafwise Dolbeault complex} of $(M,\mathcal{F})$, whose cohomology is denoted by $H_{\mathcal{F}}^{0,q}(M)$. A. El Kacimi Alaoui posed the following open question:
\begin{question}[=Question 2.10.4 in \cite{EKAa14}]
	Let $(M,\mathcal{F})$ be a complex foliation such that every leaf is a Stein manifold and closed in $M$. Is $H_\mathcal{F}^{0,q}(M)=0$ for $q\geq 1$?
\end{question}
Corollary \ref{vlmnc2} provides a partial affirmative answer to this question (i.e., the vanishing of the leafwise Dolbeault $L_{loc}^2$-cohomology groups) under the assumption that $M$ admits a smooth exhaustion function which is strictly plurisubharmonic along each leaf. For compact manifolds, a corresponding $C^k$-version of Corollary \ref{vlmnc2} can be established analogously by means of Theorem \ref{ct''}. In the case where the Levi flat CR structure is induced by a fiber bundle whose fibers are complex manifolds, a positive answer to this question is also obtained in Corollary 2.1 of \cite{CN23}.

\section{Logarithmic forms of elliptic structures}\label{logari.}
Let $\MV$ be an elliptic structure of rank $n$ over an $(m+n)$-dimensional manifold $M$. For a coordinate chart $(U;x_1,\cdots,x_{m+n})$ given by the complex Frobenius theorem (\cite{Nl57}), we set
\begin{align}\label{ncfthm}
	z=\left(x_\varrho+\sqrt{-1} x_{m+\varrho}\right)_{\varrho=1}^m,\ t=\left(x_{2m+\tau}\right)_{\tau=1}^{n-m},
\end{align}
then
\begin{align}\label{ncfthm1}
	\MV\ \text{is spanned, over}\ U,\ \text{by}\ \left\{\frac{\partial}{\partial\bar{z}_\varrho}, \frac{\partial}{\partial t_\tau}\right\}_{\begin{subarray}{}1\leq\varrho\leq m\\1\leq\tau\leq n-m \end{subarray}}.
\end{align}
For simplicity, a \emph{$\MV$-coordinate chart} centered at $\lmfp\in M$ will refer to such a chart $(U;z,t)$ with
\begin{align*}
	U\cong \Delta^m\times (-1,1)^{n-m}\ (\text{under the coordinate map}),\ z(\lmfp)=0,t(\lmfp)=0,
\end{align*}
where $\Delta\subseteq\mathbb{C}$ is the open unit disc. Obviously, we have:
\begin{itemize}
	\item Any $\MV$-coordinate chart $(\hat{z},\hat{t})$ over $U$ satisfies
	\begin{align*}
		\hat{z}=\phi(z),\ \hat{t}=\psi(z,\bar{z},t),
	\end{align*}
	where $\phi,\psi\in C^\infty(U)$ and $\phi(z)$ is holomorphic in $z$.
	\item For $p\geq0$, $\Lambda^p N^*\MV$ is a basic vector bundle (see (iii) in Example \ref{exa}) with a basic local frame (over $U$)
	\begin{align*}
		\left\{\rd z_{_{P}}\right\}_{|P|=p},
	\end{align*}
	here and hereafter we adopt the following multi-index convention:
	\begin{align*}
		P=(\varrho_{_1},\cdots,\varrho_{_a}),\ 1\leq\varrho_{_1}<\cdots<\varrho_{_a}\leq m,1\leq a\leq m,\ \rd z_{_P}=\rd z_{\varrho_{_1}}\wedge\cdots\wedge\rd z_{\varrho_{_a}}.
	\end{align*}
\end{itemize}
Let
\begin{align*}
	\Omega_\MV^p=\mathcal{O}_\MV(\Lambda^pN^*\MV),
\end{align*}
then for any $f\in C^\infty(U,\Lambda^pN^*\MV),$ $f\in\Gamma(U,\Omega_\MV^p)$ if and only if
\begin{align}\label{basicpform}
	f=\sum_{|P|=p}f_{_{P}}(z)\rd z_{_{P}},
\end{align}
with coefficients $f_{_{P}}(z)\in C^\infty(U)$ being holomorphic in $z$. Sections of $\Omega_\MV^p$ are called \emph{basic $p$-forms}. According to \eqref{basicpform}, the exterior derivative ${\rm d}$ maps basic $p$-forms to basic $(p+1)$-forms.
\begin{lemma}\label{intformula}
	Let $(U;z,t)$ be a $\MV$-coordinate chart centered at some $\lmfp\in U$, assumed to be starlike with respect to $\lmfp$.
	\begin{enumerate}
		\item[$(i)$] For $f\in\Gamma(U,\Omega_\MV^p)$ given by \eqref{basicpform}, if $\rd f=0$ and $p\geq 1$ then $f=\rd g$   for some $g\in\Gamma(U,\Omega_\MV^{p-1})$. In particular, the constant sheaf $\mathbb{C}$ has a resolution by sheaves of  basic forms $0\longrightarrow\mathbb{C}\hookrightarrow\mathcal{O}_\MV\overset{\rd}{\longrightarrow}\Omega_\MV^1\overset{\rd}{\longrightarrow}\Omega_\MV^2\overset{\rd}{\longrightarrow}\cdots.$
		\item[$(ii)$] Given $ 1\leq \varrho\leq m$ and a basic function $F\in C^\infty(U)$, if
		\begin{align*}
			z_1\cdots z_\varrho=0\ \Rightarrow\ F=0,
		\end{align*}
		then there is a unique basic function $G$ such that $F=z_1\cdots z_\varrho G$.
	\end{enumerate}
\end{lemma}

\begin{proof}
	(i) This is clear from the homotopy formula for the Poincaré lemma: $f=\rd u$ where
	\begin{align*}
		u:=\xi\lrcorner\sum_{|P|=p}\left(\int_0^1\theta^{p-1}f_{_{P}}(\theta z)\rd \theta\right)\rd z_{_{P}}
	\end{align*}
	is a basic form by \eqref{basicpform}, and $\xi$ denotes the basic vector field $\sum\limits_{\varrho=1}^mz_\varrho\frac{\partial}{\partial z_\varrho}$ over $U$.
	
	(ii) By \eqref{ncfthm1}, $F=F(z)$ is holomorphic in $z$ and therefore
	\begin{align*}
		F(z)=\int_{[0,1]^{\varrho}}\frac{\rd}{\rd u_1}\cdots\frac{\rd}{\rd u_\varrho}F(u_1z_1,\cdots,u_\varrho z_\varrho,z')\rd u= z_1\cdots z_\varrho G, 
	\end{align*}
	where $z':=(z_{\varrho+1},\cdots,z_m)$ and
	\begin{align*}
		G=\int_{[0,1]^{\varrho}}\bigg(\frac{\partial}{\partial z_1}\cdots\frac{\partial}{\partial z_\varrho}F\bigg)(u_1z_1,\cdots,u_\varrho z_\varrho,z')\rd u
	\end{align*}
	is immediately basic in view of \eqref{ncfthm1}.
\end{proof}

Corollary \ref{local existence1}, combined with (i) in the above lemma, yields
\begin{cor}[=Corollary \ref{sc} (i)]
	If $M$ admits a smooth $1$-convex exhaustion function with respect to an elliptic structure $\MV$, then for $p\geq 0,$
	\begin{align*}
		H^p(M,\mathbb{C}) =\frac{\mathrm{Ker}\left(\Gamma(M,\Omega_\MV^p)\overset{\rd}{\longrightarrow}\Gamma(M,\Omega_\MV^{p+1})\right)}{\mathrm{Im}\left(\Gamma(M,\Omega_\MV^{p-1})\overset{\rd}{\longrightarrow}\Gamma(M,\Omega_\MV^{p})\right)},
	\end{align*}
	where $\Omega_\MV^{-1}:=0.$ In particular, $H^p(M,\mathbb{C}) =0$ for $p\geq m+1$.
\end{cor}

\begin{definition}\label{basichyper}
	 Let $\MV$ be an elliptic structure and let $D\subseteq M$ be a closed nowhere dense subset. If for each  $\lmfp\in D$, there is a smooth basic function $F$ defined on an open neighborhood $U$ of $\lmfp$ such that $D\cap U=F^{-1}(0)$, then $D$ is called a basic hypersurface of $M$ $($with respect to $\MV)$. 
\end{definition}


Since the local rings of complex analytic functions are unique factorization domains, there exists a minimal defining function for any basic hypersurface of $M$. By definition, the minimal defining function is unique up to multiplication by a nowhere vanishing basic function, and this observation enables us to introduce the notion of logarithmic forms. Let $D\subseteq M$ be a basic hypersurface, for $p\in\mathbb{Z}_{\geq 0}$, we define the presheaf $\Omega^p_\MV(\log D)$ of logarithmic $p$-forms as follows:
\begin{itemize}
	\item For $\lmfp\in M\setminus D$, set 
	\begin{align*}
		\big(\Omega^p_\MV(\log D)\big)(U):=\Gamma(U,\Omega^p_\MV),
	\end{align*}
	where $U\subseteq M\setminus D$ is any open neighborhood of $\lmfp$.
	\item For $\lmfp\in D$, set
	\begin{align*}
		\big(\Omega^p_\MV(\log D)\big)(U):=&\big\{f\in\Gamma(U\setminus D,\Omega^p_\MV)\ |\ Ff\in\Gamma(U,\Omega^{p}_\MV), F\rd f\in\Gamma(U,\Omega^{p+1}_\MV)\big\}\\ =&\big\{f\in\Gamma(U\setminus D,\Omega^p_\MV)\ |\ Ff\in\Gamma(U,\Omega^p_\MV), \rd F\wedge f\in\Gamma(U,\Omega^{p+1}_\MV)\big\},
	\end{align*}
	where $U$ is any open neighborhood of $\lmfp$ with $F\in C^\infty(U)$ being a minimal defining function for $D$ over $U$.
\end{itemize}
Clearly, the above data defines a complete presheaf over $M$. Denote by the same symbol $\Omega^p_\MV(\log D)$ the sheaf generated by it, and refer to the sections of this sheaf as \emph{logarithmic $p$-forms along $D$}.

\begin{definition}
	A basic hypersurface $D$ is said to be normal crossing if for each $\lmfp\in D$ there exists a $\MV$-coordinate chart $(U;z,t)$ centered at $\lmfp$ such that \begin{align}\label{coordnc}
		D\cap U=\big\{\lmfp'\in U\ |\ (z_1\cdots z_{a})(\lmfp')=0\big\},
	\end{align}
	where $0\leq a\leq m$ depending on $\lmfp\in D$. 
\end{definition}
If $D$ is a normal crossing basic hypersurface, by (ii) in Lemma \ref{intformula},
\begin{align*}
	F:=z_1\cdots z_a
\end{align*}
is a minimal defining function for $D$ over $U$, and therefore
\begin{align}\label{logpform4}
	\Omega^0_\MV(\log D)=\mathcal{O}_\MV.
\end{align}
Now assume $p\geq 1$, for every $f\in \Gamma(U,\Omega_\MV^p(\log D))$, we have $g:=Ff\in\Gamma(U,\Omega^{p}_\MV)$ and
\begin{align}\label{logpform2}
	\bigg(\sum_{\varrho=1}^az_{1}\cdots\widehat{z_{{\varrho}}}\cdots z_{a}\rd z_{{\varrho}}\bigg)\wedge g&=F\frac{\rd F}{F}\wedge{g}=F\rd F\wedge{f}\in F\cdot\Gamma(U,\Omega^{p+1}_\MV),
\end{align}
where $\hat{\cdot}$ means the term is omitted. For multi-indices $P$ and $I$, set
\begin{align*}
	P'=(1,\cdots,a)\cap P,\ P''=(1,\cdots,a)\setminus P.
\end{align*}
Given multi-indices $P$ with $|P|=p+1$, the coefficient of $\rd z_{_{P}}$ on the left hand side of \eqref{logpform2} is given by
\begin{align*}
	\sum_{\varrho\in P'}\mathrm{sgn}(\varrho P'\setminus\{\varrho\})z_{1}\cdots\widehat{z_{{\varrho}}}\cdots z_{a}g_{_{P\setminus\{\varrho\}}},
\end{align*}
where $\mathrm{sgn}(\cdot)$ denotes the sign of a permutation, which, together with (ii) in Lemma \ref{intformula}, yields that $g_{_{P\setminus\{\varrho\}}}$ is divisible in $\Gamma(U,\mathcal{O}_\MV)$ by $z_{{\varrho}}$ for each $\varrho\in P'$. 
In other words, the above discussion demonstrates that the coefficients of $g=Ff\in \Gamma(U,\Omega_\MV^p)$ possess the following properties:
\begin{align*}
	g_{_{P}}\in\bigg(\mathsmaller{\prod}\limits_{\varrho\in P''}z_{{\varrho}}\bigg)\cdot\Gamma(U,\mathcal{O}_\MV),
\end{align*}
i.e.,
\begin{align*}
	f_{_{P}}=F^{-1}g_{_{P}}\in\bigg(\mathsmaller{\prod}\limits_{\varrho\in P'}z_{{\varrho}}\bigg)^{-1}\cdot\Gamma(U,\mathcal{O}_\MV),
\end{align*}
for all multi-indices $P$ with $|P|=p$. In summary, $\Omega_\MV^p(\log D)|_{_U}$
is generated, over $\mathcal{O}_\MV|_{_U}$, by the $p$-fold wedge products of the following $1$-forms:
\begin{align}
	\left\{\frac{\rd z_{1}}{z_{1}},\cdots,\frac{\rd z_{a}}{z_{a}},\rd z_{{a+1}},\cdots,\rd z_{m}\right\}.\label{logpform3}
\end{align}
Together with \eqref{logpform4}, the sheaf $\Omega^p_\MV(\log D)$ is locally free over $\mathcal{O}_\MV$ $(p\geq 0)$ and thus could be viewed as a basic vector bundle over $M$ whose local basic frame is provided by the $p$-fold wedge products of $1$-forms in \eqref{logpform3}. Consequently, Corollary \ref{local existence1} implies
\begin{prop}
	If $D$ is a normal crossing basic hypersurface, and $M$ admits a smooth $1$-convex exhaustion function with respect to the elliptic structure $\MV$, then
	\begin{align*}
		H^q\big(M,\Omega^p_\MV(\log D)\big)=0,\ \forall p\geq 0,\ \forall q\geq 1.
	\end{align*}
	The hypercohomology of the complex
	\begin{align*}
		0\longrightarrow \mathcal{O}_\MV\overset{\rd}{\longrightarrow} \Omega^1_\MV(\log D)\overset{\rd}{\longrightarrow} \Omega^2_\MV(\log D)\overset{\rd}{\longrightarrow}\cdots
	\end{align*}
	is thereby determined by
	\begin{align}\label{logpform6}
		\mathbb{H}^p\big(M,\Omega_\MV^*(\log D)\big) =\frac{\mathrm{Ker}\left(\Gamma\big(M,\Omega_\MV^p(\log D)\big)\overset{\rd}{\longrightarrow}\Gamma\big(M,\Omega_\MV^{p+1}(\log D)\big)\right)}{\mathrm{Im}\left(\Gamma\big(M,\Omega_\MV^{p-1}(\log D)\big)\overset{\rd}{\longrightarrow}\Gamma\big(M,\Omega_\MV^{p}(\log D)\big)\right)},\ \forall p\geq 0.
	\end{align}
\end{prop}

Analogous to the classical case of complex structures, we will proceed to prove that the inclusion homomorphism defines a quasi-isomorphism
\begin{align*}
	\Omega_\MV^*(\log D)\rightarrow i_*\mathcal{A}^*_{M\setminus D},
\end{align*}
where $D\subseteq M$ is a normal crossing basic hypersurface, $i:M\setminus D\hookrightarrow M$ is the inclusion map and $(\mathcal{A}^*_{M\setminus D},\rd)$ denotes the de Rham complex of $M\setminus D$. This amounts to showing that the inclusion homomorphism induces an isomorphism on cohomology sheaves between $\mathcal{H}^*\big(\Omega_\MV^*(\log D)\big)$ and $\mathcal{H}^*\big(i_*\mathcal{A}^*_{M\setminus D}\big)$. Note that for any $\lmfp\in M\setminus D$,
\begin{align}
	\mathcal{H}^p\big(\Omega_\MV^*(\log D)\big)_{\mfp}&=
	\begin{cases}
		\mathbb{C},&\ p=0,\\
		0,&\ p\geq 1\ (\text{by Corollaries}\ \ref{local existence}\ \text{and}\ \ref{ros}),
	\end{cases}\label{quasi1}\\ \mathcal{H}^{p}\big(i_*\mathcal{A}^*_{M\setminus D}\big)_{\mfp}&=
	\begin{cases}
		\mathbb{C},&\ p=0,\\
		0,&\ p\geq 1\ (\text{by the Poincar\'e Lemma}),
	\end{cases} \label{quasi2}
\end{align}
we only need to consider $\lmfp\in D$.

Let $(U;z,t)$ be a $\MV$-coordinate chart centered at some $\lmfp\in D$ satisfying \eqref{coordnc}, then the coordinate map gives
\begin{align}\label{coordnc1}
	U\setminus D\cong \Delta_*^a\times\Delta^{m-a}\times 
	(-1,1)^{n-m},
\end{align}
where $\Delta_*:=\Delta\setminus\{0\}$. From the K\"unneth formula and the de Rham theorem, it is easy to see from \eqref{coordnc1} that
\begin{align}\label{qusi1}
	\mathcal{H}^p\big(i_*\mathcal{A}^*_{M\setminus D}\big)_{\mfp}&=\varinjlim_{U\ni \mfp}\frac{\mathrm{Ker}\left(\Gamma\big(U,i_*\mathcal{A}^p_{M\setminus D}\big)\overset{\rd}{\longrightarrow}\Gamma\big(U,i_*\mathcal{A}^{p+1}_{M\setminus D}\big)\right)}{\mathrm{Im}\left(\Gamma\big(U,i_*\mathcal{A}^{p-1}_{M\setminus D}\big)\overset{\rd}{\longrightarrow}\Gamma\big(U,i_*\mathcal{A}^p_{M\setminus D}\big)\right)}\notag\\
	&=\varinjlim_{U\ni \mfp}H^p(U\setminus D,\mathbb{C}) \notag\\ &=\bigwedge^p\mathrm{span}_\mathbb{C}\left\{\left[\frac{\rd z_\varrho}{z_\varrho}\right]\right\}_{\varrho=1}^a.
\end{align}
By \eqref{logpform4}, we have
\begin{align}
	\mathcal{H}^0\left(\Omega_\MV^*(\log D)\right)_{\mfp}=\mathbb{C},\label{quai2}
\end{align}
it remains to consider $p\geq 1$. As
\begin{align*}
	\mathcal{H}^p\big(\Omega_\MV^*(\log D)\big)_{\mfp}=\varinjlim_{U\ni \mfp}\frac{\mathrm{Ker}\left(\Gamma\big(U,\Omega_\MV^{p}(\log D)\big)\overset{\rd}{\longrightarrow}\Gamma\big(U,\Omega_\MV^{p+1}(\log D)\right)}{\mathrm{Im}\left(\Gamma\big(U,\Omega_\MV^{p-1}(\log D)\big)\overset{\rd}{\longrightarrow}\Gamma\big(U,\Omega_\MV^{p}(\log D)\big)\right)},
\end{align*}
we first prove the following lemma.

\begin{lemma}\label{quai3}
	Let $D\subseteq M$ be a normal crossing basic hypersurface. For every $\lmfp\in D$, choose a $\mathcal V$-coordinate chart $(U;z,t)$ centered at $\lmfp$ satisfying $(\ref{coordnc})$. Then for every
	\begin{align*}
		f\in\mathrm{Ker}\left(\Gamma\big(U,\Omega_\MV^{p}(\log D)\big)\overset{\rd}{\longrightarrow}\Gamma\big(U,\Omega_\MV^{p+1}(\log D)\big)\right),
	\end{align*}
	there exists some $f_{_{\text{const}}}\in\bigwedge^p\mathrm{span}_\mathbb{C}\left\{\frac{\rd z_\varrho}{z_\varrho}\right\}_{\varrho=1}^a$ such that $f\equiv f_{_{\text{const}}}\mod\ \rd\Gamma\big(U,\Omega_\MV^{p-1}(\log D)\big)$.
\end{lemma}

\begin{proof}
	We prove this lemma by induction on $0\leq a\leq m$, where $a=0$ means $\lmfp\in M\setminus D$ and in this case the conclusion follows from Corollaries \ref{local existence} and \ref{ros} (the same argument for \eqref{quasi1}). Assume $a\geq 1$ and denote $D'= \big\{\lmfp'\in U\ |\ (z_1\cdots z_{a-1})(\lmfp')=0\big\}$. By \eqref{logpform3} and (ii) in Lemma \ref{intformula}, $f$ can be written as $f=\frac{\rd z_a}{z_a}\wedge f'+f''$ where $f'\in\Gamma\big(U,\Omega_\MV^{p-1}(\log D')\big),\ f''\in \Gamma\big(U,\Omega_\MV^{p}(\log D')\big)$ and the coefficients of $f'$ are independent of the coordinate $z_a$ (replace $f'$ in the first term by evaluating its coefficients at $z_a=0$ and incorporate the resulting difference into $f''$). Now, $\rd f=0$ implies $\rd f'=0$ and thereby $\rd f''=0$. By the inductive hypothesis, we have
	\begin{align*}
		f'\equiv f_{_{\text{const}}}'&\mod\  \rd\Gamma\big(U,\Omega_\MV^{p-2}(\log D')\big),\\
		f''\equiv f_{_{\text{const}}}''&\mod\  \rd\Gamma\big(U,\Omega_\MV^{p-1}(\log D')\big),
	\end{align*}
	for some $f_{_{\text{const}}}'\in\bigwedge^{p-1}\mathrm{span}_\mathbb{C}\left\{\frac{\rd z_\varrho}{z_\varrho}\right\}_{\varrho=1}^{a-1}$ and $f_{_{\text{const}}}''\in\bigwedge^{p}\mathrm{span}_\mathbb{C}\left\{\frac{\rd z_\varrho}{z_\varrho}\right\}_{\varrho=1}^{a-1}$. Consequently,
	\begin{align*}
		f-f_{_{\mathrm{const}}}''-\frac{\rd z_a}{z_a}\wedge f_{_{\mathrm{const}}}'
		=\rd\left(\eta''-\frac{\rd z_a}{z_a}\wedge\eta'\right),
	\end{align*}
	which completes the induction.
\end{proof}

Combining \eqref{quasi1}, \eqref{quasi2}, \eqref{qusi1}, \eqref{quai2} and Lemma \ref{quai3}, we have proved that $\Omega_\MV^*(\log D)$ is quasi-isomorphic to $i_*\mathcal{A}^*_{M\setminus D}$ if $D$ is normal crossing, which, in conjunction with \eqref{logpform6}, implies the following corollary.

\begin{cor}[=Corollary \ref{sc} $(ii)$]
	Assume that $M$ admits a smooth $1$-convex exhaustion function with respect to an elliptic structure $\MV$, and that $D\subseteq M$ is a normal crossing basic hypersurface. Then for $p\geq 0,$
	\begin{align*}
		H^p(M\setminus D,\mathbb{C}) =\frac{\mathrm{Ker}\left(\Gamma\big(M,\Omega_\MV^p(\log D)\big)\overset{\rd}{\longrightarrow}\Gamma\big(M,\Omega_\MV^{p+1}(\log D)\big)\right)}{\mathrm{Im}\left(\Gamma\big(M,\Omega_\MV^{p-1}(\log D)\big)\overset{\rd}{\longrightarrow}\Gamma\big(M,\Omega_\MV^{p}(\log D)\big)\right)},
	\end{align*}
	where $\Omega_\MV^{-1}(\log D):=0.$ In particular, $H^p(M\setminus D,\mathbb{C}) =0$ for $p\geq m+1$.
\end{cor}


To conclude this section, we consider an extension problem for the canonical bundle of the elliptic structure $\MV$. We begin with the following property for basic functions with respect to $\MV$.
\begin{prop}\label{pobf}
	Let $V\subseteq M$ be an open set. Suppose that $f\in C^\infty(V)$ is a basic function with respect to $\MV$ and $\rd_\mfp f\neq0$ at some point $\lmfp\in V$. Then there exists a coordinate chart $(U;z,t)$ centered at $\lmfp$ within $V$, of the form $(\ref{ncfthm})$ and satisfying $(\ref{ncfthm1})$, such that $f-f(\lmfp)$ coincides with one of coordinate functions $z_\varrho$ $(1\leq\varrho\leq m)$.
\end{prop}
\begin{proof}
	According to the complex Frobenius theorem, there exists a coordinate chart centered at $\lmfp$ in $V$, given by (\ref{ncfthm})
	\begin{align*}
		(U; z_1,\cdots,z_{m},t_1,\cdots,t_{n-m}),
	\end{align*}
	and satisfying (\ref{ncfthm1}). Since $f$ is basic, it is annihilated by $\MV$, i.e., $f$ is holomorphic in $z$ and independent of $t$. We compute $\rd_\mfp f=\sum_{\varrho=1}^{m}\partial_{z_\varrho}f(\lmfp)\rd z_\varrho\neq0$, then without loss of generality, we may assume $\partial_{z_{1}}f(\lmfp)\neq0$.
	
	If $\partial_{z_{1}}f(\lmfp)\neq0$, in view of the implicit function theorem, shrinking $U$ if necessary, there exists a smooth function $\hbar_1=\hbar_1(z',\bar{z}')$ where $z'=(z_2,\cdots,z_m)$ such that
	\begin{align*}
		f(\hbar_1(z',\bar{z}'),z')-f(\lmfp)\equiv0\ \text{and}\ \partial_{z_{1}}f\neq0\ \text{on}\ U.
	\end{align*}
	Applying $\partial_{\bar{z}_{\varrho}}$ ($2\leq\varrho\leq m$) to the above identity we have
	\begin{align*}
		0\equiv\partial_{\bar{z}_{\varrho}}f(\hbar_1(z',\bar{z}'),z')=\partial_{\bar{z}_{\varrho}}\hbar_1(z',\bar{z}')\partial_{{z}_{1}}f(\hbar_1,z'),
	\end{align*}
	which leads to $\hbar_1(z',\bar{z}')=\hbar_1(z')$. Define $\hat{z}_1=z_1-\hbar_1(z')$, it follows that
	\begin{align*}
		(U;\hat{z}_1,z',t)
	\end{align*}
	is a coordinate chart satisfying (\ref{ncfthm1}). Set $\hat{f}(\hat{z}_1,z'):=f(\hat{z}_1+\hbar_1,z')=f(z)$, we know that $\hat{f}(\hat{z}_1,z')-f(\lmfp)=0$ if $\hat{z}_1=0$. Thus,
	\begin{align*}
		f-f(\lmfp)=\hat{f}(\hat{z}_1,z')-f(\lmfp)=\hat{z}_1\hat{f}_1(\hat{z}_1,z'),
	\end{align*}
	where $\hat{f}_1(\hat{z}_1,z')\neq0$ on $U$. This implies that $f-f(\lmfp)$ and $\hat{z}_1$ differ by a basic invertible factor. Hence, replacing $\hat{z}_1$ with $f-f(\lmfp)$ yields a valid local coordinate chart, which proves the assertion.
\end{proof}

In what follows, let $D$ be a basic hypersurface of $M$ locally defined by
\begin{align}\label{sh.}
	D\cap U_\alpha=F_\alpha^{-1}(0)\ \text{and}\ \rd F_\alpha\neq0\ \text{on}\ D\cap U_\alpha,
\end{align}
where $\{U_\alpha\}_\alpha$ is an open covering of $M$ and $F_\alpha\in  C^\infty(U_\alpha)$ are basic functions. 
We call such $D$ a \emph{smooth basic hypersurface $($with respect to $\MV$$)$}. Moreover, for all $\alpha,\beta$
\begin{align*}
	\phi_{\alpha\beta}:=\frac{F_\alpha}{F_\beta}\in C^\infty(U_\alpha\cap  U_\beta)\ \text{are nowhere zero basic functions.}
\end{align*}
It follows that
\begin{align*}
	\phi_{\alpha\beta}\cdot\phi_{\beta\gamma}\cdot\phi_{\gamma\alpha}=1\ \text{on}\ U_\alpha\cap U_\beta\cap U_\gamma,
\end{align*}
the $1$-cocycle $\{\phi_{\alpha\beta}\}_{\alpha,\beta}$ thus defines a basic line bundle $[D]$.

Each smooth basic hypersurface $D$ is compatible with $\mathcal{V}$ in a natural way. More precisely, we have the following proposition.
\begin{prop}
	Any smooth basic hypersurface $($with respect to the elliptic structure $\MV)$
	\begin{align}\label{i.}
		\imath:D\hookrightarrow M
	\end{align}
	induces an elliptic structure $\MV_D:=\MV\cap \mathbb{C}TD$ over $D$ of corank $(m-1)$.
\end{prop}
\begin{proof}
	We first prove that $D$ is a submanifold of $M$ of real codimension $2$ by contradiction. Let $\{(U_\alpha,F_\alpha)\}_\alpha$ be a family of basic defining functions of $D$ satisfying (\ref{sh.}). We may write
	\begin{align*}
		{\rm d}F_\alpha={\rm d}u_\alpha+\sqrt{-1}{\rm d}v_\alpha,
	\end{align*}
	where $u_\alpha,v_\alpha$ are real-valued functions. For each $\alpha$, suppose that there is a point $\lmfp\in D\cap U_\alpha$ such that
	$$\rd_\mfp F_\alpha\wedge
	\rd_\mfp\overline{F_\alpha}
	=-2\sqrt{-1}{\rm d}_\mfp u_\alpha\wedge{\rm d}_\mfp v_\alpha=0.$$
	Since $F_\alpha$ is basic and $\rd_\mfp F_\alpha\neq0$, the nonzero forms $\rd_\mfp F_\alpha\in N_\mfp^*\MV$ and $\rd_\mfp\overline{F_\alpha}
	\in\overline{N_\mfp^*\MV}$ are linearly dependent, which contradicts the ellipticity of $\MV$.
	
	We deduce from Proposition \ref{pobf} that for any point $\lmfp\in D\cap U_\alpha$, there exists a coordinate chart
	\begin{align*}
		(z_1,\cdots,z_{m},t_1,\cdots,t_{n-m})
	\end{align*}
	 in $U_\alpha$ centered at $\lmfp$ satisfying $(\ref{ncfthm1})$ such that, without loss of generality, $F_\alpha$ is the coordinate function $z_1$. Then the image of the natural map $\imath^*N^*\MV\rightarrow\mathbb CT^*D$ is locally spanned by
	\begin{align*}
		\{\imath^*\rd z_2,\cdots,\imath^*\rd z_{m}\}.
	\end{align*}
	Hence, $\MV$ induces an elliptic structure $\MV_D:=\MV\cap \mathbb{C}TD$ over $D$ of corank $m-1$.
\end{proof}
\begin{remark}
	If $M$ has a $1$-convex exhaustion function $\varphi$ with respect to $\MV$, then the restriction $\varphi|_{_D}$ is a $1$-convex exhaustion function with respect to $\MV_D$ on $D$.
\end{remark}

Thanks to Proposition \ref{pobf} and (\ref{logpform3}), the sheaf $\Omega_\MV^p(\log D)$ is generated by
\begin{align*}
	\big(\Omega_\MV^p(\log D)\big)(U):=\left\{\frac{\rd F_{{_U}}}{F_{_U}}\wedge f_{_U}+g_{_U}\ |\ f_{_U}\in\Gamma(U,\Omega_\MV^{p-1}),\ g_{_U}\in\Gamma(U,\Omega_\MV^{p})\right\},
\end{align*}
where $U\subseteq M$ is any open subset such that $F_{_U}\in C^\infty(U)$ is the local basic defining function of $D$ such that $\rd F_{_U}\neq0$ on $D\cap U$.

To formulate our extension result, we define a sheaf morphism for any $p\geq0$ as follows
\begin{align}\label{sfm}
	R_p:\Omega_\MV^p(\log D)&\rightarrow \imath_*\Omega_{\MV_D}^{p-1}\\
	\frac{\rd F_{_U}}{F_{_U}}\wedge f_{_U}+g_{_U}&\mapsto f_{_U}|_{_D},\nonumber
\end{align}
where $\Omega_{\MV_D}^{p-1}:=\mathcal{O}_{\MV_D}(\Lambda^{p-1}N^*\MV_D)$ for $p\geq1$ and $\Omega_{\MV_D}^{-1}:=0$, $\imath$ is given by (\ref{i.}). Indeed, suppose $D\cap U=\{F'_{_U}=0\}$ such that $\rd F'_{_U}\neq0$, then there is a non-vanishing basic function $h_{_U}$ on $U$ satisfying $F'_{_U}=h_{_U}F_{_U}$. If
\begin{align*}
	\frac{\rd F_{_U}}{F_{_U}}\wedge f_{_U}+g_{_U}=\frac{\rd F'_{_U}}{F'_{_U}}\wedge f_{_U}'+g_{_U}'=\left(\frac{\rd h_{_U}}{h_{_U}}+\frac{\rd F_{_U}}{F_{_U}}\right)\wedge f_{_U}'+g_{_U}',
\end{align*}
then
\begin{align*}
	\frac{\rd F_{_U}}{F_{_U}}\wedge(f_{_U}-f_{_U}')=\frac{\rd h_{_U}}{h_{_U}}\wedge f_{_U}'+g_{_U}'-g_{_U},
\end{align*}
which implies that
\begin{align*}
	\left.\rd F_{_U}\wedge(f_{_U}-f_{_U}')\right|_{D\cap U}=0.
\end{align*}
Since $\rd F_{_U}\neq0$ on $D\cap U$ and $\rd F_{_U}$ vanishes on $TD$, it follows that $\imath^*f_{_U}=\imath^*f_{_U}'$. Thus $R_p$ is well-defined.

\begin{prop}[=Corollary \ref{sc} $(iii)$]
	Let $\MV$ be an elliptic structure over a manifold $M$, and let $D$ be a smooth basic hypersurface with the induced elliptic structure $\MV_D$. If $M$ admits a smooth $1$-convex exhaustion function with respect to $\MV$, then for any $0\leq p\leq m$, 
	\begin{align*}
		R_p:\Gamma(M,\Omega_\MV^p(\log D))\rightarrow \Gamma(D,\Omega_{\MV_D}^{p-1})\ \text{is surjective}.
	\end{align*}
	In particular, for $p=m$, there exists a homomorphism ${R}$ induced by $(\ref{sfm})$ such that
	\begin{align*}
		{R}:\Gamma(M,\Omega_\MV^m\otimes[D])\rightarrow \Gamma(D,\Omega_{\MV_D}^{m-1})\ \text{is surjective}.
	\end{align*}
\end{prop}

\begin{proof}
	It is straightforward to check that $R_p$ (see (\ref{sfm})) is surjective and ${\rm Ker}(R_p)=\Omega_\MV^p$. Consequently, we obtain an exact sequence on $M$
	\begin{align*}
		0\longrightarrow\Omega_\MV^p\longrightarrow\Omega_\MV^p(\log D)\stackrel{R_p}{\longrightarrow}\imath_*\Omega_{\MV_D}^{p-1}\longrightarrow0,
	\end{align*}
	which induces the long exact sequence
	\begin{align*}
		\cdots\longrightarrow \Gamma(M,\Omega_\MV^p(\log D))\stackrel{R_{p}}{\longrightarrow} \Gamma(D,\Omega_{\MV_D}^{p-1})\longrightarrow H^{1}(M,\Omega_\MV^p)\longrightarrow\cdots.
	\end{align*}
	In view of Corollary \ref{local existence1}, we have $H^{1}(M,\Omega_\MV^p)=0$ for any $0\leq p\leq m$ which implies that $R_p$ is surjective.
	
	For $p=m$, let $s_{_D}:=\{F_{\alpha}\}_{\alpha}$ be the family of local basic defining functions of $D$ satisfying (\ref{sh.}), we may define a morphism by locally multiplying with $F_{\alpha}$
	\begin{align*}
		{s}_{_D}\otimes:\Gamma(M,\Omega_\MV^m(\log D))&\rightarrow\Gamma(M,\Omega_\MV^m\otimes[D])\\
		\frac{\rd F_{\alpha}}{F_{\alpha}}\wedge f_{\alpha}+g_{\alpha}&\mapsto(\rd F_{\alpha}\wedge f_{\alpha}+F_{\alpha}g_{\alpha})\otimes\sigma_\alpha,
	\end{align*}
	where $\sigma_\alpha$ is a local basic frame of $[D]$. In fact, for
	\begin{align*}
		\left\{\frac{\rd F_{\alpha}}{F_{\alpha}}\wedge f_{\alpha}+g_{\alpha}\right\}_{\alpha}\in\Gamma(M,\Omega_\MV^m(\log D)),
	\end{align*}
	we have
	\begin{align*}
		\rd F_{\alpha}\wedge f_{\alpha}+F_{\alpha}g_{\alpha}=\frac{ F_{\alpha}}{F_{\beta}}(\rd F_{\beta}\wedge f_{\beta}+F_{\beta}g_{\beta})\ \text{on}\ U_\alpha\cap U_\beta,
	\end{align*}
	which implies the well-definedness of ${s}_{_D}\otimes$. Since the zero set of $F_{\alpha}$ is nowhere dense, it follows immediately that the homomorphism ${s}_{_D}\otimes$ is injective. On the other hand, any element in $\Gamma(M,\Omega_\MV^m\otimes[D])$ divided by $F_{\alpha}$ yields an element in $\Gamma(M,\Omega_\MV^m(\log D))$, thus ${s}_{_D}\otimes$ is surjective. We therefore conclude that ${s}_{_D}\otimes$ is an isomorphism. As a result,
	\begin{align*}
		{R}:=R_m\circ({s}_{_D}\otimes)^{-1}
	\end{align*}
	is the desired homomorphism.
\end{proof}

Many results in the case of complex structures can be explored within the framework of elliptic structures. Beyond the aforementioned cases, we developed elliptic analogues of fundamental results related to complex structures in \cite{JYY22} and \cite{JY23}, including division and extension theorems for holomorphic functions, as well as Nadel's coherence theorem.

A natural question arises regarding the extension of our results in this section to hypocomplex structures, which is a generalization of elliptic structures (see \cite{T2}). A key challenge is that hypocomplex structures are not necessarily Levi flat, whereas the definition of the quadratic form (\ref{qf.}) relies on Levi flatness. We will address this issue in a future paper.

\subsection*{Acknowledgement}
The second author thanks Dr. Yuanpu Xiong for helpful discussions.

\end{document}